\documentclass[11pt]{article}
\usepackage[utf8]{inputenc}
\usepackage{mathtools}
\usepackage{amsmath}
\usepackage{amssymb}
\usepackage{graphicx}
\usepackage{changepage}
\graphicspath{ {images/} }
\usepackage{mathrsfs}
\usepackage{listings}
\usepackage{amsthm}
\usepackage[toc,page]{appendix}
\usepackage[margin=1.1in, headsep=10pt]{geometry}
\usepackage{framed,enumitem} 

\usepackage{sectsty}
\sectionfont{\fontsize{14}{12}\selectfont}

\usepackage[colorlinks = true,
linkcolor = blue,
urlcolor  = blue,
citecolor = blue,
anchorcolor = blue]{hyperref}

\renewenvironment{abstract}
{\small
	\begin{center}
		\bfseries \abstractname\vspace{-.5em}\vspace{0pt}
	\end{center}
	\list{}{%
		\setlength{\leftmargin}{15mm}% <---------- CHANGE HERE
		\setlength{\rightmargin}{\leftmargin}%
	}%
	\item\relax}
{\endlist}

\newtheorem{thm}{Theorem}[section]
\newtheorem{defn}[thm]{Definition}
\newtheorem{cor}[thm]{Corollary}

\newtheorem{prop}[thm]{Proposition}
\newtheorem{lem}[thm]{Lemma}
\newtheorem{rk}[thm]{Remark}
\newtheorem{ass}[thm]{Assumption}

\title{The KPZ Limit of ASEP with Boundary}
\author{Shalin Parekh}
\date{\today}

\begin{document}
	
	\maketitle
	
	\begin{abstract}It was recently proved in [CS16] that under weak asymmetry scaling, the height functions for open ASEP on the half-line and on a bounded interval converge to the Hopf-Cole solution of the KPZ equation with Neumann boundary conditions. In their assumptions [CS16] chose positive values for the Neumann boundary conditions, and they assumed initial data which is close to stationarity. By developing more extensive heat-kernel estimates, we extend their results to negative values of the Neumann boundary parameters, and we also show how to generalize their results to narrow-wedge initial data (which is very far from stationarity). As a corollary via [BBCW17], we obtain the Laplace transform of the one-point distribution for half-line KPZ, and use this to prove $t^{1/3}$-scale GOE Tracy-Widom long-time fluctuations.
	\end{abstract}
	
	\textbf{}
	\\
	\tableofcontents
	
	\newpage
	
	\section{Introduction}
	
	Interacting particle systems in one spatial dimension have been extensively studied in recent years. Of particular interest is the Asymmetric Simple Exclusion Process (ASEP), in which particles on $\Bbb Z$ jump independently to the left and right at exponential rates in a collectively Markovian way, subject to a drift in one chosen direction and with other particles acting as a deterrent to the general movement. Such processes fall within what has come to be known as the ``KPZ Universality Class," a broad collection of space-time processes coming from both mathematics and physics, which are unified by some salient features [Cor12]. The ``universality" of this class refers to the \textit{long-term} behavior of the model: how it looks on large scales in both time and space. Objects which fall into this class will generally have temporal fluctuations on the order of $t^{1/3}$ and the statistics will be described by probability distributions which come out of random matrix theory, such as the GUE and GOE eigenvalue distributions. Moreover, the space-time fluctuations of such objects are generally seen to converge under suitable \textit{weak} scaling to the solution of the KPZ \textit{equation} [BG97, ACQ11, CT15, DT16, CS16, CST16]: $$\partial_T H = \frac{1}{2}\partial_X^2H +\lambda (\partial_X H)^2 + \xi$$ where $\xi$ is a Gaussian space-time white noise and $\lambda>0$ is a constant.
	\\
	\\
	The subject of the current paper is how to deal with ASEP on intervals with boundary (namely, $\Bbb Z_{\geq 0}$ and $\{0,...,N\}$), and to prove a small step towards KPZ universality of such a model. The main features of such a model were physically analyzed as early as the seventies: see [CS16, Section 1] for more physical motivation as well as a discussion of the phase diagram for Open ASEP. [Lig75] is also a great resource for the foundational work in this area. Convergence to the KPZ equation (in certain cases) was first considered and proved in [CS16]. The key feature of Open ASEP is that the boundary conditions are governed by \textit{sources} and \textit{sinks}, where particles are created and annihilated at certain exponential rates. Under the weak scaling, the choice of rates corresponds directly to choosing Neumann boundary parameters in the KPZ equation.
	\\
	\\
	At this point, let us state (slightly informally) our main results: Fix parameters $\alpha,\gamma, p, q \geq 0$. We define half-space ASEP (or ASEP-H) to be the interacting particle system on positive integers $\Bbb Z_{> 0}$ where particles at site $x>1$ jump to site $x-1$ at exponential rate $q$ if site $x-1$ is unoccupied, and particles at site $x>0$ jump to site $x+1$ at rate $p$ if site $x+1$ is unoccupied. If a neighboring site is occupied, then a particle feels no inclination to jump there (i.e., rate $0$). All jumps are independent of each other. Furthermore, particles at site $x=1$ are created at rate $\alpha$ if site $x=1$ is unoccupied, and particles at site $x=1$ are destroyed at rate $\gamma$ if site $x=1$ is occupied. We may define a height function associated to this particle system as follows: let $h_t(0)$ denote twice the number of particles annihilated minus twice the number of particles created at site $x=0$ up to time $t$. Then define $h_t(x+1)$ to be $h_t(x)+1$ if there is a particle at site $x+1$, and define $h_t(x+1)$ to be $h_t(x)-1$ otherwise. By linear interpolation, we view the height function as a random element of $C([0,\infty),\Bbb R)$.
	\\
	\\
	The preceding paragraph gives a notion of ASEP on the half-line $\Bbb Z_{\geq 0}$, but we can also define a similar particle system on the bounded interval $\{0,...,N\}$. To do this, we fix two more parameters $\beta, \delta \geq 0$, and we create and annihilate particles at site $N$ with rates $\delta, \beta$, respectively. Our main result may now be stated as follows:
	\begin{thm}[Main Result]\label{mr} Fix some parameters $A,B \in \Bbb R$. For half-line ASEP, we let $I=[0,\infty)$ and for bounded-interval ASEP, we set $I=[0,1]$. For $\epsilon>0$, we define: $$p = \frac{1}{2} e^{\sqrt{\epsilon}}, \;\;\;\;\;\;\;\;\;\;\;\; q = \frac{1}{2}e^{-\sqrt \epsilon}, \;\;\;\;\;\;\;\;\;\;\;\; \mu_A = 1-A\epsilon, \;\;\;\;\;\;\;\;\;\;\;\; \mu_B = 1-B\epsilon.$$ Let us also define the creation/annihilation rates:
	$$\alpha= \frac{p^{3/2}(p^{1/2}-\mu_Aq^{1/2})}{p-q}\;,\;\;\;\;\;\;\;\; \beta = \frac{p^{3/2}(p^{1/2}-\mu_Bq^{1/2})}{p-q},$$ $$\gamma=\frac{q^{3/2}(q^{1/2}-\mu_Ap^{1/2})}{q-p}\;,\;\;\;\;\;\;\;\; \delta=\frac{q^{3/2}(q^{1/2}-\mu_Bp^{1/2})}{q-p}.$$ Let $h_t^{\epsilon}(x)$ denote the height function associated to this particle system. We then define $$H^{\epsilon}(T,X) := \epsilon^{1/2} h^{\epsilon}_{\epsilon^{-2}T}(\epsilon^{-1}X)- \big( \frac{1}{2}\epsilon^{-1}+\frac{1}{24} \big) T .$$ Assume that $H^{\epsilon}(0,X)$ converges weakly to some initial data $H_0$ which is near equilibrium (see Definition \ref{51}). Then $H^{\epsilon}(T,X)$ converges in distribution to the Hopf-Cole solution of the KPZ equation (cf. Definition \ref{69}) on the interval $I$ with Neumann boundary parameters $A$ $($at $X=0)$ and $B$ $($at $X=1$ if $I=[0,1])$, started from $H_0$. The convergence occurs in the Skorokhod Space $D([0,\infty),C(I))$.
	
	\end{thm}
	
	This result is proved as Theorem \ref{126} below. Before this, we will first introduce the Hopf-Cole solution of KPZ with Neumann boundary conditions, and prove some existence results about it (see Section 4). We get a similar result when the initial data is not stationary, but we must subtract a logarithmically-divergent height-shift:
	
	\begin{thm}[Extension to Narrow-Wedge Initial data]\label{128} Let $A,B$ and $I$ as before. For $\epsilon>0$, let $p,q,\alpha,\beta,\gamma,\delta$, and $h_t^{\epsilon}(x)$ be the same as in Theorem \ref{mr}. We define $$H^{\epsilon}(T,X) := \epsilon^{1/2} h^{\epsilon}_{\epsilon^{-2}T}(\epsilon^{-1}X)- \big( \frac{1}{2}\epsilon^{-1}+\frac{1}{24} \big) T - \frac{1}{2}\log \epsilon .$$ Assume that we start from the initial configuration of zero particles $($i.e., $H^{\epsilon}(0,X) = -\epsilon^{-1/2}X)$. Then we have that $H^{\epsilon}(T,X)$ converges in distribution to the Hopf-Cole solution of the KPZ equation on the interval $I$ with Neumann boundary parameters $A$ $($at $X=0)$ and $B$ $($at $X=1$ if $I=[0,1])$. The initial data is narrow-wedge (more precisely, the associated stochastic heat equation starts from $\delta_0$). The convergence occurs in the Skorokhod Space $D((0,\infty),C(I))$, see Definition \ref{127}.
	
	\end{thm}
	
	This will be proved as Theorem \ref{61} below. See Section 6 for the precise result.
	\\
	\\
	Let us discuss for a moment why our results are more general than those of [CS16]. There, the authors prove convergence of Open ASEP to the KPZ equation under the assumption of non-negative values for the Neumann boundary conditions. However, the non-negativity only seems to be a technical restriction which simply makes the analysis a little bit easier. In other words, the zero value for the Neumann boundary condition does not actually correspond to any meaningful phase transition in the associated particle system [CS16, Remark 2.11], hence one would expect that it is just a superficial restriction. Thus, the main purpose of the current work is to generalize their results to the case when the boundary parameters for the PDE are negative. This corresponds to branching in the associated diffusion process, hence the associated kernels will be super-probability measures in general. This brings about some new challenges which were not seen in [CS16], and we develop several new heat-kernel estimates in order to resolve these challenges. We also found a few small mistakes in the proofs of tightness in [CS16], [CST16] and related papers; therefore we will show how to fix these issues in the current work (see the discussions at the beginning of the proofs of Propositions \ref{85} and \ref{36}).
	\\
	\\
	One further restriction which the authors used in [CS16] was to assume near-equilibrium initial data, which roughly means that the sequence of initial conditions of the particle system are close to stationarity. We also remove that assumption in the current work, and (as stated in Theorem \ref{128} above) we extend their results to narrow-wedge initial data for the ASEP (which corresponds to $\delta_0$ initial data for the stochastic heat equation).
	\\
	\\
	It may seem a little bit uncertain what the immediate usage of such a technical generalization might be, however the first and foremost application is that it implies another result in KPZ half-space universality, namely [BBCW17, Theorem B and Remark 1.1]. More specifically, it gives us the exact one-point statistics (via the moment-generating function) for half-line KPZ. Due to importance we will combine our result with theirs and reiterate it here:
	
	\begin{cor}\label{bbcw} Let $H(T,X)$ denote the solution to the KPZ equation on $[0,\infty)$ with Neumann boundary parameter $A=-1/2$ (see definition 2.4 below; this choice of $A$ corresponds exactly to the triple point of half-line ASEP; see [CS16, Section 1 and Remark 2.11]). Then for $\xi>0$ we have that $$\Bbb E \bigg[ \exp \bigg( -\xi \exp \big(H(T,0)+\frac{T}{24} \big) \bigg) \bigg] = \Bbb E \bigg[ \prod_{k=1}^{\infty} \frac{1}{\sqrt{1+4\xi \exp[(T/2)^{1/3}\mathbf{a}_k]}} \bigg] $$ where $\mathbf{a}_1 > \mathbf{a}_2 > ...$ forms the GOE Point process (see [BBCW17, Definition 6.1]).
		
	\end{cor}
	\textbf{}\\
	As an easy corollary of this fact, we obtain a limit theorem for the half-space KPZ equation with Neumann$(-1/2)$ boundary condition on $[0,\infty)$. As expected we get a random-matrix-type of distribution in the limit.
	
	\begin{cor} Let $H(T,X)$ as in the previous theorem. Then one has weak convergence $$\Bbb P \bigg(\frac{H(T,0)+\frac{T}{24}}{(T/2)^{1/3}} \leq x \bigg) \stackrel{T \to \infty}{\longrightarrow} F_{GOE}(x)$$ where $F_{GOE}(x) = \Bbb P(\mathbf{a}_1 \leq x)$ is the Tracy-Widom GOE distribution.
		
	\end{cor}
	
	\begin{proof} Fix $x$, and let $$Q(T) := H(T,0)+\frac{T}{24}-(T/2)^{1/3}x$$ Replacing $\xi$ with  $\xi e^{-(T/2)^{1/3}x}$ in Corollary \ref{bbcw}, we find that $$\Bbb E \bigg[ \exp \bigg( -\xi \exp \big( Q(T) \big) \bigg) \bigg] = \Bbb E \bigg[ \prod_{k=1}^{\infty} \frac{1}{\sqrt{1+4\xi \exp\big[(T/2)^{1/3}(\mathbf{a}_k-x) \big]}} \bigg]$$ Letting $T \to \infty$ on both sides, applying the dominated convergence theorem, and noting that $\mathbf a_1 = \max_k \mathbf a_k$, we find that for all $\xi>0$,$$\lim_{T \to \infty}\Bbb E \bigg[ \exp \bigg( -\xi \exp \big( Q(T) \big) \bigg) \bigg] = \Bbb E [ 1_{[\mathbf a_1-x \leq 0]} ] = \Bbb P(\mathbf a_1 \leq x) . $$ 
	At this point in the proof, we state a general fact: If $\{X_n\}$ is any collection of non-negative random variables such that for every $\xi>0$ we have that $\Bbb E[ e^{-\xi X_n}] \stackrel{n\to \infty}{\longrightarrow} c \in [0,1]$, then we necessarily have that $X_n$ converges in distribution (as $n \to \infty$) to a random variable which equals $0$ with probability $c$, and equals $+\infty$ with probability $1-c$. To prove this, we clearly have that $\Bbb E[f(X_n)] \stackrel{n\to \infty}{\longrightarrow} cf(0)+(1-c)f(\infty)$ for any function $f \in C([0,\infty], \Bbb R)$ which is a finite linear combination of functions of the form $x \mapsto e^{-\xi x}$, with $\xi \geq 0$. By Stone-Weierstrass, such functions are uniformly dense in $C([0,\infty], \Bbb R)$, therefore we conclude the same result for every $f \in C([0,\infty], \Bbb R)$.
	\\
	\\
	From the previous paragraph and the preceding computations, we conclude that $$e^{Q(T)} \;\stackrel{d}{\longrightarrow}\;\; 0 \cdot 1_{[ \mathbf a_1 \leq x]} + \infty \cdot 1_{ [\mathbf a_1 >x]}, \;\;\;\;\;\; \text{as }T \to \infty .$$ Consequently, we find that $$\lim_{T \to \infty} \Bbb P(Q(T) \leq 0) = \lim_{T \to \infty} \Bbb P(e^{Q(T)} \leq 1) = \Bbb P(\mathbf a_1 \leq x)$$ which is indeed the desired result.
	\end{proof}

    \textbf{}\\
	Although the primary application of our results is the aforementioned limit theorem, we will mention that there are also other applications which come to mind. The main key behind our approach is the fine heat kernel estimates which are given in Section 3 below, and these are based on the work of [DT16] and [CS16]. These heat kernels $\mathbf p_t^R$ may be interpreted as transition probabilities for random walks with branching or killing in discrete-space/continuous-time. Indeed, positive choices for the boundary parameter $A$ correspond to the random walk particle being killed at the boundary with some probability depending on $A$, whereas negative choices for $A$ correspond to a new particle being created at the boundary at some rate depending on $A$ (so we get a branching process). As an indirect corollary, we get some fairly intricate information about the behavior of continuous-time random walks and branching processes.
	\\
	\\
	Two weeks ago, there was a related result which was posted by Goncalves, Perkowski, and Simon in [GPS17]. In that paper, they prove (under the assumption of stationary initial data and boundary parameters $A=B=-\frac{1}{2}$) weakly asymmetric convergence of ASEP to stochastic Burgers using very different methods than the ones we use here (the Cole-Hopf transform). Instead, they use the notion of energy solutions which was developed in [GJ14] and proved to be unique in [GP16]. While their method has the advantage of avoiding the Cole-Hopf transform, it is only able to deal with stationary initial data due to the fact that they use various hydrodynamic estimates which are only available for stationary (Bernoulli) initial data with boundary parameters $A=B=-\frac{1}{2}$. Despite these restrictions, [GPS17] has made the important observation that the energy-solution approach leads to additional correction terms at the boundary, which are not observed in our work.
	\\
	\\
	\textbf{Outline:} The paper will be organized as follows. In Section 2 we will formally define ASEP on the half-line and on bounded intervals. Furthermore, we will define the appropriate scaling of the model under which it converges to the KPZ equation. Section 3 contains a number of heat kernel estimates for the Robin-boundary Laplacian in discrete-space and continuous-time. We also give a construction of the continuum Robin heat kernel and prove a couple of useful estimates about it. Section 3 is the main key behind our approach, hence this is the longer section which is extensively used in the remainder of the paper. In Section 4, we will define what it means to solve the KPZ equation on $[0,\infty)$ and $[0,1]$ with Neumann boundary conditions. This involves the usual notions of mild solutions and weak solutions. In Section 5, we will prove that the model defined in Section 2 converges to the KPZ equation with Neumann boundary conditions, under the assumption of ``near-equilibrium" initial data. The key behind this is to prove some uniform Hölder estimates for the weakly-scaled, exponentially transformed height functions of the ASEP, and to prove a ``crucial cancellation" as in [CS16, Section 5]. The key novelty of our approach will be that we avoid the Green's function analysis of [CS16] but instead rely purely on the heat kernel estimates to prove this cancellation. In Section 6, we will try to generalize our results to the case when the initial data in the stochastic heat equation is $\delta_0$, which is very far from equilibrium. The idea here is to note that in short time, the solution to the stochastic heat equation started from $\delta_0$ stabilizes to equilibrium.
	\\
	\\
    \textbf{Acknowledgements:} The author wishes to thank Ivan Corwin for suggesting the problem, for providing helpful discussions about various issues which came up during the writing of the paper, and also for thoroughly reading the preliminary drafts of this paper. We also wish to thank Hao Shen and Li-Cheng Tsai, who provided some very useful discussions.
	
	\section{Definition of the Model and Scalings}
	
	Much of this section is adapted directly from the primary reference [CS16], but we reiterate all of the definitions provided there (albeit emphasizing different points) so that the reader will have easy access to this material.
	\\
	\\
	Let us first introduce some notational conventions. Firstly, we will always use lowercase letters $s,t,x,y$ when working with our particle system on a microscopic (discrete) scale, or looking at ``local" properties. In contrast, we will use capital letters $S,T,X,Y$ when working on a macroscopic scale or dealing with objects which are continuous (or limiting to a continuous object). Hopefully this will become clearer as the reader progresses through the paper.
	\\
	\\
	Let us define a few operators which we will work with throughout the paper. Firstly, if $f: \Bbb Z \to \Bbb R$ is any function then $$\nabla^+f(x) := f(x+1)-f(x)$$
	$$\nabla^-f(x) := f(x-1)-f(x)$$ $$\Delta f(x): = f(x+1)+f(x-1)-2f(x).$$ One should note that $\nabla^+ \nabla^- = \nabla^- \nabla^+ = \Delta$.
	\\
	\\
	Also, we will always use $\Lambda$ to denote either the discrete interval $\{0,...,N\}$ or the discrete half-line $\Bbb Z_{\geq 0}$ or both (depending on context). Similarly, we will use $I$ to denote the continuous interval $[0,1]$ or $[0,\infty)$.
	
	\begin{defn}[ASEP-H]
		Fix $p,q,\alpha,\gamma\geq 0$. We define the half-line ASEP (or ASEP-H for brevity), to be the following continuous-time Markov process. We let $\eta(x)\in \{-1,1\}$ denote the occupation variable at site $x\ge 1$, where the value $-1$ denotes an unoccupied site and $+1$ denotes an occupied site. The state space is then $\eta\in \{-1,1\}^{\Bbb Z_{\geq 1}}$. The dynamics are given as follows: For $x>0$, a particle jumps from site $x$ to $x+1$ at exponential rate $$ \frac{p}{4}(1+\eta_t(x))(1-\eta_t(x+1)).$$ In other words, it jumps at rate $p$ if site $x$ is occupied and site $x+1$ is not, otherwise it feels no inclination to jump (or the site $x$ may be unoccupied). Similarly, for $x\geq 1$, a particle jumps from site $x+1$ to $x$ at rate $$\frac{q}{4} (1-\eta_t(x))(1+\eta_t(x+1)).$$ Furthermore, a particle at site $x=1$ is annihilated and created (respectively) at rates $$\frac{\gamma}{2}(1+\eta_t(0))\;\; \;\;\;\;\;\text{and}\;\;\;\;\;\;\;\;\frac{\alpha}{2}(1-\eta_t(0)) .$$ All jumps and annihilations/creations occur independently of each other.
	\end{defn}
	
	\begin{defn}[ASEP-B]
		Fix $p,q,\alpha,\gamma, \beta,\delta \geq 0$. We define the half-line ASEP (or ASEP-B for brevity), to be the following continuous-time Markov process. We let $\eta(x)\in \{-1,1\}$ denote the occupation variable at site $x$, where we think of $-1$ denoting an unoccupied site and $+1$ denoting an occupied site. The state space is then $\eta\in \{-1,1\}^{\{1,...,N\}}$. The dynamics are the same as those of the ASEP-H, except that a particle at site $x=N$ is annihilated or created (respectively) at rates $$\frac{\beta}{2}(1+\eta_t(0)) \;,\;\;\;\;\;\;\;\;\;\frac{\delta}{2}(1-\eta_t(0)). $$ All jumps and annihilations/creations occur independently of each other.
	\end{defn}
	
	\begin{defn}[Height Functions] Consider the model ASEP-H or ASEP-B defined above. For $t \geq 0$ we define $h_t(0)$ to be twice the net number of particles removed (i.e., twice annihilations minus creations) that have occurred up to time $t$. We then define $$h_t(x):= h_t(0)+\sum_{j=0}^x\eta_t(j)$$ which is an integrated version of the ASEP (in the sense that $\nabla^+h_t(x)= \eta_t(x+1)$).
		
	\end{defn}
	
	The purpose of the above definition of $h_t(0)$ is so that the annihilation or creation of particles does not affect the value of the height function on the interior points of the interval.
	\\
	\\
	We now move onto the appropriate scaling for our models. A well-known fact is that there is no way to scale time, space, and fluctuations in the KPZ equation so as to leave its solution invariant in law (the universal fixed point for the [1:2:3] scaling is highly nontrivial and was recently just proved for TASEP in [MQR16]). However there are a couple of well-known $weak$ scalings which fix the law, in which we simultaneously scale the model parameters $together$ with the time, space, and fluctuations. The weak scaling considered for the ASEP height functions in this paper is the following:
	
	\begin{defn}[Weakly Asymmetric Scaling]\label{124} Throughout this paper we will fix $A,B \in \Bbb R$. Let $\epsilon>0$ small enough so that all rates defined below are positive. We define $p =\frac{1}{2} e^{\sqrt{\epsilon}}$ and $q=\frac{1}{2}e^{-\sqrt{\epsilon}}$, and we also define $\mu_A=1-A\epsilon$ and $\mu_B=1-B\epsilon$. We then define $$\alpha= \frac{p^{3/2}(p^{1/2}-\mu_Aq^{1/2})}{p-q}\;,\;\;\;\;\;\;\;\; \beta = \frac{p^{3/2}(p^{1/2}-\mu_Bq^{1/2})}{p-q},$$ $$\gamma=\frac{q^{3/2}(q^{1/2}-\mu_Ap^{1/2})}{q-p}\;,\;\;\;\;\;\;\;\; \delta=\frac{q^{3/2}(q^{1/2}-\mu_Bp^{1/2})}{q-p},$$ and we will always consider ASEP-H and ASEP-B with these parameters. For ASEP-B, we make the further assumption that $\epsilon = \frac{1}{N}$, where $N$ is the length of the bounded interval.
		
	\end{defn}
	
	\begin{rk}
	As in [CS16, Remark 2.11], we note that we have the following asymptotics: $$p = \frac{1}{2} +\frac{1}{2} \sqrt \epsilon+O(\epsilon),$$ $$q = \frac{1}{2}-\frac{1}{2}\sqrt \epsilon+O(\epsilon).$$ For the  creation/annihilation rates, we have
	$$\alpha = \frac{1}{4} + \big(\frac{3}{8}+\frac{1}{4}A\big) \sqrt \epsilon+O(\epsilon), \;\;\;\;\;\;\; \beta = \frac{1}{4} + \big(\frac{3}{8}+\frac{1}{4}B\big) \sqrt \epsilon+O(\epsilon),$$ $$\gamma = \frac{1}{4} - \big(\frac{3}{8}+\frac{1}{4}A\big) \sqrt \epsilon+O(\epsilon), \;\;\;\;\;\;\; \delta = \frac{1}{4} - \big(\frac{3}{8}+\frac{1}{4}B\big) \sqrt \epsilon+O(\epsilon).$$ As stated in [CS16], this physically corresponds to ``zooming into an $\epsilon^{1/2}$-window" around the critical triple point of the Open ASEP.
	\end{rk}
	
	\begin{figure}[h]
	\centering
	\includegraphics[scale=0.6]{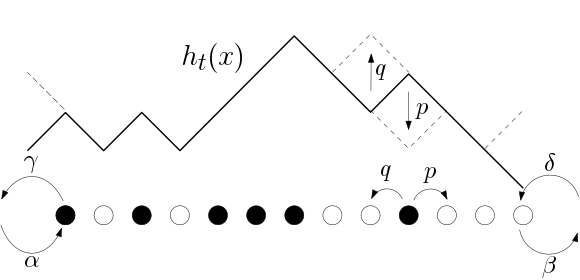}
	\caption{A depiction of the height function at a time $t$, in bold. We also show the jump/creation/annihilation rates and the way the height function changes when jumps of this rate occur.}
	\end{figure}
	
	\begin{defn}[Gartner Transform]\label{125} For $x\in \Bbb Z_{\geq 0}$ we define the ASEP-H Gartner transformed height function as $$Z_t(x):=e^{-\lambda h_t(x)+\nu t},$$ where $$\lambda = \frac{1}{2}\log \frac{q}{p} ,\;\;\;\;\;\;\;\;\; \nu = p+q-2\sqrt{pq}.$$
		The definition is the same for ASEP-B, but we only consider $x \in \{0,...,N\}$.
	\end{defn}
	
	We remark that although the above $Z$ depends on $\epsilon$ via the parameters $p,q$, we chose not to make this explicit because of notational convenience. The choice for these specific values of $\lambda,\nu$ is not immediately obvious, but becomes clear in the following theorem, which is really the key behind proving weak convergence to SHE in the limit as $\epsilon \to 0$. We also remark that $$\lambda= -\sqrt \epsilon,\;\;\;\;\;\;\;\;\;\; \nu = \frac{1}{2} \epsilon+ \frac{1}{24}\epsilon^2 +O(\epsilon^3)$$which is the reason behind the temporal drift in Theorem \ref{mr}.
	
	\begin{lem}[Hopf-Cole-Gartner Transform]\label{74} The Gartner-transformed height functions $Z_t(x)$ for ASEP-H satisfy the following discrete SHE:
		
		$$dZ_t(x) = \frac{1}{2}\Delta Z_t(x)dt +dM_t(x)$$ where $M_t(x)$ is a pure-jump martingale (i.e., a sum of compensated Poisson processes) whose predictable bracket $\langle M(x),M(y)\rangle_t$ satisfies the following asymptotics as $\epsilon \to 0$:
		
		\begin{equation}\label{75}\frac{d}{dt} \langle M(x),M(y) \rangle_t = \begin{cases} 0 & x \neq y \\ \epsilon Z_t(x)^2-\nabla^+Z_t(x)\nabla^-Z_t(x)+o(\epsilon)Z_t(x)^2 & x=y>0 \\ \epsilon Z_t(x)^2 + o(\epsilon) Z_t(x)^2 & x=y=0
		\end{cases}\end{equation} for all $x \in \Bbb Z_{\geq 0}$. Moreover, $Z_t$ satisfies a discrete Robin boundary condition for all $t$ $$Z(-1)=\mu_AZ_t(0).$$ For ASEP-B, the same holds true for $x \in\{0,...,N-1\}$, but we have that the third asymptotic in (\ref{75}) holds for $x=y=N$ and moreover $$Z_t(N+1) = \mu_B Z_t(N).$$
		
	\end{lem}
	
	\begin{proof} The complete proof is found in [CS16, Lemmas 3.1 and 3.3]. We also note that the original paper [Gar88] was the first to recognize this transformation in the whole-line case, while [BG97] used it to prove the first of these KPZ-type convergence results. We further remark that (for fixed choices of $p,q,\mu_A,\mu_B$) this lemma $only$ holds true for the exact values of $\alpha, \gamma, \beta, \delta$ chosen in Definition \ref{124} as well as the values of $\lambda, \nu$ chosen in Definition \ref{125}. See the proof in [CS16]. \end{proof}

	\section{Heat Kernel Estimates}
	
	In this (fairly lengthy) section, we will provide all of the technical estimates which will be used in the analysis used to prove convergence of the Gartner-transformed height functions to the solutions of the SHE. There will not be very much motivation, but the purpose of the estimates will become clear as the reader progresses through Sections 5 and 6. Therefore, the reader may wish to skip some of the more technical results of this section, and come back as needed while progressing through the remainder of the paper.
	\\
	\\
	Before moving onto the estimates, one should remark that the Robin heat kernels $\mathbf p_t^R(x,y)$ which are introduced below may be interpreted as transition probabilities for a continuous-time, discrete-space random walk with killing and branching at the boundary with certain probability, according the the corresponding parameters $A,B$. The case when $A>0$ corresponds to killing, while the case $A<0$ corresponds to branching (and similarly for $B$). Likewise, the continuous-space kernels $P_T$ which are introduced in Section 3.3 may be interpreted as transition densities for Brownian motion with branching/killing. See for instance [CS16, Section 4.1.1], [NMW82], or [IM63] for more on this subject.
	\\
	\\
	As a notational convention, we will usually write $C$ for constants, and we will not generally specify when irrelevant terms are being absorbed into the constants. We will also write $C(A), C(A,T), C(a,b,A,B,T)$ whenever we want to specify exactly which parameters the constant depends on. This will not always be specified, though.
	\\
	\\
	We also mention that we will occasionally use estimates from [DT16, Appendix A], where they derive very useful estimates for the (standard) whole-line heat kernel $p_t$. Hence the reader may wish to briefly look at that appendix while proceeding.
	
	\subsection{Half-Line Estimates}

	Unless otherwise specified, $A\in \Bbb R$ will be fixed. For $\epsilon>0$, we set $\mu_A = 1-A\epsilon$, and we let $\textbf{p}_t^{R}(x,y)$ be the semi-discrete heat kernel on $\mathbb Z_{\geq -1}$ with Robin boundary parameter $\mu_A$: this is defined as the fundamental solution to the discrete-space, continuous-time heat equation: $$\partial_t \mathbf p_t^R(x,y) = \frac{1}{2}\Delta \mathbf p_t^R(x,y),\;\;\;\;\;\;\;\;\; \mathbf p_0^R(x,y)=1_{\{x=y\}}, \;\;\;\;\;\;\;\;\;\; x,y \in \Bbb Z_{\geq 0}$$ with the boundary condition $$\mathbf p_t^R(-1,y) = \mu_A \mathbf p_t^R(0,y).$$Here, the discrete Laplacian $\Delta$ is taken in the second spatial coordinate. The ``generalized image method" of [CS16, Section 4.1] says that
	\begin{equation}\label{1}
	\textbf{p}_t^R(x,y) = p_t(x-y)+\mu_A p_t(x+y+1) + (1-\mu_A^{-2})\sum_{z=2}^{\infty}p_t(x+y+z)\mu_A^{z}
	\end{equation}
	where $p_t(x)$ is the standard heat kernel on the whole line $\Bbb Z$ (i.e., the unique solution to the continuous-time, discrete-space equation $\partial_tp_t(x) = \frac{1}{2} \Delta p_t(x) $ with $p_0(x) = 1_{\{x=0\}}$). One may also directly check that \eqref{1} holds true.
	\\
	\begin{prop}\label{2}
		Fix $A\in \Bbb R$ and $T>0$. For $b \geq 0$, there is a constant $C(A,b,T)$ (not depending on $\epsilon$) such that for all $t \in [0,\epsilon^{-2}T]$, and all $x,y \in \mathbb Z_{\geq 0}$ we have that $$\mathbf{p}_t^{R}(x,y) \leq C(A,b,T) (1 \wedge t^{-1/2}) e^{-b|x-y| (1 \wedge t^{-1/2})}.$$
	\end{prop}
	
	\begin{proof} Note that the first two terms appearing in the right-hand-side of Equation \eqref{1} already satisfy a bound of the desired form, by the standard (whole-line) heat kernel estimates [DT16, (A.12)]. Therefore it suffices to show that the third term appearing in the right side of (1) also satisfies a bound of the desired type. Since $\mu_A=1-A\epsilon$ it follows that $\log(\mu_A) \leq C(A)\epsilon$. For $t \in [0,\epsilon^{-2}T]$ we then have $\epsilon \leq C(T) (1 \wedge t^{-1/2})$ and therefore $\log(\mu_A) \leq C(A,T) (1 \wedge t^{-1/2})$. Summarizing, there exists $C(A,T)$ such that for all $t \in [0,\epsilon^{-2}T]$ we have that $$\mu_A \leq e^{C(A,T)(1\wedge t^{-1/2})}.$$ 
	
	If $A \geq 0$ then we can just take $C(A,T)=0$, otherwise $C$ will be positive. Using the standard heat kernel bound [DT16, (A.12)] it follows that for $b>C(A,T)$ we have 
	\begin{align*}
	\sum_{z=2}^{\infty} p_t(x+y+z)\mu_A^z &\leq C(b) \sum_{z=2}^{\infty} (1 \wedge t^{-1/2})e^{-b(x+y+z)(1 \wedge t^{-1/2})} \cdot e^{C(A,T)(1\wedge t^{-1/2})z} \\&=C(b) e^{-b(x+y) (1 \wedge t^{-1/2})} (1 \wedge t^{-1/2}) \frac{e^{2(C(A,T)-b)(1\wedge t^{-1/2})}}{1-e^{(C(A,T)-b)(1 \wedge t^{-1/2})}} \\&\leq C(A,b,T) e^{-b(x+y) (1 \wedge t^{-1/2})} 
	\end{align*}
	where we used the fact that $e^{-2q}/(1-e^{-q}) \leq 1/q$ for any $q \geq 0$. Since $t\leq \epsilon^{-2}T$ we have that $|1-\mu_A^{-2}| \leq C(A)\epsilon \leq C(A,T) (1 \wedge t^{-1/2})$ and the result follows. \end{proof}

	\begin{prop}\label{3}
		
		Fix $A\in \Bbb R$ and $T>0$. For $b \geq 0$, there is a constant $C(A,b,T)$ (not depending on $\epsilon$) such that for all $t \in [0,\epsilon^{-2}T]$, all $x,y \in \mathbb Z_{\geq 0}$, all $|n| \leq \lceil t^{1/2} \rceil$, and all $v\in[0,1]$ we have that
		$$|\mathbf{p}_t^R(x+n,y)-\mathbf{p}_t^R(x,y)| \leq C(A,b,T) (1 \wedge t^{-(1+v)/2}) |n|^v e^{-b|x-y| (1 \wedge t^{-1/2})}.$$ When $b=0$ this bound holds for all $n \in \Bbb Z$, not just for $n\leq t^{1/2}$.
	\end{prop}
	\begin{proof} The corresponding bounds already hold for the whole-line kernel $p_t(x)$, by [DT16, (A.13)]. So we proceed exactly as in Proposition \ref{2}, using \eqref{1} with $p_t$ replaced by its $n$-point gradient $\nabla_np_t$. The fact that the bound holds for all $n$ when $b=0$ is a consequence of the triangle inequality applied to the case when $n=1$ and $v=1$. \end{proof}
	
	\begin{cor}\label{4}
		For any $T \geq0$ and $a_1,a_2\geq 0$, there exists some constant $C=C(a_1,a_2,A,T)$ such that for all $x \geq 0$ and $t \leq \epsilon^{-2}T$ we have \begin{align*}\sum_{y \geq 0} \mathbf p_t^R(x,y) e^{a_1|x-y|(1 \wedge t^{-1/2})} e^{a_2 \epsilon y} &\leq Ce^{a_2 \epsilon x}, \\ 
		\sum_{y \geq 0} |\nabla^{\pm}\mathbf p_t^R(x,y)| e^{a_1|x-y|(1 \wedge t^{-1/2})} e^{a_2 \epsilon y} &\leq C(1 \wedge t^{-1/2})e^{a_2 \epsilon x}. \end{align*}
		If we replace $e^{a_2 \epsilon y}$ with $e^{a_2 \epsilon |x-y|}$ in the first two expressions, then these bounds hold without the factor $e^{a_2 \epsilon x}$ on the RHS.
	\end{cor}
	
	\begin{proof} Using Proposition \ref{2}, we find that $\mathbf p_t^R(x,y) \leq C(b,T,A) (1 \wedge t^{-1/2}) e^{-b|x-y|(1 \wedge t^{-1/2})}$, and moreover $$e^{a_2 \epsilon y} \leq e^{a_2 \epsilon x}e^{a_2 \epsilon |x-y|} \leq e^{a_2 \epsilon x} e^{a_2 T^{1/2} (1 \wedge t^{-1/2}) |x-y|}$$ since $ t \leq \epsilon^{-2}T$. Consequently,
	\begin{align*}
	\sum_{y \geq 0} \mathbf p_t^R(x,y) e^{a_1|x-y|(1 \wedge t^{-1/2})} e^{a_2 \epsilon y} & \leq C e^{a_2 \epsilon x} \sum_{y \in \Bbb Z}(1 \wedge t^{-1/2} ) e^{(-b+a_1+a_2T^{1/2})|x-y|(1 \wedge t^{-1/2})} .
	\end{align*}
	Letting $b :=1+a_1+a_2T^{1/2}$ we compute the sum on the RHS:
	\begin{align*} 
	\sum_{y \in \Bbb Z}(1 \wedge t^{-1/2} ) e^{-|x-y|(1 \wedge t^{-1/2})} &=(1 \wedge t^{-1/2}) \frac{1+e^{-(1\wedge t^{-1/2})}}{1-e^{-(1\wedge t^{-1/2})}} \leq C
	\end{align*}
	where we used the fact that $(1+e^{-q})/(1-e^{-q}) \leq 1+2/q$ for $q \geq 0$. This proves the first inequality, and the second one is proved similarly using Proposition \ref{3} instead of \ref{2}. The final statement is proved in a similar way.\end{proof}

	We now turn to proving temporal bounds for the semi-discrete heat kernel, which will be useful in proving tightness. This involves different methods that the ones used to prove the spatial bounds above.
	\\
	\begin{lem}\label{5}
		Let $p_t(x)$ denote the standard heat kernel on the whole line $\mathbb Z$ (as defined below Equation \eqref{1}). Then for $\mu>0$, $t \geq 0$ and $x \in \mathbb Z_{\geq 0}$ we have the equality $$\sum_{z=-\infty}^{\infty} p_t(x+z) \mu^z = \mu^{-x} \exp{\bigg[ \frac{1}{2} \big(\mu + \mu^{-1} -2\big)t\bigg]}.$$ 
	\end{lem}
	\begin{proof} Fixing $\mu$, let $F(t,x)$ denote the sum on the left side. Note that $F$ is defined by a convolution involving $p_t$, therefore $$\partial_t F(t,x) = \frac{1}{2} \Delta F(t,x) = \frac{1}{2}(\mu+\mu^{-1}-2)F(t,x).$$ Furthermore $F(0,x) = \mu^{-x}$, because $p_0(x) = 1_{[x=0]}$. This proves the given formula. Another way of putting this is that $F$ is defined by convolving the semigroup $p_t$ with an eigenfunction of $\Delta$, whose eigenvalue is $\mu+\mu^{-1}-2$. \end{proof}
	
	\begin{prop}\label{6}
		
		Fix $A\in \Bbb R$ and $T>0$. There is a constant $C(A,T)$ (not depending on $\epsilon$) such that for all $s<t \in [0,\epsilon^{-2}T]$, all $x,y \in \mathbb Z_{\geq 0}$, and all $v\in[0,1]$ we have that
		$$|\mathbf{p}_t^R(x,y)-\mathbf{p}_s^R(x,y)| \leq C(A,T) (1 \wedge s^{-1/2-v}) (t-s)^v.$$
	\end{prop}
	
	\begin{proof}First note that it suffices to prove these formulas when $v=0$ and $v=1$ since the middle cases follow by a straightforward interpolation. The $v=0$ case follows from Proposition \ref{2}, and thus we only need to prove the $v=1$ case. We will only consider the case when $A \leq 0$ (so $\mu_A \geq 1$), because the $A>0$ case easier and involves similar methods (see [CS16, Proposition 4.11]). To this end, we define a function $$F(t,x)= \sum_{z=-\infty}^{\infty} p_t(x+z)\mu_A^z= \mu_A^{-x}\exp{\bigg[ \frac{1}{2}(\mu_A+\mu_A^{-1}-2) t\bigg]}.$$
	
	We rewrite Equation \eqref{1} as 
	\begin{align*}
	\textbf{p}_t^R(x,y) &= p_t(x-y)+\mu_A p_t(x+y+1) + (1-\mu_A^{-2})F(t,x+y) -  (1-\mu_A^{-2})\sum_{z=-1}^{\infty}p_t(x+y-z)\mu_A^{-z}\\ &= J_1(t,x,y)\;\;\;\;+\;\;\;\;J_2(t,x,y)\;\;\;\;+\;\;\;\;J_3(t,x,y)\;\;\;\;-\;\;\;\;J_4(t,x,y).
	\end{align*}
	Now we only need to bound each of the differences $|J_i(t,x,y)-J_i(s,x,y)|$ for $1 \leq i \leq 4$. When $i=1,2$, the desired bounds follow directly from [DT16. (A.10)]. For the $i=3$ bound, note that \begin{align*}|F(t,x)-F(s,x)| &= \mu_A^{-x} e^{ \frac{1}{2}(\mu_A+\mu_A^{-1}-2) s}\bigg( e^{ \frac{1}{2}(\mu_A+\mu_A^{-1}-2) (t-s)}-1 \bigg)\\&\leq 1 \cdot e^{ \frac{1}{2}(\mu_A+\mu_A^{-1}-2) s} \cdot \bigg( \frac{1}{2}(\mu_A+\mu_A^{-1}-2) (t-s)\; e^{ \frac{1}{2}(\mu_A+\mu_A^{-1}-2) (t-s)}\bigg) \\ &= e^{ \frac{1}{2}(\mu_A+\mu_A^{-1}-2) t} \cdot \frac{1}{2}(\mu_A+\mu_A^{-1}-2) (t-s)
	\end{align*}
	where we used the bound $e^q-1 \leq qe^q$ with $q= \frac{1}{2}(\mu_A+\mu_A^{-1}-2) (t-s)$. Now we again recall that $\mu_A=1-A\epsilon$ so that $\mu_A^{-1}=1+A\epsilon+A^2\epsilon^2+O(\epsilon^3)$, and therefore $\mu_A+\mu_A^{-1}-2 \leq CA^2\epsilon^2$ for small enough $\epsilon$. Hence when $s,t\leq \epsilon^{-2}T$, the previous bound gives $$|F(t,x)-F(s,x)| \leq e^{\frac{1}{2}CA^2T} \cdot \frac{1}{2}CA^2\epsilon^2 (t-s)=C(A,T) \; \epsilon^2 (t-s).$$ Now we recall that $1-\mu_A^{-2} \leq C(A)\epsilon$ and $\epsilon< C(T) s^{-1/2}$ so that
	\begin{align*}|J_3(t,x,y)-J_3(s,x,y)| &= (1-\mu_A^{-2})|F(t,x+y)-F(s,x+y)| \\&\leq C(A,T) \; \epsilon^3 (t-s) \\&\leq C(A,T) s^{-3/2}(t-s)
	\end{align*} 
	which proves the desired bound for $J_3$.
	
	\textbf{}\\
	To prove the bound for $J_4$, we again apply the whole-line estimate [DT16, (A.10)] 
	\begin{align*}
	|J_4(t,x,y)-J_4(s,x,y)| & \leq (1-\mu_A^{-2}) \sum_{z=-1}^{\infty} \big|p_t(x+y-z)-p_s(x+y-z) \big| \mu_A^{-z} \\&\leq (1-\mu_A^{-2}) \sum_{z=-1}^{\infty} C s^{-3/2} (t-s) \mu_A^{-z}  \\&= Cs^{-3/2}(t-s) \mu_A(1+\mu_A^{-1}) 
	\end{align*}
	and we implicitly assume that $\epsilon$ is small enough so that $\mu_A(1+\mu_A^{-1}) \leq C(A)$. This proves the claim. \end{proof}
	
	\begin{prop}[Long-time Estimate]\label{lt1}
		There exist constants $C=C(A,B)$ and $K=K(A,B)$ such that for every $t\geq 0$ and $x,y \geq 0$ we have that $$\mathbf p_t^R(x,y)\leq C(t^{-1/2}+\epsilon)e^{K\epsilon^2t}. $$
	\end{prop}
	
	We remark that these are ``long-time" estimates because they are true uniformly over all $t>0$, i.e., constants don't depend on any terminal time $\epsilon^{-2}T$ (as opposed to most results here).
	
	\begin{proof} In Equation (\ref{1}), the first two terms are clearly bounded by $Ct^{-1/2}$ by [DT16, (A.10)] and the third term is bounded in absolute value by $|1-\mu_A^{-2}|\exp{\big[(\mu_A+\mu_A^{-1}-2)t\big]}$ by Lemma \ref{5}. Since $\mu_A=1-A\epsilon$, it follows that $|1-\mu_A^{-2}| <C\epsilon$ and $\mu_A+\mu_A^{-1}-2 \leq K\epsilon^2$. This completes the proof.
	\end{proof}
	
	\begin{prop}\label{7}
		For all $A$ and $T>0$, there exists $C=C(A,T)$ such that for $t\in[0,\epsilon^{-2}T]$ and $v\in [0,1]$ we have $$\sup_{ x\in \Bbb Z_{\geq 0}}\bigg| \sum_{y\geq 0} \textbf{p}^R_t(x,y) \;- \;1 \bigg| \leq C \epsilon^{v}   t^{v/2}.$$
	\end{prop}
	
	\begin{proof} One may use Lemma \ref{8} (below) in order to prove this claim. However, we choose to give a different proof which will generalize easily to the bounded-interval case. Let $$f(t,x) = \sum_{y \geq 0}  \textbf{p}^R_t(x,y).$$
	By (1) we know $ \textbf{p}^R_t(x,y) = \textbf{p}^R_t(y,x)$ and therefore
	\begin{align*}
	\partial_tf(t,x) &= \frac{1}{2}\Delta f(t,x) = \frac{1}{2}\sum_{y \geq 0}\big(  \textbf{p}^R_t(y+1,x)+\textbf{p}^R_t(y-1,x) - 2\textbf{p}^R_t(y,x) \big)\\ &= \frac{1}{2} \big( \textbf{p}^R_t(-1,x)-\textbf{p}^R_t(0,x) \big)= \frac{1}{2}(\mu_A-1)\textbf{p}^R_t(0,x)
	\end{align*} 
	where we canceled out many terms in the first equality of the second line. Note that $f(0,x)=1$, and $|\mu_A-1| = |A|\epsilon$. Finally (by Proposition 1.2) $\textbf{p}^R_t(0,x) \leq C(A,T) t^{-1/2}$, therefore when $v\in [0,1]$ we have
	\begin{align*}
	|f(t,x) -1| &\leq \int_0^t |\partial_sf(s,x)| ds =\frac{1}{2}|A| \epsilon \int_0^t \textbf{p}^R_s(0,x) ds \\ &\leq C(A,T) \;\epsilon \int_0^t s^{-1/2} ds = C(A,T) \;\epsilon \;t^{1/2} \leq C(A,T) \epsilon^v t^{v/2}
	\end{align*}
	where in the last inequality we used the fact that $\epsilon = \epsilon^{v}\epsilon^{1-v} \leq C(T)  \epsilon^{v} t^{(v-1)/2}$ since $t \in [0,\epsilon^{-2}T]$. This proves the claim. \end{proof}
	
	\textbf{}\\
	We now turn to proving certain ``cancellation estimates" which will be used in identifying the limiting measure on $C([0,T], C([0,\infty))$ as the solution to the stochastic heat equation.
	\\
	\begin{lem}\label{8}
		For the next few estimates we will distinguish between different values of $A$ by writing $\textbf{p}_t^R(x,y;A)$ for the ($\epsilon$-dependent) Robin heat kernel of parameter $A$. For all $A$, all $T>0$, and all $b \geq 0$, there exists $C(A,b,T)$ such that for all $x,y \in \Bbb Z_{\geq 0}$ and $t \in [0,\epsilon^{-2}T]$ we have
		\begin{align*}
		|\mathbf{p}^R_t(x,y;A)-\mathbf{p}^R_t(x,y;0)| &\leq C(A,b,T)\; \epsilon\; e^{-b(x+y)(1 \wedge t^{-1/2})},\\
		|\nabla^{\pm}\mathbf{p}^R_t(x,y;A)-\nabla ^{\pm} \mathbf{p}^R_t(x,y;0)| &\leq C(A,b,T) \; \epsilon \; (1 \wedge t^{-1/2}) e^{-b(x+y)(1 \wedge t^{-1/2})}.
		\end{align*}
		where $\nabla^{\pm}$ denotes the discrete gradient in the first spatial coordinate.
	\end{lem}
	
	\begin{proof} Note by (\ref{1}) that $$|\textbf{p}^R_t(x,y;A)-\textbf{p}^R_t(x,y;0)| = (\mu_A-1)p_t(x+y+1)+ (1-\mu_A^{-2})\sum_{z=2}^{\infty} p_t(x+y+z)\mu_A^z$$ where $p_t$ is the standard (whole-line) heat kernel. Since $\mu_A = 1-A\epsilon$, it follows that $|\mu_A-1|$ and $|1-\mu_A^{-2}|$ are both bounded by $C(A)\epsilon$. Moreover $p_t(x+y+1) \leq C(b) e^{-b(x+y)(1 \wedge t^{-1/2})}$ by the standard heat kernel bound [DT16, (A.12)]. Hence we just need to show that $$\sum_{z=2}^{\infty} p_t(x+y+z)\mu_A^z \leq C(A,b,T) e^{-b(x+y)(1 \wedge t^{-1/2})}.$$
	But this was done during Proposition \ref{2}. The proof for the gradient estimates is similar, but we get left with an extra $1\wedge t^{-1/2}$ by setting $v=1$ in Proposition \ref{3}. \end{proof}
	
	\begin{lem}\label{9} For $A\in \Bbb R$, $t \geq 0$, and $x,y \in \Bbb Z_{\geq 0}$ we define $$K_t(x,y;A):=\nabla^+ \mathbf{p}^R_t(x,y;A) \nabla^- \mathbf{p}^R_t(x,y;A).$$ For $a,T \geq 0$ there exists a constant $C=C(a,A,T)$ such that $$\sum_{y \geq 1}\int_0^{\epsilon^{-2}T} |K_t(x,y;A)-K_t(x,y;0)| e^{a \epsilon|x-y|} dt  \leq C\epsilon^{1/2}.$$ 
		
	\end{lem}
	
	\begin{proof} Write \begin{align*}|K_t(x,y;A)-K_t(x,y;0)|&\leq |\nabla^+ \textbf{p}^R_t(x,y;A)-\nabla^+ \textbf{p}^R_t(x,y;0)| |\nabla^- \textbf{p}^R_t(x,y;A)| \\ &\;\;\;\;+ |\nabla^+ \textbf{p}^R_t(x,y;0)| | \nabla^- \textbf{p}^R_t(x,y;A)-\nabla^- \textbf{p}^R_t(x,y;0)| \\&=: J_1(t,x,y) +J_2(t,x,y).
	\end{align*}
	By Lemma \ref{8} we have that
	\begin{align*}
	J_1(t,x,y) &= |\nabla^+ \textbf{p}^R_t(x,y;A)-\nabla^+ \textbf{p}^R_t(x,y;0)| |\nabla^- \textbf{p}^R_t(x,y;A)| \\ &\leq C(A,T) \; \epsilon\; (1 \wedge t^{-1/2}) |\nabla^- \textbf{p}^R_t(x,y;A)|.
	\end{align*}
	Since $t \in [0,\epsilon^{-2}T]$ we have that $\epsilon=\epsilon^{1/2}\epsilon^{1/2} \leq C(T) \epsilon^{1/2} (1 \wedge t^{-1/4})$ and thus we see that $$J_1(t,x,y) \leq C(A,T) \; \epsilon^{1/2} (1 \wedge t^{-3/4} ) |\nabla^- \textbf{p}^R_t(x,y;A)|.$$ Consequently
	\begin{align*}\sum_{y \geq 1} J_1(t,x,y) e^{a\epsilon |x-y|} &\leq C(A,T) \epsilon^{1/2}(1 \wedge t^{-3/4}) \sum_{y \geq 1} |\nabla^- \textbf{p}^R_t(x,y;A)|e^{a \epsilon |x-y|}\\&\leq C(a,A,T) \epsilon^{1/2} (1 \wedge t^{-5/4}).
	\end{align*}
	In the last inequality we used the second bound in Corollary \ref{4}. Now integrating both sides of this inequality from $0$ to $\epsilon^{-2}T$ we find that \begin{align*}
	\sum_{y \geq 1}\int_0^{\epsilon^{-2}T} J_1(t,x,y) e^{a \epsilon|x-y|} dt  &= \int_0^{\epsilon^{-2}T}\bigg( \sum_{y \geq 1} J_1(t,x,y) e^{a\epsilon |x-y|}\bigg)dt \\ & \leq \int_0^{\infty} C(A,T,a) \epsilon^{1/2} (1 \wedge t^{-5/4}) dt \\ &=C(A,T,a) \epsilon^{1/2}.
	\end{align*}
	A similar argument shows that $J_2$ satisfies a similar bound. This proves the claim.\end{proof}
	
	\begin{lem}\label{108}
		Let us write $\mathbf p_t^R(x,y;A)$ as in the preceding lemma. For $x, \bar x \in \Bbb Z_{\geq 0}$, we have that $$\sum_{y \geq 0}\int_0^{\infty} \nabla^+ \mathbf p_t^R(x,y;0) \nabla^+ \mathbf p_t^R (\bar x, y ;0) dt = 1_{\{x=\bar x\}}.$$ 
	\end{lem}
	
	\begin{proof} Let us recall the summation-by-parts identity: if $u,v: \Bbb Z_{\geq -1} \to \Bbb R$ are absolutely summable, then $$\sum_{y=0}^{\infty} u(y) \Delta v(y) = u(-1)\nabla^- v(0) - \sum_{y=-1}^{\infty} \nabla^+u(y) \nabla^+ v(y).$$ Letting $u = \mathbf p_t^R(x, \cdot; 0)$ and $v= \mathbf p_t^R(\bar x, \cdot;0)$, the boundary terms will vanish and we get \begin{align*}-\sum_{y=0}^{\infty} \nabla^+ \mathbf p_t^R(x,y;0) \nabla^+ \mathbf p_t^R (\bar x, y ;0) &= \sum_{y=0}^{\infty}\mathbf p_t^R(x,y;0) \Delta \mathbf p_t^R (\bar x, y ;0)=\sum_{y=0}^{\infty}\Delta\mathbf p_t^R(x,y;0) \mathbf p_t^R (\bar x, y ;0).
		\end{align*}
		Since $\Delta \mathbf p_t^R = 2\partial_t \mathbf p_t^R$, this implies that $$-\sum_{y=0}^{\infty} \nabla^+ \mathbf p_t^R(x,y;0) \nabla^+ \mathbf p_t^R (\bar x, y ;0) = \sum_{y=0}^{\infty}\mathbf p_t^R(x,y;0) \partial_t \mathbf p_t^R (\bar x, y ;0)+\partial_t\mathbf p_t^R(x,y;0) \mathbf p_t^R (\bar x, y ;0).$$ Integrating both sides from $t=0$ to $\infty$ and using the semigroup property, we find that \begin{align*}
		-\sum_{y=0}^{\infty}\int_0^{\infty} \nabla^+ \mathbf p_t^R(x,y;0) \nabla^+ \mathbf p_t^R (\bar x, y ;0) dt&= \sum_{y=0}^{\infty} \mathbf p_t^R(x,y;0)\mathbf p_t^R(\bar x,y,0) \bigg|_{t=0}^{\infty} \\ &= \mathbf p_{2t}^R(x, \bar x;0)\bigg|_{t=0}^{\infty} \\ &= 0-1_{\{x=\bar x\}}
		\end{align*}which proves the claim.
	\end{proof}
	
	\begin{prop}\label{10}
		Let $A\in \Bbb R$, $T>0$, and $a>0$. Let $K_t$ be as in Lemma \ref{9}. There exists some $\epsilon_0=\epsilon_0(A,T,a)$ and some $c_*=c_*(A,T,a)<1$ such that for $\epsilon<\epsilon_0$ and $x \geq 1$ $$\sum_{y \geq 1}\int_0^{\epsilon^{-2}T} |K_t(x,y,A)| e^{a \epsilon|x-y|} dt  \leq c_*.$$ 
	\end{prop}
	
	\begin{proof} For the moment being, let us suppose that we have already proved the claim when $A=0$: there exists some $c_*<1$ such that for small enough $\epsilon$ we have $$\sum_{y \geq 1}\int_0^{\epsilon^{-2}T} |K_t(x,y;0)| e^{a \epsilon|x-y|} dt  \leq c_*.$$
		
		By Proposition A.8, we also have that $$\sum_{y \geq 1}\int_0^{\epsilon^{-2}T} |K_t(x,y;A)-K_t(x,y;0)| e^{a \epsilon|x-y|} dt \leq C(a,A,T)\epsilon^{1/2}.$$
		
		Putting together both these bounds and using the triangle inequality, we find that 
		$$\sum_{y \geq 1}\int_0^{\epsilon^{-2}T} |K_t(x,y;A)| \; e^{a \epsilon|x-y|} dt \leq c_*+C(a,A,T) \epsilon^{1/2}.$$
		Now consider $\epsilon$ small enough so that $C(a,A,T) \epsilon^{1/2} \leq (1-c_*)/2$, and we see that the RHS is smaller than $(1+c_*)/2 =: c_*'<1$ which proves the claim for arbitrary $A$.
		\\
		\\
		It remains to prove the claim when $A=0$. From now on, we will implicitly assume that $A=0$ and we will just write $K_t(x,y)$ and $\mathbf p_t^R(x,y)$ with the understanding that $A=0$. Let us first consider the case when $a=0$. For this, we imitate the proof of [CS16, Proposition 5.4]. Using Cauchy-Schwarz, it is true that $$\sum_{y \geq 1} \int_0^{\epsilon^{-2}T} |K_t(x,y)|dt < \bigg(\sum_{y \geq 1} \int_0^{\infty} (\nabla^+\mathbf p_t^R(x,y))^2 dt \bigg)^{1/2}\bigg(\sum_{y \geq 1} \int_0^{\infty} (\nabla^-\mathbf p_t^R(x,y))^2 dt \bigg)^{1/2}.$$ Using Lemma \ref{108}, it is easy to see that the RHS of this expression is equal to $1$. Moreover, the inequality is strict since $\nabla^+ \mathbf p_t^R \neq \nabla^- \mathbf p_t^R$. This proves that, for each fixed $x \in \Bbb Z_{\geq 0}$ and $\epsilon>0$, the LHS is strictly less than $1$. However, this strict inequality may no longer be true after taking the supremum over all $x$ and $\epsilon$. Thus, a stronger argument is needed. Recall the Lagrange identity, which says $$\bigg(\sum_i a_i^2 \bigg) \bigg( \sum_i b_i^2 \bigg) - \bigg( \sum_i a_ib_i \bigg)^2= \sum_{i<j} (a_ib_j-a_jb_i)^2.$$
		This means that $$\bigg( \sum_{y \geq 1} (\nabla^+ \mathbf p_t^R(x,y))^2 \bigg) \bigg( \sum_{y \geq 1} (\nabla^- \mathbf p_t^R(x,y) )^2 \bigg) - \bigg( \sum_{y \geq 1} |K_t(x,y)|\bigg)^2$$ \begin{equation}\label{111}= \sum_{1 \leq \bar y < y} \bigg( |\nabla^+ \mathbf p_t^R(x,y) \nabla^- \mathbf p_t^R(x, \bar y) | - |\nabla^- \mathbf p_t^R(x,y) \nabla^+ \mathbf p_t^R(x, \bar y) | \bigg)^2.\end{equation} Now, just as in [CS16, Corollary 5.4], we claim that there exists some $t_0>0$ such that for every $\epsilon>0$, $x \in \Bbb Z_{\geq 0}$, and $t \leq t_0$ we have that \begin{equation}\label{110}\nabla^{\pm} \mathbf p_t^R(x,x) \leq -\frac{4}{5}, \;\;\;\; \text{and }\;\;\; |\nabla^- \mathbf p_t^R(x,x+1)| \leq \frac{1}{5}.\end{equation} Indeed, the corresponding bound may be seen to be true for the standard heat kernel on the whole line $\Bbb Z$: for $t \leq t_0$ we have \begin{equation}\label{109}\nabla^{\pm} p_t(0) \leq -\frac{9}{10} , \;\;\;\;\text{and} \;\;\;\; |\nabla^- p_t(-1) | \leq \frac{1}{10}.\end{equation} In fact when $t=0$ the left quantity is simply $-1$ and the right quantity is $0$. Moreover, there is no dependence on $\epsilon$ and both quantities are continuous in $t$, which shows that (\ref{109}) is indeed true. Now we use the simple relation (see Equation (\ref{1})) $\mathbf p_t^R(x,y) = p_t(x-y)+p_t(x+y+1)$ in order to deduce Equation (\ref{110}) from (\ref{109}).
		\\
		\\
		Now, given that $(\ref{110})$ is true, this implies that for $t \leq t_0$, $x \in \Bbb Z_{\geq 0}$, and $\epsilon>0$: $$\bigg( |\nabla^+ \mathbf p_t^R(x,x+1) \nabla^- \mathbf p_t^R(x,x)| - |\nabla^- \mathbf p_t^R(x,x+1) \nabla^+ \mathbf p_t^R(x,x)| \bigg)^2 \geq \bigg(\frac{4}{5}\cdot \frac{4}{5} - \frac{1}{5} \cdot 1 \bigg)^2 > \frac{1}{6}.$$
		This in turn implies that the expression in (\ref{111}) is bounded below by $1/6$, uniformly over $t \in [0,t_0], x \in \Bbb Z_{\geq 0}$, and $\epsilon>0$, i.e., $$\bigg(\sum_{y \geq 1} (\nabla^+ \mathbf p_t^R(x,y))^2 \bigg) \bigg( \sum_{y \geq 1} (\nabla^-\mathbf p_t^R(x,y))^2 \bigg)-\bigg(\sum_{y \geq 1} \mathbf |K_t(x,y)| \bigg)^2 >1/6  $$ Now this expression is of the form $f^2g^2-h^2 >1/6$, where $f,g,h$ are functions of $(x, \epsilon)$, defined by the previous expression. We can rewrite this as $(fg-h)(fg+h)>1/6$. Now, Cauchy-Schwarz implies that $h \leq fg$, so that $fg+h \leq 2fg$. Moreover, the heat kernel estimates (Propositions \ref{2}, \ref{3}, and Corollary \ref{4}) imply that for $t \leq \epsilon^{-2}T$, we have that $f, g \leq C$ for some absolute constant $C$. Hence $(fg-h) \cdot 2C^2 \geq (fg-h)(fg+h) >1/6$, so that $fg-h>1/(12C^2) =:c>0$. 
		Summarizing, there exists $c>0$ such that for all $t \leq t_0$, $x \in \Bbb Z_{\geq 0}$, and $\epsilon>0$, we have that $$\bigg(\sum_{y \geq 1} (\nabla^+ \mathbf p_t^R(x,y))^2 \bigg)^{1/2} \bigg( \sum_{y \geq 1} (\nabla^-\mathbf p_t^R(x,y))^2 \bigg)^{1/2}-\sum_{y \geq 1} \mathbf |K_t(x,y)|\;\; >\;\;c . $$ Using the Cauchy-Schwarz inequality and Lemma \ref{108}, the time-integral of the first term on the LHS (from $t=0$ to $\infty$) is bounded above by $1$, therefore we integrate both sides of the above expression and we get $$1- \int_0^{\infty} \sum_{y \geq 1} |K_t(x,y)| dt > ct_0$$ which completes the proof when $a=0$, with $c_* = 1-ct_0$.
		\\
		\\
		To prove the general case with $a>0$, Proposition \ref{3} gives $|K_t(x,y)| \leq C(1 \wedge t^{-2}) e^{-b|x-y|(1 \wedge t^{-1/2})}$, so that $$\sup_{x \geq 0} \sum_{y \geq 1} |K_t(x,y)| (e^{a\epsilon|x-y|}-1) \leq C(1\wedge t^{-2}) \sum_{z \in \Bbb Z} e^{-b(1 \wedge t^{-1/2})|z|}(e^{a\epsilon |z|}-1).$$ By the dominated convergence theorem, this in turn implies that: $$\lim_{\epsilon \to 0} \int_0^{\infty} \sup_{x \geq 0} \sum_{y \geq 1} |K_t(x,y)|(e^{a \epsilon |x-y|}-1) dt = 0.$$
		Note that by Equation (\ref{1}), $\mathbf p_t^R$ does not actually depend on $\epsilon$ when $A=0$, therefore the preceding expression implies that $$\lim_{\epsilon \to 0}\; \sup_{x \geq 0} \int_0^{\infty} \sum_{y \geq 1} |K_t(x,y)| e^{a \epsilon |x-y|} dt = \sup_{x \geq 0}\int_0^{\infty} \sum_{y \geq 1} |K_t(x,y)| dt \leq c_*<1.$$ Hence for small enough $\epsilon>0$, the LHS is smaller than $(1+c_*)/2 =: c_*'$, completing the proof.
	\end{proof}
	
	\begin{cor}\label{11}
		In the same setting as Proposition \ref{10}, for any $S\in[0,T]$ there is a $C=C(A,T,S,a)$ such that for $s=\epsilon^{-2}S$ we have $$\sum_{y\geq 1} \int_0^s |\nabla^+ \textbf{p}^R_t(x,y) \nabla^- \textbf{p}^R_t(x,y)| e^{a \epsilon|x-y|}  (s-t)^{-1/2} dt \leq C\epsilon.$$
	\end{cor}
	
	\begin{proof} We mimic the proof given in [CS16, Proposition 5.4]. We split the integral into two pieces, one from $0$ to $s/2$ and he other from $s/2$ to $s$. For the first integral, we note that $(s-t)^{-1/2} \leq \sqrt{2} s^{-1/2} =\sqrt{2}S^{-1/2} \epsilon$ when $t<s/2$, and therefore 
	\begin{align*}&\;\;\;\;\sum_{y\geq 1} \int_0^{s/2} |\nabla^+ \textbf{p}^R_t(x,y) \nabla^- \textbf{p}^R_t(x,y)| e^{a \epsilon|x-y|}  (s-t)^{-1/2} dt \\ &\leq C(S) \; \epsilon  \sum_{y\geq 1} \int_0^{s/2} |\nabla^+ \textbf{p}^R_t(x,y) \nabla^- \textbf{p}^R_t(x,y)| e^{a \epsilon|x-y|} dt \\ &< C(S) \;\epsilon
	\end{align*}
	where we used Proposition \ref{10} in the final line. For the second part of the integral, we note from Proposition \ref{3} that $|\nabla^+ \textbf{p}^R_t(x,y) \nabla^- \textbf{p}^R_t(x,y)| \leq Ct^{-1}|\nabla^- \mathbf p_t^R(x,y)|$ and thus Corollary \ref{4} shows
	$$\sum_{y \geq 1} |\nabla^+ \textbf{p}^R_t(x,y) \nabla^- \textbf{p}^R_t(x,y)|e^{a\epsilon|x-y|} \leq C(a,A,T) (1 \wedge t^{-3/2}).$$ Integrating both sides from $s/2$ to $s$ we see that 
	\begin{align*}&\;\;\;\;\sum_{y \geq 1} \int_{s/2}^s |\nabla^+ \textbf{p}^R_t(x,y) \nabla^- \textbf{p}^R_t(x,y)|e^{a\epsilon|x-y|} (s-t)^{-1/2}dt \\ &\leq C(a,A,T)  \int_{s/2}^s t^{-3/2} (s-t)^{-1/2}dt \\ &=C(a,A,T)  s^{-1} \int_{1/2}^1 u^{-3/2} (1-u)^{-1/2} du \\&\leq C(a,A,S,T) \epsilon^2.
	\end{align*}
	We made a substitution $t=su$ in the third line, and we used $s^{-1} = \epsilon^2 S^{-1}$ in the final line. This proves the claim. \end{proof}

	Finally we conclude this section with an estimate which falls into none of the above categories:
	
	\begin{prop}\label{33} For every $t \geq s \geq 0$ and $x,y \geq 0$ we have that $$\mathbf p_s^R(x,y) \leq e^{t-s} \mathbf p_t^R(x,y).$$
	\end{prop}
	
	\begin{proof} One may directly verify that $$\mathbf p^R_s(x,y) = \sum_{n=0}^{\infty} e^{-s}\frac{s^n}{n!} \mathbf p^n(x,y)$$ where $\mathbf p^n(x,y)$ is the fundamental solution to the discrete-time, discrete-space equation $\mathbf p^{n+1}(x,y)-\mathbf p^n(x,y) = \frac{1}{2}\Delta \mathbf p^n(x,y)$, with Robin boundary conditions $\mathbf p^n(-1,y)=\mu_A \mathbf p^n(0,y)$. Therefore $$\mathbf p_s^R(x,y) = \sum_{n=0}^{\infty} e^{-s}\frac{s^n}{n!}  \mathbf p^n(x,y)\leq e^{t-s}\sum_{n=0}^{\infty} e^{-t}\frac{t^n}{n!} \mathbf p^n(x,y) = e^{t-s} \mathbf p_t^R(x,y)$$ which proves the claim. \end{proof}

	\subsection{Bounded Interval Estimates}
	
	In this section, we prove all of the estimates of Section 3.1 for the Robin heat kernel on the bounded interval $\{0,...,N\}$. We will fix $A,B \in \Bbb R$, and we set $\mu_A:=1-A\epsilon$ and $\mu_B=1-B\epsilon$ for $\epsilon:=\frac{1}{N}$. Then we let $\textbf{p}_t^R(x,y)$ be the Robin heat kernel satisfying $\textbf{p}_t^R(-1,y)=\mu_A\textbf{p}_t^R(0,y)$ and $\textbf{p}_t^R(N+1,y)=\mu_B \textbf{p}_t^R(N,y)$. We will need a formula similar to \eqref{1}, though unfortunately we have to rely on a more complicated inductive formula. 
	\\
	\\
	In order to derive such a formula, [CS16] used the following procedure: Start with an arbitrary ``test" function $\varphi: \{0,...,N\} \to \Bbb R$. Then extend $\varphi$ to a function $\tilde \varphi$ on all of $\Bbb Z$ such that $\tilde \varphi(z-1)-\mu_A \tilde \varphi(z)$ is an odd function, and such that $\tilde \varphi(N+1+z) - \mu_B \tilde \varphi(N+z)$ is also an odd function. This may be done inductively, first defining $\tilde \varphi$ on $\{-N,...,-1\}$ and on $\{N+1,...,2N\}$, then on $\{-2N,...,-N-1\}$ and $\{2N+1,...,3N\}$, and so forth. Then it would necessarily hold true that for any $x\in \{0,...,N\}$: $$\sum_{y=0}^N \mathbf p_t^R(x,y) \varphi(y) = \sum_{y \in \Bbb Z} p_t(x-y) \tilde \varphi(y)$$ where the $p_t$ on the RHS is the standard (whole-line) heat kernel. Indeed, both the LHS and RHS solve the heat equation with Robin boundary conditions and initial data $\varphi$. Using this fact and rearranging terms, [CS16, Lemmas 4.6 and 4.7] proved the following semi-explicit formula 
	\begin{equation}\label{12}
	\textbf{p}_t^R(x,y) = \sum_{\substack{k\in \Bbb Z}} I_kp_t(x-i(y,k)) + \epsilon \sum_{\substack{k\in \Bbb Z \\ k \neq 0}} \sum_{z=k(N+1)}^{k(N+1)+N} p_t(x-z) E_k(z,y)
	\end{equation}
	where $$i(y,k) = \begin{cases} 
	(k+1)(N+1)-y-1, & k \equiv 1 \pmod 2 \\
	y+k(N+1), & k \equiv 0 \pmod 2
	\end{cases}$$
	and the factors $I_k$ and $E_k(z,y)$ (which depend on $\mu_A$ and $\mu_B$) satisfy the following inductive relations for $m \geq 0$:
	\\
	\\
	$I_0=1$ and $E_0(x,y)=0$ for $x,y \in \{0,...,N\}$. Then
	\begin{align*} I_{-(m+1)} &= \mu_AI_m, \\
	I_{m+1} &= \mu_BI_{-m}. \end{align*}
	For $x \in \{(-m-1)(N+1),..., (-m-1)(N+1)+N\}$, and $y \in \{0,...,N\}$
	\begin{align*}E_{-(m+1)}(x,y) &:= \mu_AE_m(-x-1,y) + \epsilon^{-1} (\mu_A^2-1) \sum_{k=0}^m \mu_A^{-x-2-i(y,k)} I_{k} 1_{\{i(y,k) \leq -x-2 \}} \\ &\;\;+(\mu_A^2-1) \sum_{k=0}^m \sum_{z=k(N+1)}^{k(N+1)+N} \mu_A^{-x-2-z} 1_{\{z\leq -x-2\}} E_k(z,y).
	\end{align*}
	For $x \in \{(m+1)(N+1),..., (m+1)(N+1)+N\}$, and $y \in \{0,...,N\}$
	\begin{align*}E_{m+1}(x,y) &:= \mu_BE_{-m}(2N+1-x,y) + \epsilon^{-1} (\mu_B^2-1) \sum_{k=-m}^0 \mu_B^{x-2(N+1)+i(y,k)} I_{k} 1_{\{i(y,k) \geq 2(N+1) -x\}} \\ &\;\;+(\mu_B^2-1) \sum_{k=-m}^0 \sum_{z=k(N+1)}^{k(N+1)+N} \mu_B^{x-2(N+1)+z} 1_{\{z\geq 2(N+1)-x\}} E_{k}(z,y).
	\end{align*}
	We will repeatedly use the fact that $\epsilon=N^{-1}$ throughout this section, so we strongly emphasize this relation.
	\begin{lem}\label{13} There exists a constant $C_0:=C_0(A,B)$ such that for all $k\in \Bbb Z$, 
		$$\sup_{\substack{z\in \{k(N+1),...,k(N+1)+N\} \\ y \in \{0,...,N\}}} |E_k(z,y)| \leq C_0^{|k|}.$$
	\end{lem}
	
	\begin{proof} Assuming $C_0$ has already been fixed, we we prove the claim using induction on $k$ (and we will find an explicit lower bound for $C_0$ later). It is clearly true for $k=0$, since $E_0=0$. So suppose the result is true for all $k$ such that $|k| \leq m$. Using the recursive relation for $E_{-(m+1)}$ as well as the fact that $|\mu_A^2-1| \leq 3A\epsilon$ for small enough $\epsilon$, we find that $$|E_{-m-1}(x,y)| \leq \mu_A C_0^m + 3A(m+1) \mu_A^{(m+1)N} + 3A\epsilon \bigg(\sum_{k=0}^m N \mu_A^{(m+1-k)N} C_0^k \bigg).$$ 
	Now notice that $\mu_A^{(m+1)N} = (1-A/N)^{N(m+1)} \leq e^{|A|(m+1)}$. Similarly, and using $\epsilon=N^{-1}$ we get
	$$3A\epsilon \bigg(\sum_{k=0}^m N \mu_A^{(m+1-k)N} C_0^k\bigg) \leq 3Ae^{|A|} \sum_{k=0}^{m} e^{|A|(m-k)} C_0^k = 3Ae^{|A|} \frac{C_0^{m+1} - e^{|A|(m+1)}}{C_0-e^{|A|}}.$$
	Therefore if $\epsilon$ is small enough so that $\mu_A<2$ then
	\begin{align*}|E_{-(m+1)}(x,y)| &\leq 2 C_0^m+3A(m+1)e^{|A|(m+1)} +3Ae^{|A|} \frac{C_0^{m+1} - e^{|A|(m+1)}}{C_0-e^{|A|}} \\ &\leq K(A) \;\bigg(e^{2|A|m} + \frac{C_0^{m+1}}{C_0-e^{|A|}}\bigg)
	\end{align*}
	where $K(A)$ is a constant depending only on $A$ and not $m$. We claim that for large enough choices of $C_0$ the last expression can be made smaller than $C_0^{m+1}$. Indeed, take $C_0 \geq e^{2|A|}$ such that $2K(A)/(C_0-e^{|A|}) \leq 1$. Then 
	\begin{align*}
	C_0^{-(m+1)}|E_{-(m+1)}(x,y)| &\leq K(A) \bigg[ \frac{1}{C_0} \bigg( \frac{e^{2|A|}}{C_0} \bigg)^m + \frac{1}{C_0-e^{|A|}} \bigg] \\ & \leq K(A) \bigg[ \;\;\frac{1}{C_0} \;\;\;+\;\;\; \frac{1}{C_0-e^{|A|}} \bigg] \\ &\leq 1.
	\end{align*}
	Thus the inductive hypothesis is satisfied for $k=-(m+1)$. The case when $k=m+1$ is similar, but one needs to replace $A$ with $B$ throughout the proof. \end{proof}
	
	As a consequence of these explicit formulas, we have the same spatial bounds as those of the half-line.
	
	\begin{prop}\label{14}
		Fix $A,B \in \Bbb R$ and $T>0$. For $b \geq 0$, there is a constant $C(A,B,b,T)$ (not depending on $\epsilon=N^{-1}$) such that for all $t \in [0,\epsilon^{-2}T]$, and all $x,y \in \mathbb \{0,...,N\}$ we have that $$\mathbf{p}_t^{R}(x,y) \leq C(A,B,b,T) (1 \wedge t^{-1/2}) e^{-b|x-y| (1 \wedge t^{-1/2})}.$$
	\end{prop}
	
	\begin{proof} We will prove the desired bound for both terms in the RHS of equation (2). For the first term, we only need to consider terms in the sum with $|k|>2$ since the other terms easily satisfy the desired bound. Note by induction that $I_k \leq (\mu_A \vee \mu_B)^{|k|}$ for all $k \in \Bbb Z$. Using the same argument as in Proposition \ref{2}, we have $$\mu_A \vee \mu_B \leq e^{C(A,B) \epsilon} \leq e^{C(A,B,T) (1 \wedge t^{-1/2})}$$ for all $t \in[0,\epsilon^{-2}T]$. By using the standard estimates [DT16] as well as the fact that $|x-i(y;k)| \geq (|k|-1)N$ for all $x,y \in \{0,...,N\}$ we find that if $bN>C(A,B,T)$ then
	\begin{align*}\sum_{|k|>2} I_k p_t(x-i(y;k)) &\leq \sum_{|k|>2} e^{C(A,B,T) (1 \wedge t^{-1/2})|k|} \cdot C(b) (1 \wedge t^{-1/2}) e^{-b (|k|-1) N (1 \wedge t^{-1/2})} \\ & = C(b) (1 \wedge t^{-1/2}) \frac{e^{2(C(A,B,T)-bN)(1 \wedge t^{-1/2})}}{1-e^{(C(A,B,T)-bN)(1 \wedge t^{-1/2})}} \\ & \leq C(b)\frac{e^{(C(A,B,T)-bN)(1\wedge t^{-1/2})} }{bN-C(A,B,T)} 
	\end{align*}
	where we used $e^{-2q}/(1-e^{-q}) \leq e^{-q}/q$ in the last line. Now, we may as well assume that $b$ (or $N$) is large enough so that $bN-C(A,B,T)>bN/2$. Using this together with the fact that $N\geq |x-y|$ we get that the last expression is bounded above by $$C(b)N^{-1}e^{-\frac{1}{2}b|x-y|(1 \wedge t^{-1/2})},$$ then using the fact that $N^{-1} = \epsilon \leq C(T) (1\wedge t^{-1/2})$ gives the desired bound on the second term in equation (2), after replacing $\frac{1}{2}b$ with $b$.
	
	\textbf{}\\
	Now for the second term in (\ref{12}). As before we only consider terms in the sum with $|k|>2$. Note by Lemma \ref{13} that
	\begin{align*}
	\epsilon \sum_{|k| >2} \sum_{z=k(N+1)}^{k(N+1)+N} p_t(x-z) E_k(z,y) &\leq C(b) (1 \wedge t^{-1/2}) \sum_{|k|>2} e^{-b(|k|-1)N (1 \wedge t^{-1/2})} C_0^{|k|} \\ & = C(b) (1 \wedge t^{-1/2}) e^{-bN(1\wedge t^{-1/2})}\sum_{k \in \Bbb Z} e^{\big(\log C_0 - bN(1\wedge t^{-1/2})\big)|k|}
	\end{align*}
	
	where we used $\epsilon=N^{-1}$ and $|x-z| \geq (|k|-1)N$ in the first inequality. Noting that $N(1 \wedge t^{-1/2}) > T^{-1/2}$ we may assume that $b>2T^{1/2}\log C_0$ so that $\log C_0 - bN(1\wedge t^{-1/2}) < -\log C_0$ and thus the last expression is bounded by $$C(b)(1 \wedge t^{-1/2}) e^{-bN(1\wedge t^{-1/2})} \frac{1}{1-C_0^{-1}} \leq C(A,B,b,T) (1 \wedge t^{-1/2}) e^{-b|x-y|(1\wedge t^{-1/2})} $$
	where we used $|x-y| \leq N$ and assumed wlog that $C_0>1$. This proves the claim. \end{proof}
	\textbf{}\\
	\begin{prop}\label{15}
		Fix $A,B \in \Bbb R$ and $T>0$. For $b \geq 0$, there is a constant $C(A,B,b,T)$ (not depending on $\epsilon=N^{-1}$) such that for all $t \in [0,\epsilon^{-2}T]$, all $v\in[0,1]$, all $|n| \leq \lceil t^{1/2} \rceil$ and all $x,y \in \mathbb \{0,...,N\}$ we have that $$|\mathbf{p}_t^{R}(x+n,y)-\mathbf{p}_t^{R}(x,y)| \leq C(A,B,b,T) (1 \wedge t^{-1/2-v}) |n|^v e^{-b|x-y| (1 \wedge t^{-1/2})}.$$ When $b=0$ this bound holds for all $n \in \Bbb Z$, not just for $n\leq t^{1/2}$.
	\end{prop}
	
	\begin{proof} Again we use equation (2), replacing $p_t$ with $\nabla_np_t$ and noting that the corresponding bound already holds for the whole-line kernel [DT16, (A.13)]. Then we proceed exactly as in Proposition \ref{14}. The fact that the bound holds for all $n$ when $b=0$ is a consequence of the triangle inequality applied to the case when $n=1$. \end{proof}
	
	\begin{cor}\label{16}
		Fix $A,B \in \Bbb R$. For any $T \geq 0$ and $a \geq 0$, there exists some constant $C=C(a,A,B,T)$ such that for all $x \in \{0,...,N\}$ and $t \leq \epsilon^{-2}T$ we have \begin{align*}\sum_{y \geq 0} \mathbf p_t^R(x,y) e^{a|x-y|(1 \wedge t^{-1/2})} &\leq C , \\ \sum_{y \geq 0} |\nabla^+ \mathbf p_t^R(x,y)| e^{a|x-y|(1 \wedge t^{-1/2})} &\leq C(1 \wedge t^{-1/2}) .\end{align*}
	\end{cor}
	
	\begin{proof} One may mimic the proof of Corollary \ref{4} with $a_2=0$, but use Propositions \ref{14} and \ref{15} instead of \ref{2} and \ref{3}. \end{proof}

	We will now prove the temporal estimate (analog of Proposition \ref{6} on bounded intervals). For this, the above spatial methods do not work well, so we need information on the spectrum of the Laplacian with Robin boundary conditions.
	\textbf{}\\
	\begin{lem}\label{17}
		Consider the operator $\frac{1}{2}\Delta$ on $\Bbb R^{\{0,...,N\}}$ with Robin boundary conditions $f(-1):=\mu_Af(0)$ and $f(N+1):=\mu_Bf(N)$. This is a symmetric operator (and hence orthonormally diagonalizable with real eigenvalues). Consider $N$ large. For $1 \leq k \leq N-1$, there is an eigenvalue of the form $\lambda_k = 1 - \cos \omega_k$ with $\omega_k \in [k\pi/(N+1), (k+1)\pi/(N+1)]$. In particular there are at most two positive eigenvalues $\lambda_N, \lambda_{N+1}$, and moreover these satisfy the bound $\lambda_{N+1} \vee \lambda_N \leq C(A,B) \epsilon^2$.
	\end{lem}
	
	\begin{proof} The fact that $\frac{1}{2}\Delta$ is symmetric follows from the matrix representation using the standard basis $\{e_x\}_{0 \leq x \leq N}$ where $e_x(y) = 1_{[x=y]}$:
	
	$$\Delta = \begin{bmatrix}
	\mu_A-2 & 1 & 0 & \dots & 0 & 0 \\
	1 & -2 & 1 & \dots  & 0 & 0\\
	0 & 1 & -2 & \dots & 0 & 0 \\
	\vdots & \vdots & \vdots & \ddots & \vdots & \vdots\\
	0 & 0 & 0 & \dots & -2 & 1 \\
	0 & 0 & 0 & \dots  & 1 & \mu_B-2
	\end{bmatrix}.$$
	
	If $A=B=0$ (which means that $(\mu_A,\mu_B) = (1,1)$) then one may check directly that $N+1$ independent eigenvectors are given by $\cos(\omega_k x)$ where $\omega_k = k\pi/(N+1)$ for $k \in \{0,...,N\}$. In this case the eigenvalues are precisely $\lambda_k=1-\cos \omega_k$.
	
	\textbf{}\\
	If $(\mu_A,\mu_B) \neq (1,1)$ then by [CS16, (4.27)], all negative eigenvalues of this operator are of the form $\lambda = 1-\cos \omega$, where $\omega \in (0,\pi)$ solves the following equation $$f(\omega):=\sin(\omega(N+2)) - (\mu_A+\mu_B)\sin(\omega(N+1)) +\mu_A \mu_B \sin(\omega N) = 0.$$
	
	Letting $\omega=k\pi/(N+1)$, we see that \begin{align*}f \bigg( \frac{k\pi}{N+1}\bigg) & = \sin\bigg( k\pi + \frac{k\pi}{N+1}\bigg) +\mu_A\mu_B \sin \bigg( k\pi - \frac{k\pi}{N+1}\bigg) \\ &= (1-\mu_A\mu_B) \sin\bigg( k\pi + \frac{k\pi}{N+1}\bigg) \\ &= (-1)^k (1-\mu_A\mu_B) \sin \bigg(\frac{k\pi}{N+1}\bigg) \\ &= (-1)^k \bigg( \frac{A+B}{N} - \frac{AB}{N^2}\bigg) \sin \bigg(\frac{k\pi}{N+1}\bigg).
	\end{align*}
	Now notice that unless $A=B=0$, the middle term in this last expression cannot vanish for arbitrarily large $N$. We already ruled out this case, thus when $N$ is large enough this last expression will always alternate sign as a function of $k \in \{1,...,N-1\}$. By the intermediate value theorem applied to $f$, we find $N-1$ solutions to the eigenvalue equation above, one in each of the intervals $\big(k\pi/(N+1), (k+1)\pi/(N+1)\big)$ for $k\in\{1,...,N-1\}$ (this does not work for $k\in \{0,N\}$ since $\omega=0,\pi$ are solutions which do not correspond to nontrivial eigenvectors). Thus the first claim is proved.
	\\
	\\
	Now let us consider the (at most two) positive eigenvalues. By the basic methods of recursive sequences, there must be an associated eigenfunction of the form $$\psi(x) = \mu^{-x}+c\mu^{x-N}$$ for some $\mu>0$ and $c \in \Bbb R$ which satisfy the relations $$\mu_A = \frac{\mu+c\mu^{-(N+1)}}{1+c\mu^{-N}}=:h_1(c,\mu) \;\;\;\;\;,\;\;\;\;\; \mu_B = \frac{\mu^{-(N+1)} +c\mu}{\mu^{-N}+c} =: h_2(c,\mu).$$
	Note that the functions $h_1,h_2$ just defined satisfy the relation $h_2(c,\mu) = h_1(c,\mu^{-1}) = h_1(c^{-1},\mu)$ which means that we may assume $\mu>1$ and $|c|\leq 1$, after possibly interchanging the roles of $\mu_A$ and $\mu_B$ a couple of times (this does not change the eigenvalues or eigenfunctions).
	\\
	\\
	If the above relations are satisfied with $\mu>1$ and $c\in [-1,0]$ then $\mu<h_1(c,\mu)$ and thus $1<\mu<\mu_A$ so that the eigenvalue associated with $\mu$ satisfies $$\lambda = \mu+\mu^{-1}-2 \leq \mu_A+\mu_A^{-1}-2 \leq C(A) \epsilon^2.$$ In the last inequality we used the fact that $\mu_A=1-A\epsilon$ so that $\mu_A^{-1} = 1+A\epsilon+A^2\epsilon^2+O(\epsilon^3)$.
	\\
	\\
	On the other hand, if the above relations are satisfied with $\mu>1$ and $c \in (0,1]$, then $\mu_A\mu_B>1$ (because $h_1(c,\mu)h_2(c,\mu)>1$ for $c>0$ by direct computation) and also $\mu_A<\mu$ (since $h_1(c,\mu)<\mu$ for $c>0$). The fact that $\mu_A\mu_B>1$ implies that $A<0$ or $B<0$, let's say $A<0$. Note that
	$$|\mu_A-1| = \frac{|\mu-1||1-c\mu^{-(N+1)}|}{1+c\mu^{-N}}.$$ Since $\mu^{-1}<\mu_A^{-1}$ and $c \in [0,1]$, it follows that $1-c\mu^{-(N+1)} \geq 1-\mu_A^{-(N+1)}$. But $\mu_A^{-(N+1)} = (1-A/N)^{-(N+1)} \leq e^{-|A|}$ for large enough $N$. Using these bounds together with the last expression shows $$|\mu-1| = \frac{1+c\mu^{-N}}{1-c\mu^{-(N+1)}}|\mu-1| \leq \frac{2}{1-e^{-|A|}} |\mu_A-1|=C(A)|\mu_A-1|$$ and thus $\mu < 1+ C(A)|\mu_A-1| = 1+C(A)\epsilon$, so that the associated eigenvalue satisfies $$\lambda = \mu+\mu^{-1}-2 \leq C(A)\epsilon^2. $$
	This completes the proof. \end{proof} 
	\begin{rk}\label{18}
		It is worth mentioning that positive eigenvalues will exist if and only if $A+B+AB<0$, though we do not need this stronger claim. Similarly, zero will be an eigenvalue iff $A+B+AB=0$. The key idea in proving these statements is to note that whenever $0$ is an eigenvalue, the corresponding eigenfunction must be of the form $c+dx$ for some $c,d \in \Bbb R$, and then checking the boundary conditions necessarily forces $A+B+AB=0$.
	\end{rk}
	\begin{lem}\label{19}
		In the same setting as Lemma B.4, let $\psi_k$ denote the $L^2$-normalized eigenfunction associated to the eigenvalue $\lambda_k$. There exists some $C=C(A,B)$ such that for large enough $N$ we have that $|\psi_k(x)| \leq CN^{-1/2}$ for all $0 \leq x,k \leq N$.
	\end{lem}
	
	\begin{proof} If $\lambda_k<0$ then this is proved in [CS16, Lemma 4.10]. 
	
	\textbf{}\\
	If $\lambda_k>0$, then we can write the non-normalized eigenfunction as $$\psi_k(x)=\mu^{-x}+c\mu^{x-N}$$ where $\mu>1$ and $|c|\leq 1$ as in the proof of Lemma \ref{17}. Then \begin{align*}
	\sum_{k=0}^N \psi_k(x)^2 &= (1+c^2) \frac{\mu^{2(N+1)}-1}{\mu-1} + 2c(N+1) \\ &\leq 4 (N+1)\mu^{2(N+1)} +2(N+1)
	\end{align*}
	where we used the fact that $(x^k-1)/(x-1) \leq kx^k$ for $x>1$. From the proof of Lemma \ref{17}, we know that $\mu \leq 1+C(A,B)\epsilon$, hence we see that $\mu^{2(N+1)} \leq e^{2C(A,B)} = C'(A,B)$. Hence the expression above is bounded by $(4C(A,B)+2)(N+1) \leq C'(A,B)N$. So by renormalizing, we find that $\psi_k(x) \leq C(A,B)N^{-1/2}$.
	
	\textbf{}\\
	If $\lambda_k=0$ then the associated eigenfunction is $$\psi_k(x) = c+dx.$$
	This can happen for arbitrarily large $N$ iff $A+B+AB=0$ and $d/c = A/N$ (as one may check by writing out the boundary conditions, for instance $1-d/c = \psi(-1)/\psi(0) = \mu_A = 1-A/N$, etc). The condition for the eigenfunction to be normalized simplifies to $$c^2(N+1)+cdN(N+1)+d^2N(N+1)(2N+1)/6 = 1.$$ Then by writing $d=Ac/N$ or $c=dN/A$ we can check that $c^2 \leq C_1(A) N^{-1}$ and $d^2 \leq C_2(A) N^{-3}$. Consequently $\psi_k(x) =c+dx \leq C(A) ( N^{-1/2}+(N^{-3/2})N ) =  C(A) N^{-1/2}$ for $x\in \{0,...,N\}$. This proves the claim. \end{proof}
	\begin{prop}\label{20}
		Fix $A,B \in \Bbb R$ and $T>0$. There is a constant $C(A,B,T)$ (not depending on $\epsilon=N^{-1}$) such that for all $s< t \in [0,\epsilon^{-2}T]$ and all $v\in [0,1]$ we have $$|\mathbf{p}_t^R(x,y)-\mathbf{p}_s^R(x,y)| \leq C(A,B,T) (1 \wedge s^{-1/2-v}) (t-s)^v.$$
	\end{prop}
	
	\begin{proof} Just like Proposition \ref{6}, we only need to prove this when $v=0$ and $v=1$. The $v=0$ case follows from Proposition \ref{14}. So let us prove the $v=1$ case.
	
	\textbf{}\\
	Let $S(t):=e^{\frac{1}{2}t \Delta}$ and let $\{e_x\}$ denote the standard basis functions on $\Bbb R^{\{0,...,N\}}$. Let $\psi_k$ $(1 \leq k \leq N+1)$ denote the orthonormal eigenunctions of $\frac{1}{2}\Delta$. Then $e_x = \sum_k \langle e_x, \psi_k \rangle \psi_k = \sum_k \psi_k(x) \psi_k$. So letting $\langle \cdot, \cdot \rangle$ denote the inner product on $L^2(\{0,...,N\})$ we find $$\mathbf{p}_t^R(x,y) = \langle S(t) e_x, e_y \rangle = \sum_{k, \ell} \psi_k(x)\psi_{\ell}(y) \langle S(t)\psi_k, \psi_{\ell} \rangle= \sum_k \psi_k(x)\psi_k(y) e^{\lambda_kt} $$ which implies that $$ \mathbf{p}_t^R(x,y)-\mathbf{p}_s^R(x,y) =\sum_k \psi_k(x)\psi_k(y) e^{\lambda_ks}(e^{\lambda_k(t-s)}-1). $$ We will split this last sum into two pieces based on the sign of the eigenvalues. Let us first consider negative eigenvalues. Using Lemmas \ref{17} and \ref{19}, together with the following $$|e^{-q}-1| \leq q, \;\;\;q\geq 0$$ $$c_1u^2 \leq |1-\cos u| \leq c_2u^2,\;\;\; u\in [0,\pi]$$ we obtain the following bound \begin{align*}\sum_{k:\lambda_k<0} \psi_k(x)\psi_k(y) e^{\lambda_ks}(e^{\lambda_k(t-s)}-1) &\leq C(A,B) \sum_{k:\lambda_k<0} N^{-1} e^{\lambda_ks} |\lambda_k| (t-s)\\ &\leq \frac{C(A,B)}{N}(t-s) \sum_{k: \lambda_k<0} e^{\lambda_ks}\bigg| 1-\cos \bigg( \frac{k\pi}{N+1} \bigg) \bigg| \\ &\leq \frac{C(A,B)}{N}(t-s) \sum_{k=0}^N e^{-c_1sk^2/(N+1)^2} \frac{c_2 k^2}{(N+1)^2}.
	\end{align*}
	This last expression can be interpreted as a Riemann sum for the integral \begin{align}\label{93}C(A,B) (t-s) \int_0^1 x^2 e^{-c_1sx^2}dx &\leq C(A,B) (t-s) \int_0^{\infty} x^2 e^{-c_1sx^2}dx \nonumber \\ &= C(A,B) (t-s) s^{-3/2} \int_0^{\infty} u^2e^{-c_1u^2}du \nonumber\\ &=C(A,B)s^{-3/2}(t-s)\end{align} where we made the substitution $x\sqrt{s}=u$ in the second line. So when $N$ is large, we are close enough to this integral that the same bound holds.
	
	\textbf{}\\
	Next we consider the terms with positive eigenvalues. By Lemmas \ref{17} and \ref{19}, and the fact that $e^q-1 \leq qe^q$ for $q \geq 0$ we see that 
	\begin{align}\label{94}
	\sum_{k:\lambda_k>0} \psi_k(x)\psi_k(y) e^{\lambda_ks}(e^{\lambda_k(t-s)}-1) & \leq 2 \cdot \frac{C}{N}\cdot e^{C(A,B)\epsilon^2s} \cdot \epsilon^2(t-s) e^{C(A,B) \epsilon^2(t-s)}  \nonumber \\ &=2C \epsilon^3 (t-s) e^{C(A,B)\epsilon^2t} \\ &\leq C(A,B,T) s^{-3/2} (t-s) \nonumber
	\end{align}
	where in the last inequality we used the fact that $s<t<\epsilon^{-2}T$ so that $\epsilon^2t \leq T$ and $\epsilon^3 \leq C(T)s^{-3/2}$. This proves the claim. \end{proof}
	
	\begin{prop}[Long-time Estimate]\label{lt2}
		There exist constants $C=C(A,B)$ and $K=K(A,B)$ such that for every $t\geq 0$ and $x,y \geq 0$ we have that $$\mathbf p_t^R(x,y)\leq C(t^{-1/2}+\epsilon)e^{K\epsilon^2t} .$$
	\end{prop}
	
	Just like Proposition \ref{lt1} this is a ``long-time" estimate because it is true uniformly over all $t>0$, i.e., constants don't depend on any terminal time $\epsilon^{-2}T$.
	
	\begin{proof} In equations (\ref{93}) and (\ref{94}) above, note that the constants do not depend on the terminal time $T$. This means that there are $C$ and $K$ such that for any $s<t$ we have that $$|\mathbf p_t^R(x,y) - \mathbf p_s^R(x,y)| \leq C \epsilon^3 (t-s)e^{K\epsilon^2t} + Cs^{-3/2} (t-s). $$ Dividing by $t-s$ and letting $s \to t$, we find $$|\partial_t \mathbf p^R_t(x,y)| \leq C\epsilon^{3}e^{K\epsilon^2t}+Ct^{-3/2}$$By Proposition \ref{14}, the desired bound already holds when $t\leq \epsilon^{-2}$, thus we only consider the case when $t>\epsilon^{-2}$. Using the above expression and Proposition \ref{14}, \begin{align*}\mathbf p_t^R(x,y) &\leq \mathbf p_{\epsilon^{-2}}^R(x,y) + \int_{\epsilon^{-2}}^t |\partial_s \mathbf p_s^R(x,y)|ds \\ &\leq C(\epsilon^{-2})^{-1/2} + C\int_{\epsilon^{-2}}^t(\epsilon^3e^{K\epsilon^2 s} +Cs^{-3/2})ds \\ &= C\epsilon + \frac{C}{K} \epsilon (e^{K\epsilon^2 t}-e^K) + 2C(\epsilon-t^{-1/2}) \\ &\leq  C'\epsilon e^{K\epsilon^2 t}\end{align*} which proves the claim.
	\end{proof}
	
	\begin{prop}\label{21}
		For all $A,B \in \Bbb R$ and $T>0$, there exists $C=C(A,B,T)$ such that for $t\in[0,\epsilon^{-2}T]$ and $v\in [0,1]$ we have $$\sup_{ 0 \leq x \leq N}\bigg| \sum_{y = 0}^N \mathbf{p}^R_t(x,y) \;- \;1 \bigg| \leq C \epsilon^{v}   t^{v/2}.$$
	\end{prop}
	
	\begin{proof} The proof is basically the same as Proposition \ref{7}. We define $$f(t,x):= \sum_{y = 0}^N \textbf{p}^R_t(x,y).$$ The fact that $\mathbf{p}_t^R$ is symmetric in $x$ and $y$ follows easily by symmetry of the operator $\frac{1}{2}\Delta$ with Robin boundary conditions. The same arguments used in Lemma \ref{7} then show that 
	\begin{align*}\partial_t f(t,x) &= \frac{1}{2} (\mu_A-1)\textbf{p}^R_t(0,x)+ \frac{1}{2}(\mu_B-1)\textbf{p}^R_t(N,x) \end{align*}
	so that $|\partial_tf(t,x)| \leq C(A,B,T) \;\epsilon \; t^{-1/2}$ by Proposition \ref{14}. Now we proceed exactly as in Lemma \ref{7}. \end{proof}
	
	\textbf{}\\
	We now turn to proving the cancellation estimates (the analogs of Propositions \ref{10} and \ref{11} for bounded-interval Robin heat kernels). For this we would like a result like Lemma \ref{8}, but we are not able to prove the corresponding result on bounded intervals, so we opt for a weaker result which will suffice for us (Lemma \ref{23}).
	\textbf{}\\
	\begin{lem}\label{22}
		Let $M$ and $N$ be symmetric $n \times n$ matrices (or more generally, self-adjoint operators on some Hilbert space) whose set of eigenvalues (or spectral values) is bounded above by $\alpha \in \Bbb R$. Then $$\|e^{M}-e^{N}\| \leq e^{\alpha} \|M-N\|$$ where $e^{M}$ denotes the matrix exponential, and $\| \cdot \|$ is the operator norm with respect to the underlying Hilbert space norm.
	\end{lem}
	
	\begin{proof} When $M$ and $N$ commute the proof is easy by simultaneous diagonalization. For the general case, it is tempting to use $\|e^{M}-e^{N}\| \leq \|M-N\|e^{\max\{\|M\|,\|N\|\}}$, however this crude bound fails since it only takes into account the magnitude of the eigenvalues and not their sign (note that $\alpha$ might be negative). Instead we use Frechet calculus. If $\cal B$ is a Banach space, $f:\cal B \to B$ is Frechet differentiable, and $\gamma:[0,1] \to \cal B$ is a smooth curve, $$f(\gamma(1))-f(\gamma(0)) = \int_0^1 Df(\gamma(t)) \gamma'(t)dt$$ where $Df(x) \in L(\cal B,B)$ is the Frechet derivative. As an immediate corollary, we have for all $a,b \in \cal B$ $$\|f(a)-f(b)\| \leq \bigg( \sup_{x \in [a,b]} \|Df(x)\| \bigg) \|a-b\|$$ where $[a,b]:= \{(1-t)a+tb:t\in[0,1]\}$. We now specialize this bound to the case when $f(X)=e^X$ with $\mathcal B = L(\cal H,H)$ for a Hilbert space $\cal H$. In this case, there is a well-known formula for the Frechet derivative $$Df(X)H = \int_0^1 e^{sX}He^{(1-s)X}ds$$ which immediately implies that the operator norm of $Df(X)$ satisfies $$\|Df(X)\| \leq \int_0^1 \|e^{sX}\| \cdot \|e^{(1-s)X}\|ds.$$ 
	Now suppose that $M,N$ are self-adjoint operators on $\cal H$ which satisfy $\langle Mu,u \rangle \leq \alpha\|u\|^2$ and $\langle Nu,u \rangle \leq \alpha\|u\|^2$ (i.e., the largest eigenvalue of $M$ and of $N$ is bounded above by $\alpha$, which can be negative). In this case we easily have that $\langle Xu,u \rangle \leq \alpha\|u\|^2$ for all $X$ in the interval $[M,N]$ (as defined above). Consequently $\|e^{sX}\| \leq e^{\alpha s}$ for all $s$. Thus if $X\in [M,N]$ $$\|Df(X)\| \leq \int_0^1 e^{s\alpha}e^{(1-s)\alpha}ds = e^{\alpha}$$ so that $\|f(M)-f(N)\| \leq e^{\alpha} \|M-N\|$ as desired. \end{proof}
	
	\begin{lem}\label{23}
		For the next few estimates we will distinguish between different values of $(A,B)$ by writing $\mathbf{p}_t^R(x,y;A,B)$ for the ($\epsilon$-dependent) Robin heat kernel of parameters $A,B$. For all $A,B \in \Bbb R$, all $T>0$, and all $b \geq 0$, there exists $C=C(A,B,b,T)$ such that for all $x,y \in \Bbb Z_{\geq 0}$, all $t \in [0,\epsilon^{-2}T]$, and all $v\in [0,1]$ we have
		\begin{align*}
		|\mathbf{p}^R_t(x,y;A,B)-\mathbf{p}^R_t(x,y;0,0)| &\leq C\; \epsilon^v \;(t^v \wedge t^{(3v-1)/2} )\; e^{-(1-v)b|x-y|(1 \wedge t^{-1/2})}, \\
		|\nabla^{\pm}\mathbf{p}^R_t(x,y;A,B)-\nabla ^{\pm} \mathbf{p}^R_t(x,y;0,0)| &\leq C \; \epsilon^v \; (t^v \wedge t^{2v-1}) e^{-(1-v)b|x-y|(1 \wedge t^{-1/2})},	\end{align*}
		where $\nabla^{\pm}$ denotes the discrete gradient in the first spatial coordinate.
	\end{lem}
	
	\begin{proof} It suffices to prove the claim $v=0$ and $v=1$. The middle cases then follow from interpolation. The $v=0$ case follows easily from Propositions \ref{14} and \ref{15}, thus we only consider the $v=1$ case. So we will show that 
	\begin{equation*}
	|\textbf{p}^R_t(x,y;A,B)-\textbf{p}^R_t(x,y;0,0)| \leq C\; \epsilon \;t,
	\end{equation*}
	\begin{equation*}
	|\nabla^{\pm}\textbf{p}^R_t(x,y;A,B)-\nabla ^{\pm} \textbf{p}^R_t(x,y;0,0)| \leq C \; \epsilon \; t,
	\end{equation*}
	for some constant $C=C(A,B,T)$ (we postulate that these bounds are not optimal, but they suffice to prove Proposition \ref{24} below). To prove these bounds, we will use Lemma \ref{22} with the two matrices $\frac{t}{2}\Delta_{A,B}$ and $\frac{t}{2}\Delta_{0,0}$. Here $\Delta_{A,B}$ denotes the Laplacian on $\{0,...,N\}$ with Robin boundary parameters $\mu_A, \mu_B$. Then Lemma \ref{22} and Lemma \ref{17} show that $$\| e^{\frac{1}{2}\Delta_{A,B}t} - e^{\frac{1}{2} \Delta_{0,0}t}  \| \leq e^{C(A,B)\epsilon^{2}t}\cdot  \frac{t}{2}\|\Delta_{A,B} - \Delta_{0,0}\|.$$
	But $t \leq \epsilon^{-2}T$ so that $\epsilon^2t \leq T$ and thus $e^{C(A,B)\epsilon^{2}t} \leq C(A,B,T)$. Now notice that $$\Delta_{A,B} - \Delta_{0,0} = \begin{bmatrix}
	\mu_A-1 & 0 & \dots & 0 & 0 \\
	0 & 0  & \dots  & 0 & 0\\
	\vdots & \vdots & \ddots & \vdots & \vdots\\
	0 & 0 & \dots & 0 & 0\\
	0 & 0 & \dots  & 0 & \mu_B-1
	\end{bmatrix}.$$
	This is a diagonal matrix with eigenvalues $\{-A\epsilon, 0,...,0,-B\epsilon\}$ and thus we see that $\|\Delta_{A,B} - \Delta_{0,0}\| \leq C(A,B) \epsilon$. Summarizing, we have shown that $$\| e^{\frac{1}{2}\Delta_{A,B}t} - e^{\frac{1}{2} \Delta_{0,0}t}  \| \leq C(A,B,T) \; \epsilon \; t.$$ Therefore
	\begin{align*}
	|\textbf{p}^R_t(x,y;A,B)-\textbf{p}^R_t(x,y;0,0)| &= \bigg|\big\langle  (e^{\frac{1}{2}\Delta_{A,B}t} - e^{\frac{1}{2} \Delta_{0,0}t}) \mathbf 1_x, \mathbf 1_y \big\rangle \bigg| \\ & \leq \|e^{\frac{1}{2}\Delta_{A,B}t} - e^{\frac{1}{2} \Delta_{0,0}t} \| \\ &\leq C(A,B,T) \epsilon t.
	\end{align*}
	Similarly,
	\begin{align*}
	|\nabla^{\pm}\textbf{p}^R_t(x,y;A,B)-\nabla^{\pm}\textbf{p}^R_t(x,y;0,0)| &= \bigg|\big\langle  (e^{\frac{1}{2}\Delta_{A,B}t} - e^{\frac{1}{2} \Delta_{0,0}t}) \mathbf (\mathbf 1_x-\mathbf 1_{x \pm 1}), \mathbf 1_y \big\rangle \bigg| \\ & \leq 2\|e^{\frac{1}{2}\Delta_{A,B}t} - e^{\frac{1}{2} \Delta_{0,0}t} \| \\ &\leq C(A,B,T) \epsilon t
	\end{align*}
	which proves the claim. \end{proof}
	\begin{lem}\label{24} For $A,B\in \Bbb R$, $t \geq 0$, and $x,y \in \Bbb Z_{\geq 0}$ we define $$K_t(x,y;A,B):=\nabla^+ \mathbf{p}^R_t(x,y;A,B) \nabla^- \mathbf{p}^R_t(x,y;A,B).$$ For $T \geq 0$ there exists a constant $C=C(A,B,T)$ such that $$\sum_{y =1}^{N-1}\int_0^{\epsilon^{-2}T} |K_t(x,y;A,B)-K_t(x,y;0,0)| dt  \leq C\epsilon^{1/8}.$$
		
	\end{lem}
	
	\begin{proof} The proof is very similar to that of Lemma \ref{9} with $a=0$. The only difference is that instead of using Lemma \ref{8}, we now use the second bound in Lemma \ref{23} with $v=1/8$, and consequently we get a factor of $\epsilon^{1/8}$ instead of $\epsilon^{1/2}$.
	\end{proof}
	
	\begin{lem}\label{107}
		Let us write $\mathbf p_t^R(x,y;A,B)$ as in the preceding lemma. For $x, \bar x \in \{0,...,N\}$, we have $$\sum_{y \geq 0}\int_0^{\infty} \nabla^+ \mathbf p_t^R(x,y;0,0) \nabla^+ \mathbf p_t^R (\bar x, y ;0,0) dt = -\frac{1}{N+1} 1_{\{x \neq \bar x\}} + \frac{N}{N+1} 1_{\{x=\bar x\}}.$$ 
	\end{lem}
	
	\begin{proof} Let us recall the summation-by-parts identity: if $u,v: \{-1,...,N+1\} \to \Bbb R$, then $$\sum_{y=0}^{N} u(y) \Delta v(y) = u(N+1)\nabla^+ v(N)+u(-1)\nabla^- v(0) - \sum_{y=-1}^{N} \nabla^+u(y) \nabla^+ v(y).$$ Letting $u = \mathbf p_t^R(x, \cdot;0,0)$ and $v= \mathbf p_t^R(\bar x, \cdot;0,0)$, the boundary terms will all vanish and we get \begin{align*}-\sum_{y=0}^{N} \nabla^+ \mathbf p_t^R(x,y) \nabla^+ \mathbf p_t^R (\bar x, y) &= \sum_{y=0}^{N}\mathbf p_t^R(x,y) \Delta \mathbf p_t^R (\bar x, y )=\sum_{y=0}^{N}\Delta\mathbf p_t^R(x,y) \mathbf p_t^R (\bar x, y).
		\end{align*}
		Since $\Delta \mathbf p_t^R = 2\partial_t \mathbf p_t^R$, this implies that $$-\sum_{y=0}^{N} \nabla^+ \mathbf p_t^R(x,y) \nabla^+ \mathbf p_t^R (\bar x, y) = \sum_{y=0}^{N}\mathbf p_t^R(x,y) \partial_t \mathbf p_t^R (\bar x, y)+\partial_t\mathbf p_t^R(x,y) \mathbf p_t^R (\bar x, y).$$ Integrating both sides from $t=0$ to $\infty$ and using the semigroup property, we find that \begin{align*}
		-\sum_{y=0}^{N}\int_0^{\infty} \nabla^+ \mathbf p_t^R(x,y) \nabla^+ \mathbf p_t^R (\bar x, y) dt&= \sum_{y=0}^{N} \mathbf p_t^R(x,y)\mathbf p_t^R(\bar x,y) \bigg|_{t=0}^{\infty} \\ &= \mathbf p_{2t}^R(x, \bar x)\bigg|_{t=0}^{\infty} \\ &= \frac{1}{N+1}-1_{\{x=\bar x\}}
		\end{align*}which proves the claim. In the final line, we used the fact that $\lim_{t \to \infty} \mathbf p_t^R(x,\cdot; 0,0)$ is the uniform measure on $\{0,...,N\}$, which follows from the basic theory of finite-state Markov Chains.
	\end{proof}
	
	\begin{prop}\label{25}
		Let $A,B\in \Bbb R$, and $T>0$. Let $K_t$ be as in Lemma \ref{24}. There exists some $\epsilon_0=\epsilon_0(A,B,T)$ and some $c_*=c_*(A,B,T)<1$ such that for $\epsilon<\epsilon_0$ and $x \in \{1,...,N-1\}$ $$\sum_{y = 1}^{N-1}\int_0^{\epsilon^{-2}T} |K_t(x,y;A,B)| dt  \leq c_*.$$ Moreover, for any $S\in[0,T]$ there is a $C=C(A,B,T,S)$ such that for $s=\epsilon^{-2}S$ we have $$\sum_{y= 1}^{N-1} \int_0^s |K_t(x,y;A,B)|  (s-t)^{-1/2} dt \leq C\epsilon.$$
	\end{prop}
	
	\begin{proof} The proof of the first bound is the same as that of Proposition \ref{10} with $a=0$, the only difference is the factor of $\epsilon^{1/8}$ rather than $\epsilon^{1/2}$. Obviously one should use Lemmas \ref{24} and \ref{107} instead of the analogous half-line estimates \ref{9} and \ref{108}.
	\\
	\\
	The proof of the second bound can be copied verbatim from the proof of Proposition \ref{11}, with $a=0$. \end{proof}
	
	Finally we conclude this section with the analogue of Proposition \ref{33}:
	
	\begin{prop}\label{26} For every $t \geq s \geq 0$ and $x,y \geq 0$ we have that $$\mathbf p_s^R(x,y) \leq e^{t-s} \mathbf p_t^R(x,y).$$
	\end{prop}
	
	\begin{proof} Similar to the proof of Proposition \ref{33}. \end{proof}
	
	\subsection{Continuum Estimates}
	
	Throughout this section, $I$ will denote either $[0, \infty)$ or $[0,1]$. Similarly, $\Lambda$ will denote either $\Bbb Z_{\geq 0}$ or $\{0,...,N\}$. We will fix Robin boundary parameters $A$ and $B$. As before, $\mathbf p_t^R$ will denote the (discrete-space, continuous-time, $\epsilon$-dependent) Robin heat kernel on $\Lambda$, with boundary parameters $\mu_A, \mu_B$. By an abuse of notation, we will write $\mathbf p_t(x,y)$ even when $x,y \in \Bbb R$, and this quantity is meant to be understood as a linear interpolation of the values of $\mathbf p_t^R$ from nearby integer-coordinate points.
	
	\begin{thm}[Existence of the Continuum Robin Heat Kernel]\label{controb} For $T \geq 0$ and $X,Y \in I$ let $$P^{\epsilon}_T(X,Y):= \epsilon^{-1} \mathbf p_{\epsilon^{-2}T}^R (\epsilon^{-1}X, \epsilon^{-1}Y).$$ Then $P^{\epsilon}_T(X,Y)$ converges to a limit $P_T(X,Y)$ as $\epsilon \to 0$. For any $0<\delta<\tau$, the convergence is uniform over $(T,X,Y) \in [\delta,\tau] \times I \times I$. Moreover, the limit $ P_T(X,Y)$ satisfies $$\partial_T P_T(X,Y) = \frac{1}{2} \partial_X^2 P_T(X,Y),\;\;\;\;\;\;\;\;\;\;\;\;\;\;\;\;\;\;\;\;\;$$ $$\partial_XP_T(0,Y) = AP_T(0,Y),\;\;\;\;\;\;\;\;\;\;\;\;\;\;\;\;\;\;\;\;\;\;\;\;\;$$ $$\partial_XP_T(1,Y) = BP_T(1,Y), \;\;\;\;\; \text{if }\; I=[0,1].$$ Furthermore, for every $X \in I$, $P_T(X, \cdot)$ converges weakly to $\delta_X$ as $T \to 0$.
		
	\end{thm}
	
	\begin{proof} Let $\tau> \delta >0$. Using Propositions \ref{2}, \ref{3}, and \ref{6} (or \ref{14}, \ref{15}, and \ref{20}) with $b=0$ and $v=1$, we easily verify that for every $S<T \in [\delta, \tau]$ and $X,Y,Z \in I$ we have (say, for $\epsilon \leq 1$) $$\;\;\;\;\;\;\;\;\;\;P_T^{\epsilon}(X,Y) \leq CT^{-1/2}, $$ $$|P_T^{\epsilon}(X,Z)-P_T^{\epsilon}(Y,Z)| \leq CT^{-1} |X-Y|, $$ $$|P_T^{\epsilon}(X,Y) - P_S^{\epsilon}(X,Y)| \leq CS^{-3/2} |T-S|,$$ where $C$ is a constant depending only on $A,B$, and the terminal time $\tau$. Since $S,T>\delta$, we find that $T^{-1/2} \leq \delta^{-1/2}$, also $T^{-1} \leq \delta^{-1}$, and similarly $S^{-3/2} \leq \delta^{-3/2}$. Therefore, we have proved that the collection $\{P_{\cdot}^{\epsilon}(\cdot,\cdot)\}_{\epsilon \in (0,1]}$ is uniformly bounded and uniformly Lipchitz (in both the time variable and in both spatial variables by symmetry) on $[\delta , \tau] \times I \times I$. By the Arzela-Ascoli theorem, we conclude that the family $\{P_{\cdot}^{\epsilon}(\cdot,\cdot)\}_{\epsilon\in (0,1]}$ is precompact in $C([\delta , \tau] \times I \times I)$, so there is at least one limit point as $\epsilon \to 0$.
		\\
		\\
		We will now show that any limit point of the $\{P_{\cdot}^{\epsilon}(\cdot,\cdot)\}$ coincides in a \textit{weak} sense with the fundamental solution of the Robin-boundary heat equation on $I$ (this weak formulation is good enough, because any continuous weak solution is automatically a strong solution by the standard methods of PDE). In other words, if $P_T(X,Y)$ is a limit point, we will show that for any $\phi \in C_c^{\infty}((0,\infty)\times \Bbb R)$ satisfying $\partial_X\phi(T,0)=A\phi(T,0)$ (and $\partial_X\phi(T,1)=B\phi(T,1)$ if $I=[0,1]$), \begin{equation}\label{w}
		-\int_I \int_0^{\infty} P_T(X,Y)\partial_T\phi(T,X) dTdX = \int_I P_T(X,Y) \partial_X^2\phi(T,X)dT dX. \end{equation} To prove this, first note that for any $\epsilon>0$ and any $X,Y \in \epsilon \Bbb Z$, we have (by definition) that $\partial_TP_T^{\epsilon}(X,Y) = \epsilon^{-2}( P_T^{\epsilon}(X+\epsilon,Y)+P_T^{\epsilon}(X-\epsilon,Y)-2P_T^{\epsilon}(X,Y))$. By the linear interpolation, this is still true for $X,Y \in \Bbb R$. Therefore we find that \begin{align}\label{g} -\int_I \int_0^{\infty} P_T^{\epsilon}(X,Y)\partial_T\phi(T,X) dTdX = \int_I \int_0^{\infty} \partial_TP_T^{\epsilon}(X,Y)\phi(T,X) dTdX
		\end{align}
		\begin{align}\label{int}
		&= \int_I \int_0^{\infty} \epsilon^{-2}\big[ P_T^{\epsilon}(X+\epsilon,Y)+P_T^{\epsilon}(X-\epsilon,Y)-2P_T^{\epsilon}(X,Y)\big]\phi(T,X) dTdX.
		\end{align}
		We can separate the expression (\ref{int}) as the sum of three separate integrals based on the three terms appearing in the square parentheses, then we make the substitution $X \mapsto X-\epsilon$ in the first integral, we make the substitution $X \mapsto X+\epsilon$ in the second integral, and we leave the third integral as is. After those calculations we obtain the following expression: 
		\begin{equation}\label{f}
		\int_I \int_0^{\infty} P_T^{\epsilon}(X,Y)\cdot \epsilon^{-2} \big[ \phi(T,X+\epsilon)+\phi(T,X-\epsilon)-2\phi(T,X) \big] dTdX\;+\;O(\epsilon)
		\end{equation}
		where the error term is due to the boundary correction near the endpoints of $I$. The fact that this boundary correction is $O(\epsilon)$ is a consequence of the boundary conditions: since $\mathbf p_t(-1,y) = \mu_A \mathbf p_t(0,y)$, it follows that $\partial_X^{\pm} P_T^{\epsilon}(0,Y) = AP_T(0,Y)+O(\epsilon)$, where $\partial_X^{\pm}$ denotes left/right derivatives. Similarly if $I=[0,1]$, then we also have $\partial_X^{\pm}P_T^{\epsilon}(1,Y) = B P_T^{\epsilon}(1,Y)+O(\epsilon)$. Also recall that $\partial_X \phi(T,0) = A \phi(T,0)$ (and $\partial_X \phi(T,1) = B \phi(T,1)$ if $I=[0,1]$). Using these facts and performing a first-order Taylor expansion of the integrand of the boundary correction gives the $O(\epsilon)$ error.
		\\
		\\
		Since $\partial_X^2\phi$ is continuous and compactly supported, it follows that $$\lim_{\epsilon \to 0}\; \sup_{\substack{T>0 \\ X \in \Bbb R}} \;\bigg| \partial_X^2\phi(T,X) - \epsilon^{-2} \big[ \phi(T,X+\epsilon)+\phi(T,X-\epsilon)-2\phi(T,X) \big] \bigg| = 0.$$
		Indeed, this can be seen by using the fundamental theorem of calculus to rewrite the parenthetical term: $\phi(T,X+\epsilon)+\phi(T,X-\epsilon)-2\phi(T,X) = \int_X^{X+\epsilon} \int_{Z-\epsilon}^{Z}\partial_X^2 \phi(T,W)dWdZ  $, and then using uniform continuity of $\partial_X^2\phi$ on $(0,\infty) \times \Bbb R$.
		\\
		\\
		So if $P_T(X,Y)$ is a limit point of $\{P_{\cdot}(\cdot, \cdot) \}_{\epsilon \in (0,1]}$, then there is a subsequence of $\epsilon \to 0$ such that $P_T^{\epsilon}(X,Y) \to P_T(X,Y)$ along that subsequence. Letting $\epsilon \to 0$ along the same subsequence in the LHS of (\ref{g}) and in the equivalent expression (\ref{f}) proves the desired identity (\ref{w}), thus completing the proof that $P^{\epsilon}_T(X,Y)$ converges to a function $P_T(X,Y)$ which satisfies the heat equation with Robin boundary.
		\\
		\\
		All that is left to show is that $P_T(X, \cdot)$ converges weakly to $\delta_X$ as $T \to 0$. In other words, if $\varphi \in C_c^{\infty}(\Bbb R)$, we wish to show that $\int_I P_T^R(X,Y)\varphi(Y)dY \stackrel{T\to 0}{\longrightarrow} \varphi(X)$. This will require the usage of Propositions \ref{27} and \ref{29} proved below, hence the reader may wish to take a look at those estimates and return to this proof a bit later (note that the proofs of those estimates do not use this property of $P_T$, hence there is no circular logic here). 
		First note by Proposition \ref{29} that there exists some $C>0$ such that for $T \leq 1$ we have $$\bigg| \int_IP_T(X,Y) dY -1 \bigg| \leq CT^{1/2}.$$ Therefore by the triangle inequality \begin{align*} \bigg|\int_I P_T^R(X,Y)\varphi(Y)dY - \varphi(X)\bigg| &\leq \int_I P_T(X,Y) \big| \varphi(Y)-\varphi(X) \big| dY+ \bigg| \int_IP_T(X,Y) dY -1 \bigg|\varphi(X) \\ &\leq \int_I P_T(X,Y) \big| \varphi(Y)-\varphi(X) \big|dY+CT^{1/2} \varphi(X). \end{align*}
		Hence it suffices to show that $$\int_I P_T(X,Y) \big| \varphi(Y)-\varphi(X) \big| dY\; \stackrel{T \to 0}{\longrightarrow}\; 0 .$$ Since $\varphi \in C_c^{\infty}(\Bbb R)$ it follows that $\varphi$ is Lipchitz so that $|\varphi(Y)-\varphi(X)| \leq C|Y-X|$. Moreover by applying the first estimate in Proposition \ref{27} with $b=1$ we find that there exists $C$ such that for $T \leq 1$, we have $P_T(X,Y) \leq C T^{-1/2} e^{-|X-Y|/\sqrt{T}}$. Thus for $T \leq 1$, \begin{align*} 
		\int_I P_T(X,Y) |\varphi(Y) - \varphi(X)| dY & \leq C T^{-1/2} \int_I e^{-|X-Y| /\sqrt{T} } |Y-X| dY \\ &\leq CT^{-1/2}  \int_{\Bbb R} e^{-|Z|/\sqrt{T}} |Z| dZ \\ & \leq CT^{-1/2} \int_{\Bbb R} e^{-|W|} |T^{1/2}W| (T^{1/2}dW) \\ &= CT^{1/2}
		\end{align*}
		where we made the substitution $Z=Y-X$ in the second line, and another substitution $W=Z/\sqrt{T}$ in the third line. Now we let $T \to 0$ which proves the claim.
	\end{proof}
	
	From now onward, $P_T(X,Y)$ will denote the continuum Robin heat kernel which has been constructed in Theorem \ref{controb}.
	
	\begin{prop}\label{27} Fix a terminal time $\tau>0$. For any $b \geq 0$, there exists some constant $C=C(A,B,b,\tau)$ such that for all $S<T \leq \tau$ and $X,Y,Z \in I$ we have that \begin{equation}\label{31}
		P_T(X,Y) \leq CT^{-1/2} e^{-b|X-Y|/\sqrt{T}},
		\end{equation}
		\begin{equation}\label{30}
		|P_T(X,Z)-P_T(Y,Z)| \leq CT^{-1} |Y-Z|,
		\end{equation}
		\begin{equation}
		|P_T(X,Y)-P_S(X,Y)| \leq CS^{-3/2}|T-S|.
		\end{equation}
		
	\end{prop}
	
	\begin{proof} For $\epsilon>0$, let $P_T^{\epsilon}(X,Y)$ be as in Theorem \ref{controb}. Note that all of these estimates already hold for $P_T^{\epsilon}(X,Y)$, by Propositions \ref{2}, \ref{3}, and \ref{6} (or by \ref{14}, \ref{15}, and \ref{20}). Moreover the constant $C$ does not depend on $\epsilon$. Letting $\epsilon \to 0$, it follows that these estimates still hold in the limit. \end{proof}
	
	\begin{prop}\label{28} Fix a terminal time $\tau >0$. For any $a\geq 0$, there exists a constant $C=C(a,\tau, A,B)$ such that for $T \leq \tau$ and $X \in I$, $$\int_I P_T(X,Y) e^{aY} dY \leq Ce^{aX}.$$ 
	\end{prop}
	
	\begin{proof} Using (\ref{31}) with $b=1+a\tau^{1/2}$, and the fact that $e^{aY} \leq e^{aX} e^{a|X-Y|}$ for $X \in I$, one sees that there exists $C=C(a,A,B,\tau)$ such that for all $T \leq \tau$ and $X \in I$
		\begin{align*} 
		\int_IP_T(X,Y)e^{aY}dY &\leq C e^{aX}T^{-1/2} \int_I e^{-b|X-Y|/\sqrt{T}} e^{a|X-Y|}dY \\ &\leq Ce^{aX} T^{-1/2} \int_{\Bbb R} e^{-b|Z|/\sqrt{T}} e^{a|Z|}dZ \\ & = Ce^{aX} \int_{\Bbb R} e^{-b|W|} e^{aT^{1/2}|W|}dW \\ &\leq Ce^{aX} \int_{\Bbb R} e^{-|W|} dW
		\end{align*}
		where we made a substitution $Z=Y-X$ in the second line, and another substitution $W=T^{-1/2}Z$ in the next line. In the final line we used the fact that $aT^{1/2}-b \leq a\tau^{1/2}-b=-1$.
	\end{proof}
	
	\begin{prop}\label{29} Fix a terminal time $\tau \geq 0$. Then there is a constant $C=C(A,B,\tau)$ such that for $T \leq \tau$ and $X \in I$ we have $$ \bigg|\int_I P_T(X,Y)dY-1 \bigg| \leq CT^{1/2} .$$
		
	\end{prop}
	
	\begin{proof} Let $P_T^{\epsilon}(X,Y)$ be as in Theorem \ref{controb}. Using Proposition \ref{7} or \ref{21} with $v=1$, there exists $C>0$ such that for all $T \leq \tau$, $X \in I$, and (small enough) $\epsilon>0$, \begin{equation}\label{32}\bigg| \epsilon \sum_{Y \in \epsilon \Lambda} P_T^{\epsilon}(X,Y) -1 \bigg| \leq CT^{1/2}.\end{equation} Notice that $$\epsilon\sum_{Y\in \epsilon \Lambda} P_T^{\epsilon}(X,Y) = \int_I P_T^{\epsilon}(X,\epsilon\lfloor \epsilon^{-1}Y \rfloor)dY.$$ Moreover, $P_T^{\epsilon}(X,\epsilon\lfloor \epsilon^{-1}Y \rfloor) \stackrel{\epsilon \to 0}{\longrightarrow} P_T(X,Y)$, by uniform convergence of $P_T^{\epsilon}(X, \cdot)$ to $P_T(X, \cdot)$.
		\\
		\\
		Using Proposition \ref{2} or \ref{14} with $b=1$, we have the following bound, uniformly over all small enough $\epsilon>0$: $$|P_T^{\epsilon}(X,Y)| \leq CT^{-1/2} e^{-|Y-X|/\sqrt{T}}.$$ For each fixed $X$ and $T$, the RHS is an integrable function of $Y$, so it follows from the dominated convergence theorem (together with the preceding observations) that $$\lim_{\epsilon \to 0} \epsilon \sum_{Y \in \epsilon \Lambda} P_T^{\epsilon}(X,Y) = \lim_{\epsilon \to 0} \int_I P_T^{\epsilon}(X,\epsilon\lfloor \epsilon^{-1}Y \rfloor)dY= \int_I P_T(X,Y)dY.$$ Letting $\epsilon \to 0$ in (\ref{32}) gives the result.
	\end{proof}
	
	\section{The SHE with Robin Boundary Conditions}
	
	Next we want to describe the continuum version of the height functions in our particle system, which we expect will (in a sense) solve the KPZ Equation on the spatial domain $I$. Recall that this equation is formally given by $$\partial_TH = \frac{1}{2} \partial_X^2 H + \frac{1}{2}(\partial_XH)^2 + \xi $$
	where $\xi$ is a Gaussian space-time white noise, meaning informally that $\Bbb E [ \xi(S,X)\xi(T,Y)] = \delta(S-T)\delta(X-Y)$. 
	\\
	\\
	In order to solve this equation, let us first make precise how to rigorously define the noise term. One may construct $\xi$ as the distributional time-derivative $\xi = \partial_TW$ of a cylindrical Wiener process $W=(W_T)_{T \geq 0}$ over $L^2(I)$, in which case each individual $W_T$ may be viewed as a random element of the Sobolev space $H^s_{loc}(I)$ for $s<-1/2$. This viewpoint will be very useful to us because it allows one to define stochastic integrals against $\xi$, which in turn allows us to construct strong solutions to parabolic PDEs which are driven by $\xi$. See for instance [Hai09] or [DPZ92] or [Wal86] for the general theory of space-time stochastic integrals.
	\\
	\\
	Given the regularity (or lack thereof) of the noise $\xi$, one may then heuristically compute (using Schauder estimates or the  Kolmogorov continuity theorem) that the solution to the KPZ equation should be locally Hölder $1/2-$ in space and Hölder $1/4-$ in time, but no better. In particular, we are faced with two serious problems:
	
	\begin{enumerate}
		\item The nonlinear term $(\partial_XH)^2$ is undefined, since it is not possible to square the derivative of a function which is Hölder $1/2-$. So the PDE is ill-posed. 
		
		\item To make matters worse, the boundary parameters for ASEP should somehow correspond to Neumann boundary conditions for the PDE so that $\partial_XH(T,0) = A$ in the half-line case (and also $\partial_XH(T,1) = -B$ for the bounded interval case). But once again these quantities are ill-posed.
	\end{enumerate}
	The first of the two problems described may be fixed by the well-known Cole-Hopf transform, in which we define $\mathcal Z:=\exp H$ and then formally $\cal Z$ solves the multiplicative Stochastic Heat Equation (SHE): $$\partial_T\mathcal Z = \frac{1}{2}\Delta \mathcal Z+ \mathcal Z \xi$$
	\\
	As it turns out, this transformation now makes the equation well-posed, and by considering the noise as a cylindrical Wiener process, one can hope to make sense of solutions in a Duhamel form, see Theorem \ref{70} below. Once we construct this solution and prove it is positive, we can \textit{formally} just define $H:=\log \cal Z$. One may wish to look at [GJ14] and [GPS17] for methods which avoid the Cole-Hopf transformation.
	\\
	\\
	So now we only need to make sense of the Neumann Boundary conditions described above for the KPZ equation. Under the Cole-Hopf transform just discussed, the Neumann boundary conditions for the KPZ equation formally become $Robin$ boundary conditions for the SHE: $$\partial_X\mathcal Z(T,X)|_{X=0}=A\mathcal Z(T,0)$$ (and also $\partial_X \mathcal Z(T,X)|_{X=1}=-B\mathcal Z(T,1)$ for the bounded-interval case). If we were to forget about the noise term in the SHE for the moment being (thus we consider the equation $\partial_T\mathcal Z = \frac{1}{2} \Delta \cal Z$), then by theorem \ref{controb} the solution is given by a semigroup kernel: $$\mathcal Z(T,X) = P_T * \mathcal Z_0 (X) = \int_I P_T(X,Y)\mathcal Z(0,Y)dY.$$ This fact motivates the following definition of solutions in mild (Duhamel) form:
	
	\begin{defn}[Cole-Hopf Solution with Neumann Boundary Conditions]\label{69} Let $P_T$ denote the continuum Robin heat kernel constructed in Theorem \ref{controb}, and let $\xi$ denote a space-time white noise on some probability space $(\Omega, \mathcal F, \Bbb P)$. Let $\mathcal Z_0$ denote some (random) Borel measure which is $\Bbb P$-almost surely supported on $I$. We say that a space-time process $\mathcal Z = (\mathcal Z(T,X))_{T > 0, X \in I}$ is a mild solution to the SHE satisfying Robin boundary conditions if $\Bbb P$-almost surely, for every $T>0$ and $X \in I$ we have $$\mathcal Z(T,X) = \int_I P_T (X,Y)\mathcal Z_0(dY)+\int_0^T\int_I P_{T-S}(X,Y) \mathcal Z(S,Y) \xi(S,Y) dSdY$$ where the integral against the white noise is meant to be interpreted in the Itô sense (see [Wal86], [Hai09], or [DPZ92]). If $\xi = \partial_T W$ for a cylindrical Wiener Process $W$, then we may abbreviate the above expression as $$\mathcal Z_T = P_T * \mathcal Z_0 + \int_0^T P_{T-S} * (\mathcal Z_S \cdot dW_S)$$ where the $*$ always denotes a spatial convolution.
		
	\end{defn}
	
	The following proposition gives us conditions for the existence, uniqueness, and positivity of solutions, starting from initial data which is defined pointwise at each $X \in I$, and bounded in $L^2(\Bbb P)$ under an exponential weight function.
	
	\begin{prop}[Existence/Uniqueness/Positivity of Mild Solutions with $L^2$-bounded initial data]\label{70} Let $\xi$ be a space-time white noise. Suppose that we have some (random, function-valued) initial data $\mathcal Z_0$ which satisfies the following condition for some $a>0$: $$\sup_{ \substack{X\in I}} e^{-aX} \Bbb E[\mathcal Z_0(X)^2]<\infty.$$ If $I=[0,1]$ we can just assume $a=0$. Then there exists a mild solution to the SHE which is adapted to the natural filtration generated by the noise, $\mathcal F_T:=\sigma( \{\xi(S,\cdot) \}_{S\leq T})$. This mild solution is unique in the class of adapted processes $\mathcal Z$ satisfying $$\sup_{ \substack{X\in I \\ S \in [0,T]}} e^{-aX} \Bbb E[\mathcal Z_S(X)^2]<\infty.$$ Furthermore, if $\mathcal Z_0$ is a.s. positive, then $\mathcal Z(S, \cdot)>0,\; \forall S$ a.s. 
		
	\end{prop} 
	
	\begin{proof} The argument here is adapted from [Wal86]. The informal argument is as follows: We fix a terminal time $\tau>0$, and we define the following sequence of iterates for $T\leq \tau$ and $X \in I$: \begin{align*}u_0(T,X)&:= \int_I P_T(X,Y) \mathcal Z_0(Y)dY, \\ u_{n+1}(T,X)&:= \int_0^T \int_I P_{T-S}(X,Y) u_n(S,Y) \xi(S,Y)dYdS.\end{align*}
		The fact that the stochastic integral defining $u_{n+1}$ actually exists will follow from Equations (\ref{67}) and (\ref{68}) below, but we will leave the formalism for later. If we let $z_N:=\sum_0^Nu_n$ then we have the relation that $$z_{N+1}(T,X) = \int_I P_T(X,Y) \mathcal Z_0(Y)dY+\int_0^T\int_I P_{T-S}(X,Y)z_N(S,Y) \xi(S,Y)dYdS.$$ By letting $N \to \infty$, it follows that our mild solution should heuristically be given by $\mathcal Z=\lim_{N\to \infty} z_N= \sum_{0}^{\infty} u_n$. Therefore, the aim is now to prove that the series $\sum_n u_n$ converges absolutely in the appropriate Banach Space.
		\\
		\\
		Let us now formalize the argument. Consider the Banach space $\cal B$ consisting of $C(I)$-valued, adapted processes $u=(u(T,\cdot))_{T \in [0,\tau]}$ which satisfy the condition $$\|u\|_{\mathcal B}^2 := \sup_{ \substack{X\in I \\ T \in [0,\tau]}} e^{-aX} \Bbb E[ u(T,X)^2]<\infty .$$ 
		Let $u_n$ be given as above. Define a sequence of functions $(f_n)_{n\geq 0}$ from $[0,\tau] \to \Bbb R_+$ by $$f_n(T):= \sup_{\substack{X\in I \\ S \in [0,T]}} e^{-aX} \Bbb E [ u_n(S,X)^2]$$
		where the RHS is defined to be $+\infty$ if the stochastic integral defining $u_{n}$ fails to exist (which will happen iff $\int_{[0,T] \times I} P_{T-S}(X,Y)^2 \Bbb E[u_n(S,Y)^2] dYdS = +\infty$).
		\\
		\\
		By Itô's isometry and the definition of $f_n$, we see that
		\begin{align}\label{67}
		\Bbb E[u_{n+1}(T,X)^2] &= \int_0^T  \int_I P_{T-S}(X,Y)^2 \Bbb E[u_n(S,Y)^2] dYdS \nonumber \\ & \leq \int_0^T \bigg( \int_I P_{T-S}(X,Y)^2 \cdot e^{aY} dY \bigg) f_n(S) dS.
		\end{align} By Propositions \ref{27} and \ref{28} above, we easily obtain the bound: \begin{equation}\label{68}\int_I P_T(X,Y)^2e^{aY} dY \leq C T^{-1/2} e^{aX},\;\;\;\;\;\;\;\; \forall X\in I, \;T \leq \tau. \end{equation} Here $C$ is a constant depending only on $A,B$ and the time horizon $\tau$. Note that $f_n$ is an increasing function, which implies that $T \mapsto \int_0^T (T-S)^{-1/2} f_n(S)dS$ is also increasing in $T$ (which can be proved by substituting $S=TU)$. Together with equations (\ref{67}) and (\ref{68}), this implies that $$f_{n+1}(T) \leq C \int_0^T (T-S)^{-1/2} f_n(S)dS$$ which we can iterate twice to obtain that $$f_{n+2}(T) \leq C' \int_0^T f_n(S)dS$$ 
		By assumption, we have that $\sup_{X \in I} e^{-aX} \Bbb E[\mathcal Z_0(X)^2] <\infty$ which implies (for instance by (\ref{68})) that $f_0(T) < \infty$ for all $T \in [0,\tau]$. So by iterating this recursion, one obtains the result that $f_n(T) \leq CT^{n/2}/(n/2)!$, which implies that the stochastic integral defining $u_n$ always exists, and moreover that $\sum_n \|u_n\|_{\mathcal B}<\infty$, as desired.
		\\
		\\
		Uniqueness is proved in a similar way: If $\mathcal Z, \mathcal Z'$ are two different solutions then $$\mathcal Z(T,X) - \mathcal Z'(T,X) = \int_0^T \int_I P_{T-S}(X,Y) \big[ \mathcal Z(S,Y) - \mathcal Z'(S,Y) \big] \xi(S,Y) dY dS$$ Squaring both sides and taking expectations, $$\Bbb E\big[\big(\mathcal Z(T,X) - \mathcal Z'(T,X)\big)^2\big] = \int_0^T \int_I P_{T-S}(X,Y)^2 \Bbb E \big[ \big(\mathcal Z(S,Y) - \mathcal Z'(S,Y)\big)^2 \big] dYdS$$ so now Gronwall's lemma (or direct iteration) shows that both sides must be zero.
		\\
		\\
		The positivity result can be adapted from [Mue91].
	\end{proof}
	
	Although the above proposition gives us existence and uniqueness results for a wide class of initial data, there are still some natural choices of initial data which are not covered. In section 6, we will especially need the case of $\delta_0$ initial data, which is clearly not covered by the above proposition (since it is not function-valued).
	
	\begin{prop}[Existence/Uniqueness/Positivity of Mild Solutions with $\delta_0$-initial data]\label{71} Fix a space-time white noise $\xi$. There exists a $C(I)$-valued process $(\mathcal Z_T)_{T \geq 0}$ which is adapted to $\mathcal F_T:=\sigma( \{\xi(S,\cdot) \}_{S\leq T})$, and satisfies the conditions of Definition \ref{69} with $\mathcal Z_0=\delta_0$: $$\mathcal Z_T(X) = P_T(X,0) + \int_0^T \int_I P_{T-S}(X,Y) \mathcal Z_S(Y) \xi(S,Y) dYdS.$$ This mild solution is unique in the class of adapted processes $\mathcal Z$ satisfying $$\sup_{ \substack{X\in I \\ S \in [0,T]}} S\cdot  \Bbb E[\mathcal Z_S(X)^2]<\infty.$$ Furthermore, $\mathcal Z_S>0,\; \forall S$ a.s. 
		
	\end{prop}
	
	\begin{proof} Similar to before, let us fix a terminal time $\tau$ and then define a Banach space $\mathcal B$ which consists of $C(I)$-valued, adapted processes $u=(u(T,\cdot))_{T \in [0,\tau]}$ satisfying $$\|u\|_{\mathcal B}^2 := \sup_{ \substack{X\in I \\ T \in [0,\tau]}} T\cdot \Bbb E[u(T,X)^2]<\infty.$$
	As before, define a sequence of iterates $T\leq \tau$ and $X \in I$: \begin{align*}u_0(T,X)&:= P_T(X,0), \\ u_{n+1}(T,X)&:= \int_0^T \int_I P_{T-S}(X,Y) u_n(S,Y) \xi(S,Y)dYdS.\end{align*}
	As before, we just need to show that $\sum_n \|u_n\|_{\mathcal B}<\infty$.
		\\
		\\
		Similar to the proof of Proposition \ref{70}, we define $$f_n(T):= \sup_{\substack{X\in I \\ S \in [0,T]}} S^{1/2} P_S(X,0)^{-1} \Bbb E [ u_n(S,X)^2]$$
		where the RHS is defined as $+\infty$ if the stochastic integral defining $u_n$ fails to exist.
		\\
		\\
		Using the Itô isometry, the definition of $f_n$, and the fact that (by Proposition \ref{27}) $P_{T-S} \lesssim (T-S)^{-1/2}$, we compute \begin{align*}\Bbb E [u_{n+1}(T,X)^2] &= \int_0^T \int_I P_{T-S}(X,Y)^2 \Bbb E[u_n(S,Y)^2] dYdS \\ &\leq \int_0^T \int_I P_{T-S}(X,Y)^2 \cdot S^{-1/2} P_S(Y,0) f_n(S) dY dS \\ &\leq C \int_0^T (T-S)^{-1/2}S^{-1/2} \bigg[\int_I P_{T-S}(X,Y) P_S(Y,0) dY \bigg]f_n(S) dS \\ & = C P_T(X,0) \int_0^T (T-S)^{-1/2}S^{-1/2}  f_n(S)dS \end{align*}
		where we used the semigroup property in the final line. Multiplying both sides by $T^{1/2}P_T(X,0)^{-1}$, we find that \begin{align*}T^{1/2}P_T(X,0)^{-1} \Bbb E[u_{n+1}(T,X)] &\leq CT^{1/2} \int_0^T (T-S)^{-1/2}S^{-1/2}f_n(S)dS .
		\end{align*}
		Just like before, $f_n$ is an increasing function, therefore (by making a substitution $S=TU)$ one may see that the RHS of the last expression is an increasing function of $T$. It follows that \begin{align*}f_{n+1}(T) &\leq CT^{1/2} \int_0^T (T-S)^{-1/2}S^{-1/2} f_n(S)dS \end{align*} which we can iterate twice to obtain $$f_{n+2}(T) \leq C'T^{1/2} \int_0^T S^{-1/2}f_n(S)dS.$$ Using the fact that (by Prop. \ref{27}) $\sup_{T\in [0,\tau]}f_0(T)\leq C$, we can iterate this recursion to obtain $$f_n(T) \lesssim T^{n/2}/(n/2)! $$ We have just proved that $$\Bbb E[u_n(T,X)^2] \leq CP_T(X,0) T^{(n-1)/2}/(n/2)!$$ for a constant $C$ not depending on $T \in [0,\tau]$, $X \in I$, or $n \in \Bbb N$. By Proposition \ref{27} we know that $P_{T}(X,0) \leq CT^{-1/2}$, and therefore it follows that $\sum_n \|u_n\|_{\mathcal B}<\infty$, which proves existence of the mild solution. 
		\\
		\\
		Uniqueness and positivity are proved in the same manner as in Proposition \ref{70}.
	\end{proof}
	
	Although the mild solution is one notion of what it means to solve the SHE, it is not the only natural definition of a solution. In particular, the notion of a weak solution (which uses the idea of pairing against test functions) is also very useful, and is actually equivalent to the notion of a mild solution:
	
	\begin{prop}[Equivalence of Weak Solutions and Mild Solutions]\label{103}Let $W$ be a cylindrical Wiener process defined on some $(\Omega, \mathcal F, \Bbb P$). Denote by $\mathscr T$ the collection of all $\varphi \in \mathcal S (\Bbb R)$ such that $\varphi'(0) = A \varphi(0)$ and also $\varphi'(1) = B \varphi(1)$ if $I=[0,1]$. Then a $C(I)$-valued process $(\mathcal Z_T)_{T \geq 0}$ is a mild solution to the SHE if and only if for every $\varphi \in \mathscr T$ we have that $$(\mathcal Z_T, \varphi) = (\mathcal Z_0, \varphi) +\frac{1}{2} \int_0^T (\mathcal Z_S, \varphi'') dS + \int_0^T (\mathcal Z_S \varphi , dW_S)$$ where $(\varphi, \psi):=\int_I \varphi \psi$ denotes the $L^2(I)$-pairing. If the latter relation holds, we call $(\mathcal Z_T)$ a \textit{weak} solution of the SHE.
		
	\end{prop}
	
	\begin{proof}
		If $\mathcal Z_0$ is a mild solution as defined in \ref{69}, then $$\mathcal Z_T = P_T*\mathcal Z_0 + \int_0^T P_{T-S} *(\mathcal Z_S dW_S).$$ Pairing both sides against a $\varphi$ and using self-adjointness of $\frac{1}{2}\Delta$ with Robin boundary conditions, \begin{equation}\label{104}(\mathcal Z_T, \varphi) = ( P_T *\mathcal Z_0, \varphi) + \int_0^T \big(\mathcal Z_S \cdot (P_{T-S} * \varphi), dW_S\big).\end{equation} Next, noting that $\partial_T ( P_T *\mathcal Z_0, \varphi) = \frac{1}{2}(P_T * \mathcal Z_0 , \varphi '')$ we have that \begin{align}\label{105}(P_T *\mathcal Z_0, \varphi) &= (\mathcal Z_0, \varphi) + \frac{1}{2}\int_0^T (P_S *\mathcal Z_0, \varphi'')dS \nonumber\\ &= (\mathcal Z_0, \varphi) + \frac{1}{2} \int_0^T \bigg(  \mathcal Z_S - \int_0^S P_{S-U} *(\mathcal Z_U dW_U)\;\;,\;\; \varphi'' \bigg) dS \nonumber\\ &= (\mathcal Z_0, \varphi) + \frac{1}{2} \int_0^T (\mathcal Z_S, \varphi'')dS - \frac{1}{2} \int_0^T \bigg( \mathcal Z_U \cdot \bigg[\int_U^T P_{S-U} * \varphi'' dS\bigg] \;\;, dW_U \bigg) \nonumber\\ &= (\mathcal Z_0, \varphi) + \frac{1}{2} \int_0^T (\mathcal Z_S, \varphi'')dS - \int_0^T \bigg( \mathcal Z_U \cdot \bigg[P_{T-U}* \varphi - \varphi\bigg] \;\;, dW_U \bigg). \end{align} In the third line we distributed terms and switched the order of integration (cf. stochastic Fubini's theorem). In the last line we merely applied the fundamental theorem of calculus to the term within the square bracket. Now Equations (\ref{104}) and (\ref{105}) together imply that $(\mathcal Z_T)$ is a weak solution. Thus mild solutions are weak.
		\\
		\\
		To prove the converse, note that if $$(\mathcal Z_T, \varphi) = (\mathcal Z_0, \varphi) +\frac{1}{2} \int_0^T (\mathcal Z_S, \varphi'') dS + \int_0^T (\mathcal Z_S \varphi , dW_S).$$ Using Itô's formula for Hilbert-space valued processes, we see that if $\varphi \in C^{\infty} ([0,\tau] \times \Bbb R)$ such that $\varphi_T:=\varphi(T,\cdot) \in \mathscr T$ for all $T$, then $$d(\mathcal Z_T, \varphi_T) \;\;=\;\; (d\mathcal Z_T, \varphi_T) + (\mathcal Z_T, d\varphi_T) \;\;=\;\; \big(\mathcal Z_T, \frac{1}{2}\varphi_T'' + \partial_T \varphi_T \big)dT + (\mathcal Z_T \varphi_T, dW_T).$$ Then fixing a time $T>0$, and setting $\varphi_S:= P_{T-S}*\varphi$, we can integrate from $0$ to $T$ to obtain: \begin{align*}
		(\mathcal Z_T, \varphi) - (P_T *\mathcal Z_0, \varphi) &= (\mathcal Z_T, \varphi_T) - (\mathcal Z_0, \varphi_0) \\ &= \int_0^T \big( \mathcal Z_S, \frac{1}{2} \varphi_S'' +\partial_S \varphi_S \big) dS + \int_0^T (\mathcal Z_S \varphi_S, dW_S ) \\ &=0+ \int_0^T \big( \mathcal Z_S \cdot (P_{T-S} * \varphi), dW_S\big)
		\end{align*}
		which (after rearranging terms) is equivalent to $$(\mathcal Z_T, \varphi) = \bigg( P_T*\mathcal Z_0 + \int_0^T P_{T-S} *(\mathcal Z_S dW_S)\;\;,\; \varphi \bigg) $$ so that $(\mathcal Z_T)$ is a mild solution.
	\end{proof}
	
	\section{Proof of Tightness and Identification of the Limit }
	
	Let us first establish a few topological conventions. Throughout this section, $I$ will denote either the interval $[0,\infty)$ or $\Bbb [0,1]$. We will endow $C(I)$ with the topology of uniform convergence on compact sets if $I=[0,\infty)$, and the topology of uniform convergence when $I=[0,1]$. The space $D([0,\tau],C(I))$ will denote the space of all right-continuous functions from $[0,\tau] \to C(I)$ which have left limits. We will endow $D([0,\tau],C(I))$ with the Skorokhod topology, which may be metrized by $$\sigma(\Phi, \Psi) = \inf_{\lambda \in \Lambda} \max \big\{ \|\lambda-id\|_{L^{\infty}[0,\tau]}\;,\; \sup_{T \in [0,\tau]} \|\Phi(T)-\Psi(\lambda(T)) \|_{C(I)}\big\} $$ where $\Lambda$ is the space of increasing homeomorphisms from $[0,\tau]$ to itself, and $\|\cdot\|_{C(I)}$ is a norm inducing the topology on $C(I)$.
	
	\begin{defn}[Parabolic Scaling]\label{40} Let $\Lambda = \Bbb Z_{\geq 0}$ for ASEP-H and $\Lambda = \{0,...,N\}$ for ASEP-B. Let $Z_t(x)$ denote the Gartner-transformed height functions from Section 2. For $\epsilon>0$, $X \in \epsilon^{-1}\Lambda$, and $T \geq 0$, we define the parabolically scaled process $$\mathcal Z^{\epsilon}(T,X) = Z_{\epsilon^{-2}T}(\epsilon^{-1}X)$$ where $Z_t(x)$ is the Gärtner-transformed height function from Definition 2.7. We extend $\mathcal Z^{\epsilon}(T,\cdot)$ from $\epsilon\Lambda$ to the whole interval $I$ by linear interpolation. Throughout this section, we will fix a terminal time $\tau$, and consider (the law of) $\mathcal Z^{\epsilon}$ as a measure on the Skorokhod space $D([0,\tau], C(I))$.
	\end{defn}
	
	We remark that $\mathcal Z^{\epsilon}$ depends on $\epsilon$ in two completely different ways, first due to the space-time scaling in Definition \ref{40}, but secondly also because the non-scaled discrete-space process $Z_t$ depends on $\epsilon$ via the $\epsilon$-scaled parameters $p,q, \alpha, \beta, \gamma, \delta$. This is what we meant by a weak scaling: we scale the model parameters along with the height functions.
	
	\begin{ass}[Near-Equilibrium Initial Conditions]\label{51} For ASEP-H, we will always assume that the initial data $\mathcal Z^{\epsilon}$ satisfies the following bounds: There exists $a \geq 0$ such that for each terminal time $\tau>0$, each $p \geq 1$, and each $\alpha \in [0,\frac{1}{2})$ there exists some constant $C=C(p,\tau, \alpha)$ such that for all $T \in [0,\tau]$, $X,Y \in I$, and $\epsilon>0$ small enough: \begin{align*}\|\mathcal Z^{\epsilon}_0(X)\|_p &\leq Ce^{aX}, \\ \|\mathcal Z^{\epsilon}_0(X)-\mathcal Z^{\epsilon}_0(Y) \|_p &\leq C |X-Y|^{\alpha} e^{a(X+Y)}.\end{align*} Here $\|Z\|_p:= \Bbb E[|Z|^p]^{1/p}$ denotes the $L^p$-norm with respect to the probability measure. For ASEP-B, we will assume the same bounds but with $a=0$.
	\end{ass}
	
	As a consequence of Kolmogorov's continuity criterion, Assumption \ref{51} ensures that with large probability, the random functions $\{\mathcal Z_{0}^{\epsilon}\}_{\epsilon>0}$ are pointwise bounded and equi-Hölder $1/2-$ (locally). In turn, Arzela-Ascoli's theorem and Prohorov's theorem ensure that as $\epsilon \to 0$ the $\mathcal Z_0^{\epsilon}$ converge weakly (along a subsequence) to some random function which has the same regularity as Brownian motion. Hence, Assumption \ref{51} is merely a technical restriction which ensures that the initial data have limit points which are not too wildly behaved on small scales, and whose moments grow exponentially in $X$ (at worst). For instance a product of Bernoulli's will satisfy the assumptions in 3.2, as the associated height function converges to a Brownian motion. On the other hand, narrow-wedge (zero-particle) initial data will not satisfy Assumption \ref{51}. We will find a way to deal with such data in Section 6 below.
	\\
	\\
	We are almost ready to move onto the main results, however we will need a technical result which will be a very useful black box for obtaining $L^p$ estimates. This result was originally obtained in [DT16, Lemma 3.1] and was further elaborated in [CS16, Lemma 4.18], so we will not give the proof.
	
	\begin{lem}\label{73} Let $\|X\|_p = \Bbb E[|X|^p]^{1/p}$ denote the $L^p$ norm with respect to the probability measure, and let $M$ denote the martingale appearing in Proposition \ref{74}. Let $F: [0,\infty) \times \Bbb Z_{\geq 0} \to \Bbb R$ be any bounded function. There exists a constant $C$ (not depending on the function $F$) such that that for any $t>1$ we have that $$\bigg\| \int_0^t \sum_{y \in \Lambda} F(s,y) dM_s(y) \bigg\|_p^2 \leq C \epsilon \int_0^t \sum_y \overline{F} (s,y)^2 \|Z_s(y)\|_p^2 ds$$ where the bar denotes a local supremum: $$\overline{F}(s,y) := \sup_{|u-s|\leq 1} F(u,y).$$
		
	\end{lem}
	
	In practice, we will almost always use the above lemma when $F(s,y) = \mathbf p_s^R(x,y)$ for some $x \in \Lambda$. In this case, by proposition \ref{26} or \ref{33}, it always holds true that $$\sup_{|u-s| \leq 1} \mathbf p_u^R(x,y) \leq e^{1}\mathbf p^R_{s+1}(x,y)$$ and therefore by Proposition \ref{2} or \ref{14} (with $b=0$) it then follows that \begin{equation}\label{76}\overline{\mathbf p}_s^R(x,y)^2 \leq e^{2}\mathbf p^R_{s+1}(x,y)^2 \leq Cs^{-1/2}\mathbf p_{s+1}^R(x,y).
	\end{equation}
	This fact will be used repeatedly during the technical estimates below.
	\\
	\\
	We now move onto the main results.
	
	\begin{prop}[Tightness]\label{85} Fix a terminal time $\tau\geq 0$, and assume that the sequence $\{\mathcal Z_0^{\epsilon} \}_{\epsilon>0}$ of initial data satisfies Assumption \ref{51}. For all $p \geq 1$ and $\alpha \in [0,1/2)$, there exists a constant $C=C(\alpha, p, \tau)$ such that
		
		\begin{equation}\label{77}\|\mathcal Z^{\epsilon}(T,X)\|_p \leq Ce^{aX}\;\;\;\;\;\;\;\;\;\; \end{equation}
		\begin{equation}\label{78}\| \mathcal Z^{\epsilon}(T,Y)-\mathcal Z^{\epsilon}(T,X)\|_p \leq C|X-Y|^{\alpha}e^{a(X+Y)}\;\;\;\;\;\;\;\end{equation}
		\begin{equation}\label{79}\| \mathcal Z^{\epsilon}(T,X) - \mathcal Z^{\epsilon}(S,X)\|_p \leq Ce^{2aX} (|T-S|^{\alpha/2}\vee \epsilon^{\alpha}) \end{equation} uniformly over all $X,Y \in[0,\infty)$ and $S ,T \in [0,\tau]$. For ASEP-B, the same estimates are satisfied for $X,Y \in [0,1]$ and $a=0$. Consequently, the laws of the $\mathcal Z^{\epsilon}$ are tight in the Skorokhod space $D([0,\tau],C(I))$, and moreover the limit point lies in $C([0,\tau],C(I))$.
	\end{prop}
	
	\begin{proof} This was originally proved as [CS16, Proposition 4.15]. However, after a thorough discussion with the authors, we actually found a small mistake in [CS16, (4.57)] where $\mathbf p_{t-s}^R(x,y)$ should be replaced by $\mathbf p_{t-s+1}^R(x,y)$, and this messes up the iteration. This is not so obviously fixed; therefore we will give a new variant of the proof here, which is loosely based on the argument in [DT16, Proposition 3.2].
		\\
		\\
		For the first estimate, note by Proposition $\ref{74}$ that $Z_t$ will satisfy the Duhamel-form equation $$Z_t(x) = \mathbf p_t^R * Z_0(x) + \int_0^t \sum_{y \in \Lambda} \mathbf p_{t-s}^R (x,y) dM_s(y).$$
		Using this identity and the fact that $(x+y)^2 \leq 2x^2+2y^2$, we find that \begin{equation}\label{80}\|Z_t(x)\|_p^2 \leq 2 \big\|\mathbf p_t^R * Z_0(x)\big\|_p^2 + 2\bigg\|\int_0^t \sum_{y \in \Lambda} \mathbf p_{t-s}^R (x,y) dM_s(y) \bigg\|_p^2.\end{equation}
		The first term on the RHS of (\ref{80}) can be bounded as follows: \begin{equation}\label{81}\bigg\| \sum_{y \in \Lambda} \mathbf p_t^R(x,y) Z_0(y) \bigg\|_p \leq \sum_y \mathbf p_t^R(x,y) \|Z_0(y)\|_p \leq \sum_y \mathbf p_t^R(x,y)  e^{a\epsilon y} \leq Ce^{a\epsilon x}\end{equation}
		where we applied Assumption \ref{51} in the second inequality and Corollary \ref{4} in the next one. For $t \geq 1$, the second term on the RHS of (\ref{80}) can be bounded using Lemma \ref{73} with (\ref{76}): \begin{equation}\label{82}\bigg\|\int_0^t \sum_{y \in \Lambda} \mathbf p_{t-s}^R (x,y) dM_s(y) \bigg\|_p^2 \leq C \epsilon \int_0^t (t-s)^{-1/2} \sum_y \mathbf p_{t-s+1}^R(x,y) \|Z_s(y)\|_p^2 ds.\end{equation}
		Let us define the following quantity: $$[Z_t]_p := \sup_{x \in \Lambda} e^{-a \epsilon x} \|Z_t(x)\|_p. $$Multiplying both sides of (\ref{80}) by $e^{-2a\epsilon x}$ and then using (\ref{81}) and (\ref{82}), we see that \begin{align}e^{-2a\epsilon x}\|Z_t(x)\|_p^2 &\leq C + C \epsilon e^{-2a\epsilon x} \int_0^t (t-s)^{-1/2} \bigg( \sum_y \mathbf p_{t-s+1}^R(x,y) e^{2a \epsilon y}\bigg) [Z_s]_p^2 \;ds \nonumber \end{align} Now, by Corollary \ref{4} we know that $$\sup_{x \geq 0}\; e^{-2a\epsilon x} \sum_y \mathbf p_{t-s+1}^R(x,y)e^{2a\epsilon y}<\infty$$ so by taking the sup over $x$ on both sides of the previous expression, we find that \begin{equation}\label{83}[Z_t]_p^2 \leq C+C\epsilon \int_0^t (t-s)^{-1/2} [Z_s]_p^2ds\end{equation}Now we would like to iterate the inequality in (\ref{83}), however the problem is that we have only proved this bound for $t\geq 1$ (see the statement of Lemma \ref{73}), but we also need to prove it for $t \in[0,1]$ in order to apply the iteration. For $t \leq 1$, we have (by the definition of $Z_t)$ that $Z_t(x) \leq e^{2\sqrt{\epsilon} N(t)} Z_0(x)$, where $N(t)$ denotes the net number of jumps which occur at site $x$ up to time $t$. Since $t \leq 1$, each $N(t)$ is stochastically dominated by a Poisson random variable of rate $1$. Then by Assumption \ref{51}, $\|Z_t(x)\|_p \leq \Bbb E[e^{pN(t)}]^{1/p} \cdot \|Z_0\|_p \leq Ce^{a\epsilon x}$, and thus $\sup_{t \leq 1} [Z_t]_p<\infty$, which proves that (\ref{83}) holds for $t \leq 1$ as well.
		\\
		\\
		Iterating (\ref{83}) twice, we find that \begin{equation}\label{84}[Z_t]_p^2 \leq C+ C^2\epsilon t^{1/2} + \alpha C^2\epsilon^2 \int_0^t [Z_s]_p^2ds \end{equation} where $\alpha = \int_u^t (s-u)^{-1/2}(t-s)^{-1/2}ds$, which is a constant not depending on $u$ or $t$, as can be verified by making the substitution $v=(t-s)/(s-u)$. Since $\epsilon t^{1/2} \leq T^{1/2}$, the second term on the RHS of (\ref{84}) may be absorbed into the first one, only changing the constant $C$. Then an easy application of Gronwall's lemma shows that $$[Z_t]_p \leq [Z_0]_p e^{\alpha C^2 \epsilon^2t} \leq C$$ where we used the fact that $\epsilon^2t \leq T$ which is a constant not depending on $\epsilon$ or $t$ or $x$. We have shown that $$\|Z_t(x)\|_p \leq Ce^{a\epsilon x}$$ which (after changing to macroscopic variables $X=\epsilon x$ and $T=\epsilon^2 t$) proves the first bound (\ref{77}).
		\\
		\\
		Let us now move onto the second bound (\ref{78}). Note that $$Z_t(x)-Z_t(y) = \sum_z \big( \mathbf p_t^R(x,z)-\mathbf p_t^R(y,z) \big) Z_0(z) + \int_0^t \sum_z \big(\mathbf p_t^R(x,z)-\mathbf p_t^R(y,z) \big) dM_s(z).$$ Let us name the terms on the right side as $I_1,I_2$ respectively. 
		\\
		\\
		In order to bound $I_1$, let us start by extending $Z_0$ to a function $\tilde Z_0$, defined on all $\Bbb Z$, such that $\tilde Z_0(z-1)-\mu_A \tilde Z_0(z)$ is an odd function. In the bounded interval case, we also require that $\tilde Z_0(N+1+z)-\mu_B\tilde Z_0(N+z)$ is an odd function. This extended $\tilde Z_0$ can be constructed inductively: first let $\tilde Z_0(-1) =\mu_AZ(0)$, then define $\tilde Z_0(-2) = \mu_A \tilde Z_0(-1) - \big( Z_0(0)-\mu_A Z_0(1) \big)$, and so on. In the bounded interval case, one would first construct $\tilde Z_0$ on $\{-N,...,-1\}$ and $\{N+1,...,2N\}$, then on $\{-2N,...,-N-1\}$ and $\{2N+1,...,3N\}$, etc.
		\\
		\\
		By Assumption $\ref{51}$, we know that $\|Z_0(z)-Z_0(w)\|_p \leq Ce^{a\epsilon(x+y)} (\epsilon|z-w|)^{\alpha}$ for $\alpha < 1/2$. In the half-line case, one may check (using the defining property of $\tilde Z_0$) that $\tilde Z_0$ will satisfy the same property with the same constant $a$. In the bounded interval case, $\tilde Z_0$ will satisfy this property on $\Bbb Z$, for \textit{some} large enough $a>0$ which does not depend on $\epsilon$ (see Lemma \ref{13}). By construction, it is true that $$\sum_{z \in \Lambda} \mathbf p_t^R(x,z) Z_0(z) = \sum_{z \in \Bbb Z} p_t(x-z)\tilde Z_0(z)$$ where the $p_t$ on the RHS is the standard (whole-line) heat kernel on $\Bbb Z$. Consequently, \begin{align*}\|I_1\|_p &\leq \sum_{z \in \Bbb Z} p_t(x-z) \big\|\tilde Z_0(z) -\tilde Z_0(z+y-x)\big\|_p \\ &\leq C\sum_{z \in \Bbb Z} p_t(x-z) (\epsilon |x-y|)^{\alpha} e^{a\epsilon(2z+y-x)} \\ &\leq C (\epsilon |x-y|)^{\alpha} e^{a \epsilon (x+y)} .
		\end{align*}
		In the final line, we used the fact that $\sum_z p_t(x-z)e^{2a\epsilon z} = e^{2a\epsilon x} e^{(\cosh (2a\epsilon ) - 1)t} \sim e^{2a\epsilon x} e^{2a^2\epsilon^2 t} \leq Ce^{2a\epsilon x}$, which is true because $e^{2a\epsilon z}$ is an eigenfunction of $\frac{1}{2}\Delta$ on $\Bbb Z$ with eigenvalue $\cosh (2a\epsilon)-1$. This proves that $I_1$ satisfies the desired bound.
		\\
		\\
		Next we need to bound $I_2$, so we are going to apply Lemma \ref{73} with $F(s,z) = \mathbf p_{t-s}^R(x,z)-\mathbf p_{t-s}^R(y,z)$. Note that if $|t-s-u| \leq 1$, then by the triangle inequality, Proposition \ref{3} (or \ref{14}), and Proposition \ref{33} (or \ref{26}), we have \begin{align*}(\mathbf p_u^R(x,z)-\mathbf p_u^R(y,z))^2 &\leq |\mathbf p_u^R(x,z)-\mathbf p_u^R(y,z)| (\mathbf p_u^R(x,z) + \mathbf p_u^R(y,z)) \\ & \leq C (1\wedge u^{-(1+\alpha)/2}) |x-y|^{\alpha} \cdot e^{1} (\mathbf p_{t-s+1}^R(x,z) + \mathbf p_{t-s+1}^R(y,z)) \\ & \leq C(t-s)^{-(1+\alpha)/2} |x-y|^{\alpha} \cdot (\mathbf p_{t-s+1}^R(x,z)+\mathbf p_{t-s+1}^R(y,z)). \end{align*}
		Hence $\sup_{|s-u| \leq 1} F_u(s,z)^2$ is bounded by the last expression. Thus by Lemma \ref{73} and Equation (\ref{77}) above, we find that for $t \geq 1$ and $\alpha \in [0,1)$ we have:
		\begin{align*}
		\|I_2\|_p^2 & \leq C \epsilon |x-y|^{\alpha} \int_0^t (t-s)^{-(1+\alpha)/2} \sum_z \big(\mathbf p_{t-s+1}^R(x,z)+\mathbf p_{t-s+1}^R(y,z)\big) \|Z_s(z)\|_p^2 ds \\ & \leq C\epsilon |x-y|^{\alpha} \int_0^t (t-s)^{-(1+\alpha)/2} \sum_z \big(\mathbf p_{t-s+1}^R(x,z)+\mathbf p_{t-s+1}^R(y,z)\big) e^{2a\epsilon z} ds \\ &= C\epsilon |x-y|^{\alpha} t^{(1-\alpha)/2} \big(e^{2a \epsilon x} + e^{2a\epsilon y}\big) \\ & \leq C (\epsilon |x-y|)^{\alpha} e^{2a\epsilon (x+y)}.
		\end{align*}
		Let us justify each of the inequalities above. In the second line we used (\ref{77}) to bound $\|Z_s(z)\|_p^2$ by $Ce^{2a\epsilon z}$. In the third line, we applied Corollaries \ref{4} and \ref{16} in order to bound the sum over $z$, and then we used the fact that $\int_0^t(t-s)^{-(1+\alpha)/2} ds = Ct^{(1-\alpha)/2}$. In the final line, we used the fact that $t^{(1-\alpha)/2} \leq C\epsilon^{\alpha-1}$ since $t \leq \epsilon^{-2}T$, and we also noted that $e^{2a\epsilon x} + e^{2a\epsilon y} \leq 2e^{2a\epsilon(x+y)}$. Taking square roots in the above expression, we find that for $\alpha<1$, $$\|I_2\|_p \leq C (\epsilon |x-y|)^{\alpha/2} e^{a\epsilon (x+y)}$$ which (after changing to macroscopic variables) completes the proof of the second bound (\ref{47}).
		\\
		\\
		Let us move onto the third estimate. Using Lemma \ref{74} and the semigroup property of the heat kernel, we can write 
		$$Z_t(x) = \sum_{y \geq 0} \mathbf p_{t-s}^R(x,y)Z_s(y) + \int_s^{t} \sum_{y \geq 0} \mathbf{p}_{t-u}^R(x,y)dM_u(y) $$ where $\mathbf{p}_t^R$ is the Robin heat kernel on $\Bbb Z_{\geq 0}$ with Robin boundary conditions. Therefore \begin{align*}Z_t(x)-Z_s(x) &= \sum_{y\geq 0} \mathbf{p}_{t-s}^R(x,y)\big( Z_s(y)-Z_s(x) \big) \\ &\;\;\;+ \bigg[ \sum_{y \geq 0} \textbf p_{t-s}^R(x,y)\;-\;1 \bigg]Z_s(x)+ \int_s^t \sum_{y \geq 0} \mathbf{p}_{t-u}^R(x,y) dM_u(y).\end{align*}
		Naming the terms on the RHS $J_1,J_2,J_3$ (in that order), we will prove the desired bound for each of the terms $J_i$.
		\\
		\\
		For $J_1$ we use the spatial $L^p$ bound (\ref{78}) at time $S=\epsilon s$, and the fact that $r^{\alpha} \leq e^r$ (since $\alpha< 1$) to obtain
		\begin{align}\label{44}
		\|J_1\|_p &\leq \sum_{y \geq 0} \mathbf p_{t-s}^R(x,y) \|Z_s(y)-Z_s(x)\|_p \nonumber
		\\ &\leq C \sum_{y \geq 0} \mathbf p_{t-s}^R(x,y) \cdot |\epsilon x-\epsilon y|^{\alpha}e^{a(\epsilon x+ \epsilon y)} \nonumber \\& \leq C \sum_{y \geq 0} \mathbf p_{t-s}^R(x,y) \cdot \epsilon^{\alpha} [1 \vee (t-s)^{\alpha/2}] e^{ |x-y| [1 \wedge (t-s)^{-1/2}]} e^{a(\epsilon x+\epsilon y)} \nonumber  \\& \leq C\epsilon^{\alpha} [1 \vee (t-s)^{\alpha/2}]e^{2a\epsilon x}.
		\end{align}
		where we used Corollary \ref{4} in the last inequality. Now bounding $J_2$, we can simply use Proposition \ref{7} together with (\ref{77}) to note that 
		\begin{align*}\|J_2\|_p &\leq \bigg| \sum_{y \geq 0} \mathbf p_{t-s}^R(x,y) \;-\;1 \bigg| \cdot \|Z_s(x)\|_p \\ &\leq C \epsilon^{\alpha}(t-s)^{\alpha/2} \cdot e^{a \epsilon x}
		\end{align*}
		where $C$ does not depend on $x$. As for $J_3$, we use Lemma \ref{73} and (\ref{76}) to obtain for $t-s \geq 1$, $$\bigg\| \int_s^t \sum_{y \geq 0} \mathbf p_{t-u}^R (x,y) dM_u(y) \bigg\|_p^2 \leq C\epsilon \int_s^t (t-u)^{-1/2}\sum_{y \geq 0}  \mathbf p_{t-u+1}^R (x,y) \|Z_u(y)\|_p^2 du. $$ Using (\ref{77}) and then Corollary \ref{4} we find that $$\sum_{y \geq 0} \mathbf p_{t-u+1}^R (x,y) \|Z_u(y)\|_p^2 \leq C \sum_{y \geq 0} \mathbf p_{t-u+1}^R(x,y) \cdot e^{2 a \epsilon y} \leq Ce^{2 a \epsilon x}.$$ Consequently \begin{align*}C\epsilon \int_s^t (t-u)^{-1/2}\sum_{y \geq 0}  \mathbf p_{t-u+1}^R (x,y) \|Z_u(y)\|_p^2 du &\leq C \epsilon e^{2a \epsilon x} \int_s^t (t-u)^{-1/2} du \\ &= C\epsilon (t-s)^{1/2} e^{2a \epsilon x}\\ &\leq C \epsilon^{2\alpha} (t-s)^{\alpha} e^{2a\epsilon x}.
		\end{align*}
		In the last line we used the fact that $2\alpha<1$ and that $t-s \leq \epsilon^{-2}\tau$, so that $\epsilon = \epsilon^{2\alpha}\epsilon^{1-2\alpha} \leq C \epsilon^{2\alpha} (t-s)^{\alpha-1/2}$. 
		\\
		\\
		To show tightness on $D([0,\tau], C(I))$ for $0<\delta \leq \tau$, these three estimates imply (respectively) that the $\{\mathcal Z^{\epsilon}\}_{\epsilon\in (0,1]}$ are uniformly bounded, uniformly spatially Hölder, and uniformly temporally Hölder (\textbf{except} for jumps of order $\epsilon^{\alpha}$) with large probability. Now we apply the version of Arzela-Ascoli for Skorokhod spaces (together with Prohorov's theorem) to obtain tightness, see [Bil97, Chapter 3] for the precise formulation of compactness in $D$.
		\\
		\\
		The fact that any limit point lies in $C([\delta,\tau],C(I))$ is a straightforward consequence of Kolmogorov's continuity criterion.
	\end{proof}
	
	\begin{defn}[Martingale Problem for the SHE]\label{57} Fix a terminal time $\tau \geq 0$. Let $\Bbb P$ be a probability measure on $\Omega:=C([0,\tau],C(I))$, and for $T \in [0,\tau]$, denote by $\mathcal L_T: \Omega \to C(I)$ the evaluation map at time $T$. Define $\mathscr T$ to be the set of all test functions $\varphi \in C_c^{\infty}(\Bbb R)$ such that $\varphi'(0)=A\varphi(0)$, and also $\varphi'(1)=B\varphi(1)$ if $I=[0,1]$. We say that $\Bbb P$ solves the martingale problem for the SHE with Robin boundary parameters $A$ and $B$ on $I$ if for every $\varphi \in \mathscr T$, the processes \begin{align*}Y_T(\varphi) &:= (\mathcal L_T, \varphi) - (\mathcal L_0, \varphi)- \int_0^T (\mathcal L_S, \varphi'') dS, \\ Q_T(\varphi) &:= Y_T(\varphi)^2 - \int_0^T \|\mathcal L_S\varphi\|_{L^2(I)}^2 dS\end{align*}
	are $\Bbb P$-local martingales (with respect to the filtration $\mathcal F_T:=\sigma( \{\mathcal L_S:S\leq T\})$). Here $(\psi,\varphi):=\int_I \psi \varphi $ denotes the $L^2(I)$ pairing.
	\end{defn}
	
	Here is the motivation behind the preceding definition. If $\mathcal L$ is the solution of the SHE, then formally, we have that $$\mathcal L_T - \mathcal L_0 -\frac{1}{2} \int_0^T \Delta \mathcal L_S dS = \int_0^T \mathcal L_SdW_S $$ for some cylindrical Wiener Process $(W_S)_{S \geq 0}$ over $L^2(I)$. Consequently, if we integrate both sides in the above expression against any test function $\varphi \in \mathscr T$, and use self-adjointness of the operator $\Delta$ with Robin boundary conditions on $I$, we should have that \begin{equation}\label{49}(\mathcal L_T, \varphi) - (\mathcal L_0, \varphi)- \int_0^T (\mathcal L_S, \varphi'') dS = \int_0^T (\mathcal L_S \varphi, dW_S )\end{equation} One may verify directly from the definition of mild solutions that this is indeed true. The RHS in this expression is clearly a martingale for any $\varphi$ (being an Itô integral against a cylindrical Wiener process). Furthermore by the Itô isometry, the quadratic variation of the martingale appearing on the RHS of (\ref{49}) is given by $$\int_0^T \|\mathcal L_S \varphi \|_{L^2(I)}^2dS$$ which shows that the expressions $Y_T(\varphi)$ and $Q_T(\varphi)$ appearing in Definition \ref{57} are indeed martingales when $\mathcal Z$ is a solution to the SHE.
	\\
	\\
	Conversely, if $\mathcal L$ is any $C(I)$-valued process (not necessarily the solution to the SHE) such that $Y_T(\varphi)$ and $Q_T(\varphi)$ are always local martingales, one might hope that (by some version of Levy's characterization theorem) it should be possible to construct a cylindrical Wiener Process $W$ such that (\ref{49}) holds true for every $\varphi$, and thus $\mathcal L$ solves the SHE in some weak (PDE) sense. This is indeed true, as proved in [CS16, Proposition 5.9]. We will reiterate the proof here because it will be needed later.
	
	\begin{prop}[Uniqueness of Solution to Martingale Problem]\label{58} Fix a terminal time $\tau$, and let $\Omega$ and $\mathcal L_T$ be as in Definition \ref{57}. Suppose that $\Bbb P$ is a solution to the martingale problem for the SHE. Then we may enlarge the probability space $\big(\Omega, \;\mathcal{B} orel, \;\Bbb P)$ so as to admit a space-time white noise $W$ such that $\Bbb P$-almost surely, for any $T\in[0,\tau]$ we have that $$\mathcal L_T(X) =  P_{T}*\mathcal L_0(X) + \int_0^T P_{T-S}*(\mathcal L_S dW_S)$$
		
	\end{prop}
	
	\begin{proof}
		Let us first illustrate this same principle in a one-dimensional setting. Let $b,\sigma: \Bbb R \to \Bbb R$, and suppose that $L_t$ is a process for which $$Y_t:= L_t - \int_0^t b(L_s)ds,$$ $$Q_t:= Y_t^2 - \int_0^t \sigma(L_s)^2ds$$ are both local martingales. Now suppose that we want to construct a Brownian motion $W_t$ such that $dL_t=b(L_t)dt+\sigma(L_t)dW_t$. The solution is then to define \begin{equation}\label{55}W_t:=\int_0^t \sigma^{-1}(L_s)1_{\{\sigma(L_s) \neq 0\}} dY_s+\int_0^t 1_{\{\sigma(L_s) = 0\}} d\overline W_s\end{equation} for an independent BM $\overline W$. It is then easily checked that $\langle W \rangle_t=t$ so that $W$ is a Brownian motion, and $W$ clearly satisfies the desired relation.
		\\
		\\
		The same general idea will apply in the infinite-dimensional setting, where $b(z) = \frac{1}{2} \Delta z$ and $\sigma(z)=z$. However, the formalisms needed to define integrals against the (now infinite-dimensional) process $Y_s$ become much more subtle, therefore we will use the notion of martingale measures as developed by Walsh, see for instance [Wal86, Chapter 2]. 
		\\
		\\
		Returning to the original statement, we define $Y_T(\varphi)$ and $Q_T(\varphi)$ to be the local martingales appearing in Definition \ref{57}. By localization, we are just going to assume that $Y$ and $Q$ are themselves martingales, with the understanding that we should technically choose a sequence of stopping times $\tau_n \to \infty$ and apply the usual tricks used in localization.
		\\
		\\
		By the assumption that $Q_T(\varphi)$ is a martingale, we have that $\langle Y(\varphi) \rangle_T = Y_T(\varphi)^2-Q_T(\varphi)=\int_0^T \|\mathcal L_S \varphi \|_{L^2(I)}^2 dS$ and thus by polarization we see that $$\langle Y(\varphi), Y(\psi) \rangle_T = \int_0^T ( \mathcal L_S \varphi, \mathcal L_S \psi) dS$$ so if $\varphi$ and $\psi$ have disjoint supports, then $\langle Y(\varphi), Y(\psi) \rangle=0$. Thus $Y$ is actually an \textit{orthogonal} martingale measure [Wal86, page 288], so we can define stochastic integrals of predictable space-time processes against it. In particular, we define $$W_T(\varphi):= \int_0^T \int_I \mathcal L_S(X)^{-1} \varphi(X) 1_{\{L_S(X) \neq 0\}} Y(dX\; dS) + \int_0^T \int_I 1_{\{\mathcal L_S(X)=0\}} \varphi(X) \overline W(dX\;dS)$$ where $\overline W$ is an independent space-time white noise over $[0,\tau) \times I$. Note the analogy between this formula and equation (\ref{55}) in the one-dimensional setting. Then one may check that $$\langle W(\varphi), W(\psi) \rangle_T = (\varphi, \psi) T$$ so by Levy's characterization theorem, it follows that $W$ is a space-time white noise over $[0, \tau] \times I$, and moreover, it is true by construction that for every $\varphi \in \mathscr T$, $$Y_T(\varphi) = \int_0^T (\mathcal L_S \varphi, dW_S)$$ which implies that \begin{align*}(\mathcal L_T,\varphi)&=(\mathcal L_0,\varphi)+\frac{1}{2}\int_0^T (\mathcal L_S,\varphi'')dS+Y_T(\varphi)\\& =(\mathcal L_0,\varphi)+\frac{1}{2}\int_0^T (\mathcal L_S,\varphi'')dS+ \int_0^T (\mathcal L_S \varphi, dW_S).\end{align*}
		Thus $(\mathcal L_T)$ is a weak solution, which (by Proposition \ref{103}) coincides with the mild solution.
	\end{proof}
	
	\begin{thm}[Main Theorem: Identification of Limit]\label{126} Assume that as $\epsilon \to 0$, the sequence of initial data $\mathcal Z_0^{\epsilon}$ converges weakly in $C(I)$ to some initial data $\mathcal Z_0$. Then, as $\epsilon \to 0$, the laws of the $\mathcal Z^{\epsilon}$ converge weakly in $D([0,T],C(I))$ to the law of the mild solution of the SHE with Robin boundary conditions, and initial data $\mathcal Z_0$.
	\end{thm}
	
	\begin{proof} The previous proposition shows that the laws of the $\mathcal Z^{\epsilon}$ do indeed converge (subsequentially) to some limiting measure as $\epsilon \to 0$. When $A,B \geq 0$ the identification of the limit is shown in [CS16] via Lemma 5.8, Lemma 5.7, Proposition 5.6, and Proposition 5.9 (in that logical order). When $A<0$ or $B<0$ the proof is very similar, but a small modification is needed in Lemma 5.8 (because it relies on a previous result, Proposition 5.1, which is only true for $A,B \geq 0$). Rather than trying to pinpoint the numerous little modifications which are needed, we will reproduce the whole proof here in generality, for the sake of completeness and clarity.
	\\
	\\
	By Proposition \ref{58}, it suffices to show that if $\Bbb P$ is a probability measure on $C([0,\tau],C(I))$ which is a limit point of the laws of the $\mathcal Z^{\epsilon}$, then the processes $$ Y_T(\varphi):= (\mathcal L_T, \varphi) - \frac{1}{2} \int_0^T (\mathcal L_S, \varphi'')dS$$ $$Q_T(\varphi):= Y_T(\varphi)^2 - \int_0^T (\mathcal L_T^2,\varphi^2)dS$$ are both $\Bbb P$-local martingales, where $\mathcal L_T:C([0,\tau],C(I)) \to C(I)$ denotes the canonical projection (e.g., evaluation at $T$). The main idea is to note that for any (fixed) $\epsilon>0$, if we define $$(\varphi,\psi)_{\epsilon} := \epsilon \sum_{x \in \Lambda} \varphi(\epsilon x) \psi(\epsilon x)$$then by Theorem \ref{74}, the process $$Z_t(x) - \frac{1}{2}\int_0^t\Delta Z_t(x)dt $$ is a martingale, so after summing against a test function and changing to macroscopic variables, the process $$Y_T^{\epsilon}(\varphi) := (\mathcal Z_T^{\epsilon}, \varphi)_{\epsilon} - \frac{1}{2} \int_0^{T} (\Delta_{\epsilon} \mathcal Z_S^{\epsilon}, \varphi)_{\epsilon} dS$$ is also martingale, where $$\Delta_{\epsilon} \varphi(X):= \epsilon^{-2}(\varphi(X+\epsilon)+\varphi(X-\epsilon)-2\varphi(X)).$$ 
	Recall that $\varphi \in \mathscr T$, which means that $\varphi'(0)=A\varphi(0)$. Using summation-by-parts and a Taylor series expansion of $\varphi$ near $X=0$, we have: \begin{align*}
	(\Delta_{\epsilon} &\mathcal Z_S^{\epsilon}, \varphi)_{\epsilon} - (\mathcal Z_S^{\epsilon}, \Delta_{\epsilon} \varphi)_{\epsilon} =  \epsilon^{-2} \cdot \epsilon \bigg[\mathcal Z_S^{\epsilon}(-\epsilon) \big(\varphi(0)-\varphi(-\epsilon)\big) - \varphi(-\epsilon) \big(\mathcal Z_S^{\epsilon}(0) - \mathcal Z_S^{\epsilon}(-\epsilon)\big) \bigg]\\ &= \epsilon^{-1} \bigg[ \mathcal Z_S^{\epsilon}(-\epsilon) \cdot \big( \epsilon A \varphi(0) + O(\epsilon^2) \big)- \varphi(-\epsilon) \cdot \epsilon A \mathcal Z_S^{\epsilon}(0) \bigg] \\ &= \epsilon^{-1} \bigg[   (1- \epsilon A) \mathcal Z_S^{\epsilon}(0) \cdot \big( \epsilon A \varphi(0) + O(\epsilon^2) \big)- \big( \varphi(0) - \epsilon A \varphi(0) +O(\epsilon^2) \big) \cdot \epsilon A \mathcal Z_S^{\epsilon}(0) \bigg] \\ &= \epsilon^{-1} \bigg[ \mathcal Z_S^{\epsilon}(0)\cdot O(\epsilon^2)\bigg]
	\end{align*}where the $O(\epsilon^2)$ term is a non-random quantity depending only on the test function $\varphi$. This computation was specific to the half-line case, but the bounded-interval case is similar (except that the number of boundary terms doubles since there are two boundary points instead of just one) and we get the same bound (except that there will be a $\mathcal Z_S^{\epsilon}(1)$ term appearing next to the $O(\epsilon^2)$ term as well).
	\\
	\\
	Summarizing, we have shown that \begin{equation}
	\label{112}\big| (\Delta_{\epsilon} \mathcal Z_S^{\epsilon}, \varphi)_{\epsilon} - (\mathcal Z_S^{\epsilon}, \Delta_{\epsilon} \varphi)_{\epsilon} \big| \leq C \epsilon \big| \mathcal Z_S^{\epsilon}(0) \big|\end{equation}and by (\ref{77}) the RHS tends to zero in $L^2(\Omega)$ as $\epsilon \to 0$. Now, by the fundamental theorem of calculus, one may easily check that $\Delta_{\epsilon} \varphi(X) = \int_X^{X+\epsilon} \int_{Z-\epsilon}^Z \varphi''(W)dWdZ$, and hence by uniform continuity of $\varphi''$, it follows that \begin{equation}\label{113} \lim_{\epsilon \to 0} \sup_{X \in \Bbb R} \big| \varphi''(X) - \Delta_{\epsilon} \varphi(X) \big| =0.\end{equation} Letting $F(\epsilon):= \sup_{X \in \Bbb R} \big| \varphi''(X) - \Delta_{\epsilon} \varphi(X) \big|$, we find that \begin{align}\label{114}\big\| (\mathcal Z_S^{\epsilon}, \Delta_{\epsilon} \varphi)_{\epsilon}- (\mathcal Z_S^{\epsilon},  \varphi'')_{\epsilon}\big\|_2 &\leq  \epsilon \sum_{X \in \epsilon \Lambda} \|Z_S^{\epsilon}(X)\|_2 \cdot F(\epsilon) \cdot 1_{\{ \varphi(X) >0\}} \nonumber\\ &\leq  \bigg(\epsilon \sum_{X \in \epsilon \Lambda} Ce^{aX} \cdot 1_{\{\varphi(X)>0\}} \bigg) \cdot F(\epsilon) \nonumber\\ &\leq C | supp(\varphi)| e^{a|supp(\varphi)|} F(\epsilon)\nonumber \\ &=C\cdot F(\epsilon)\end{align} where $|supp(\varphi)|$ denotes the supremum of the support of $\varphi$. In the second line we used (\ref{77}) to bound $\|Z_S^{\epsilon}(X)\|_2$ by $Ce^{aX}$, and in the third line we used compact support of $\varphi$.
	\\
	\\
	Using (\ref{112}), (\ref{113}), and (\ref{114}), it follows that $$\big\| (\Delta_{\epsilon} \mathcal Z_S^{\epsilon} , \varphi)_{\epsilon} - (\mathcal Z_S^{\epsilon}, \varphi'')_{\epsilon} \big\|_2\;\; \stackrel{ \epsilon \to 0}{\longrightarrow} \;\;0.$$ Here, as usual, we are denoting $\|X\|_p:= \Bbb E[|X|^p]^{1/p}$. Therefore, we can write for $T \leq \tau$: \begin{equation}\label{115}Y_T^{\epsilon}(\varphi) = (\mathcal Z_T^{\epsilon}, \varphi)_{\epsilon} - \frac{1}{2} \int_0^{T} ( \mathcal Z_S^{\epsilon}, \varphi'')_{\epsilon} dS + R_T^{\epsilon}(\varphi)\end{equation} where $\lim_{\epsilon \to 0} \|R_T^{\epsilon} \|_2 = 0$. Letting $\epsilon \to 0$ in this last expression, it follows that $Y_T(\varphi)$ is a $\Bbb P$-martingale for any limit point $\Bbb P$ of the law of $\mathcal Z^{\epsilon}$. Here we are using the fact that $(\mathcal Z_T^{\epsilon},\varphi)_{\epsilon}$ converges weakly along some subsequence to $(\mathcal L_T, \varphi)$ under $\Bbb P$, which can be seen using a Riemann sum approximation.
	\\
	\\
	Thus we only need to show that $Q_T(\varphi)$ is a $\Bbb P$-martingale (with respect to the canonical filtration on $C([0,\tau],C(I))$) for any limit point $\Bbb P$ of the law of $\mathcal Z^{\epsilon}$. This will be the the more difficult part of the proof (due to the extra term appearing in (\ref{75})), and it is really where the brunt of our efforts will go. The basic idea is as follows: We define $$Q_T^{\epsilon}(\varphi) := Y_T^{\epsilon}(\varphi)^2 - \langle Y^{\epsilon}(\varphi) \rangle_T$$ where $\langle Y^{\epsilon}(\varphi) \rangle$ denotes the predictable bracket of the martingale $Y_T^{\epsilon}(\varphi)$. Our goal will be to write $Q_T^{\epsilon}$ in a form similar to (\ref{115}), i.e., $$Q_T^{\epsilon}(\varphi) = Y_T^{\epsilon}(\varphi)^2 - \int_0^T( (\mathcal Z_S^{\epsilon})^2 , \varphi^2 )_{\epsilon} dS + R^{\epsilon}_T$$ where $R^{\epsilon}_T$ is some error term whose $L^2$ norm vanishes as $\epsilon \to 0$. So the remainder of the proof will be devoted to this goal. Indeed, using Equation (\ref{75}), it holds that $$\langle M \rangle_t = \int_0^t \bigg((\epsilon + o(\epsilon)) Z_s(x)^2 - \nabla^+ Z_s(x) \nabla^-Z_s(x)\bigg) ds.$$
	Consequently, we find that $$\langle Y^{\epsilon}(\varphi) \rangle_T = \epsilon^{-2} \int_0^T \bigg((\epsilon+o(\epsilon) )^2 ((\mathcal Z_S^{\epsilon})^2, \varphi^2)_{\epsilon} + \epsilon \cdot (\nabla^+_{\epsilon} \mathcal Z_S^{\epsilon} \cdot \nabla^-_{\epsilon} \mathcal Z_S^{\epsilon}\;,\; \varphi^2)_{\epsilon} \bigg)dS$$
	where $\nabla^{\pm}_{\epsilon} \varphi(X)=\varphi(X) - \varphi(X \pm \epsilon)$. Since $\epsilon^{-2} (\epsilon +o(\epsilon))^2 = 1+o(1)$, it is clear that $$\epsilon^{-2} \int_0^T (\epsilon+o(\epsilon) )^2 ((\mathcal Z_S^{\epsilon})^2, \varphi^2)_{\epsilon}dS$$ converges weakly to $\int_0^T (\mathcal L_S^2,\varphi^2)dS$ (under $\Bbb P$) as $\epsilon \to 0$ (subsequentially). Thus defining $$R^{\epsilon}_T(\varphi):= \epsilon^{-1} \int_0^T (\nabla^+_{\epsilon} \mathcal Z_S^{\epsilon} \cdot \nabla^-_{\epsilon} \mathcal Z_S^{\epsilon}\;,\; \varphi^2)_{\epsilon} dS$$the proof will be completed if we can show that $\Bbb E[ R_T^{\epsilon}(\varphi)^2] \to 0$ as $\epsilon \to 0$. Changing back to microscopic variables, we note that $$R^{\epsilon}_T(\varphi) = \epsilon^2 \int_0^{\epsilon^{-2}T} \sum_{x \in \Lambda} \nabla^+Z_{s}(x) \nabla^-Z_{s}(x) \varphi(\epsilon x)^2 ds.$$
	Using the identity $\big(\int_0^t f(s)ds\big)^2 =2 \int_0^t \int_0^s f(s)f(s')ds'ds$, we find that $$R^{\epsilon}_T(\varphi)^2 = 2 \epsilon^4 \int_0^{\epsilon^{-2}T} \int_0^s \sum_{x , y \in \Lambda} \varphi(\epsilon x)^2 \varphi(\epsilon y)^2 \nabla^+ Z_{s'}(x) \nabla^- Z_{s'}(x) \nabla^+ Z_{s}(y) \nabla^- Z_{s}(y) \;ds' ds .$$ Recall the filtration $\mathcal F_s = \sigma(\{\eta_s(x) : x\in \Lambda,\;s\leq t\})$, and define $$U^{\epsilon}(y,s',s):= \Bbb E\big[ \nabla^+ Z_s(y) \nabla^- Z_s(y) \big| \mathcal F_{s'} \big].$$Then we find that $$\Bbb E[R^{\epsilon}_T(\varphi)^2] = 2 \epsilon^4 \int_0^{\epsilon^{-2}T} \int_0^s \sum_{x,y \in \Lambda} \varphi(\epsilon x)^2 \varphi(\epsilon y)^2 \Bbb E\big[ \nabla^+ Z_s(x) \nabla^- Z_{s'}(x) U^{\epsilon}(y,s',s) \big] \;ds'ds.$$ Now since $Z_t(x) = e^{2 \sqrt{\epsilon} h_t(x) - \nu t}$, and since $|e^q-e^p| \leq |q-p|e^{q}$ for $p<q$, it follows that we have the ``brute-force" bound: $|\nabla^{\pm} Z_t(x)| \leq C\epsilon^{1/2} Z_t(x) $, hence the preceding expression yields \begin{equation}\label{117}\Bbb E[R^{\epsilon}_T(\varphi)^2] \leq C \epsilon^5 \int_0^{\epsilon^{-2}T} \int_0^s \sum_{x,y \in \Lambda} \varphi(\epsilon x)^2 \varphi(\epsilon y)^2 \Bbb E\big[ Z_{s'}(x)^2 U^{\epsilon}(y,s',s) \big]\;ds'ds.\end{equation}
	Now, we make the following claim, the proof of which we will postpone until a bit later: \begin{equation}\label{claim1} \sup_{x \in \Lambda} e^{-2a \epsilon x} \Bbb E|U^{\epsilon}(x,s,t)| \leq C\epsilon^{1/8} (t-s)^{-1/2},\;\;\;\;\;\; \epsilon^{-3/2} \leq s \leq t \leq \epsilon^{-2}T  .\tag{\textbf{Claim 1}}\end{equation} The $1/8$ appearing in the power is not sharp, it is just a convenient bound which suffices for our purposes. Before proving \eqref{claim1}, let us see why it implies the result. First we are going to split up the integral in \eqref{117} as three terms: \begin{equation}\label{119} \int_0^{\epsilon^{-3/2}} \int_0^s (-) ds'ds\;+\;\int_{\epsilon^{-3/2}}^{\epsilon^{-2}T} \int_0^{\epsilon^{-3/2}} (-)ds'ds \;+\; \int_{\epsilon^{-3/2}}^{\epsilon^{-2}T} \int_{\epsilon^{-3/2}}^s(-)ds'ds.\end{equation} 
	To deal with the first two terms, we use the brute-force bound $|\nabla^+ Z_s(y) \nabla^-Z_s(y)| \leq \epsilon Z_s(y)^2$, which in turn implies that the expectation appearing in \eqref{117} is bounded by $\epsilon \cdot \Bbb E[ Z_{s'}(x)^2 Z_s(y)^2]$. By Cauchy-Schwarz, this is in turn bounded by $\epsilon \cdot \|Z_{s'}(x)\|_4^2 \|Z_s(y)\|_4^2$, which by \eqref{77} is further bounded by $\epsilon \cdot C e^{2a(x+y)}$. Noting (for instance by Riemann sum approximation) that \begin{equation} \label{120}\epsilon^2 \sum_{x,y \in \Lambda} \varphi(\epsilon x)^2 \varphi(\epsilon y)^2 e^{Ca\epsilon(x+y)} \leq C,\end{equation} it follows that the first integral in \eqref{119} will be bounded by $C\epsilon$ and the second one will be bounded by $C \epsilon^{1/2}$, which is sufficient, since these quantities approach zero as $\epsilon \to 0$.
	\\
	\\
	In order to deal with the third integral in \eqref{119}, we define events $A_{K,s',x}:=\{ |Z_{s'}(x)| \leq K \}$, and we are going to split up the expectation in (\ref{117}) according to whether or not $A_{K,x,s'}$ occurs or not: \begin{equation}\label{118}\Bbb E\big[ Z_{s'}(x)^2\cdot 1_{A_{K,x,s'}} U^{\epsilon}(y,s',s) \big]\;\;\;+\;\;\;\Bbb E\big[ Z_{s'}(x)^2 \cdot 1_{A_{K,x,s'}^c} \nabla^+ Z_s(y) \nabla^- Z_s(y) \big]\end{equation} 
	where $A^c$ denotes the complement of $A$. The first term of \eqref{118} can be bounded by $K^2 \Bbb E[U^{\epsilon}(y,s',s)]$, which by \eqref{claim1} is bounded by $CK^2 \epsilon^{1/8} (t-s)^{-1/2}e^{a\epsilon y}$. Now for the second term. Since $1_{A^c_{K,x,s'}} \leq K^{-2} Z_{s'}(x,y)$, and since $\nabla^+Z_s(y) \nabla^-Z_s(y) \leq \epsilon Z_s(y)^2$, the second term of \eqref{118} is bounded by $K^{-2} \epsilon \Bbb E[ Z_{s'}(x)^4 Z_s(y)^2]$, which (by the Cauchy-Schwarz inequality) is in turn bounded by $K^{-2} \epsilon \|Z_{s'}(x)\|_8^4 \| Z_s(y) \|_4^2$. By \eqref{77}, this is further bounded by $K^{-2}\epsilon \cdot C e^{4a \epsilon x +2a\epsilon y}$. Then applying \eqref{120}, it will follow that the third term in \eqref{119} will be bounded by $C(\epsilon^{1/8}K^2 +K^{-2}) $. Letting $\epsilon \to 0$ and then $K \to \infty$, it finally follows that the RHS of \eqref{117} approaches zero as $\epsilon \to 0$.
	\\
	\\
	Thus all that is left to do is to prove \eqref{claim1}. This is done separately below. $\Box$
	\\
	\\
	\textit{Proof of \eqref{claim1}:} To prove the ``key estimate" (as [BG97] calls it) we are going to write: $$Z_t(x) = L_t(x) + D_t^t(x)$$ where $L$ stands for the solution to the \textbf{L}inear equation: $$L_t(x):= \mathbf p_t^R * Z_0(x)$$ and $D$ stands for the \textbf{D}uhamel contribution from the random noise: $$D^t_s(x) := \int_0^s \sum_y \mathbf p_{t-u}^R(x,y) dM_u(y).$$ We note that \begin{align*}\nabla^+ Z_t(x) \nabla^- Z_t(x) = &\nabla^+L_t(x) \nabla^- L_t(x) + \nabla^+D_t^t(x) \nabla^- D_t^t(x) \\&+ \nabla^+D_t^t(x) \nabla^- L_t(x) + \nabla^+ L_t(x) \nabla^- D_t^t(x).
	\end{align*}
	Noting that $D^t_s(x)$ is a martingale in $s$, it follows that $\nabla^{\pm} D_s^t(x)$ is a martingale in $s$, and hence $\nabla^+D_s^t(x) \nabla^- D_s^t(x) - \langle \nabla^+ D^t(x), \nabla^- D^t(x) \rangle_s$ is a martingale. Using the condition that $\langle M(x), M(y) \rangle_t = 0$ if $x \neq y$, it follows that $\langle \nabla^+ D^t(x), \nabla^- D^t(x) \rangle_s = \int_0^s \sum_y K_{t-u}(x,y;A) d\langle M(y) \rangle_u$, where $$K_t(x,y;A) = \nabla^+\mathbf p_t^R(x,y;A) \nabla^- \mathbf p_t^R(x,y;A).$$The reason we specify $A$ here but not elsewhere will be made clear below. Summarizing the last paragraph, we find that if $s \leq r \leq t$ then $$\Bbb E \big[ \nabla^+D_r^t(x) \nabla^- D_r^t(x) \; \big|\; \mathcal F_s \big] = \nabla^+D_s^t(x) \nabla^- D_s^t(x) + \Bbb E \bigg[ \int_s^r K_{t-u}(x,y;A) d\langle M(y) \rangle_u \; \bigg|\; \mathcal F_s \bigg].$$ Hence \begin{align*}U^{\epsilon}(x,s,t) = &\nabla^+L_t(x) \nabla^-L_t(x) + \nabla^+ D_s^t(x) \nabla^-D_s^t(x) + \Bbb E \bigg[ \int_s^t K_{t-u}(x,y;A) d\langle M(y) \rangle_u \; \bigg|\; \mathcal F_s \bigg] \\ &+ \nabla^+ D_s^t(x) \nabla^- L_t(x) + \nabla^+ L_t(x) \nabla^- D^t_s(x) .\end{align*}
	Let us call the terms on the RHS as $I_1,...,I_5$, respectively. In order to bound the expectations of $I_1,I_2,I_4$, and $I_5$, it suffices (by the Cauchy-Schwarz inequality) to bound both $\Bbb E[ (\nabla^{\pm} L_t(x) )^2 ]$ and $\Bbb E[(\nabla^{\pm} D_s^t(x) )^2]$ by $C \epsilon^{1/8}(t-s)^{-1/2} e^{2a\epsilon x}$, where $C$ does not depend on $x,\epsilon,$ or $s,t \in [\epsilon^{-3/2}, \epsilon^{-2}T]$. The bound for $I_3$ is more involved, and will be done later.
	\\
	\\
	Let us start be getting the desired bound on $\Bbb E[(\nabla^{\pm} L_t(x))^2]$. Using Assumption \ref{51} and then Corollary \ref{4} or \ref{16}, we have that \begin{align*}
	\|\nabla^{\pm} L_t(x)\|_2 &\leq \sum_y | \nabla^{\pm} \mathbf p_t^R(x,y)| \cdot \|Z_0(y)\|_2 \\ &\leq C \sum_y| \nabla^{\pm} \mathbf p_t^R(x,y)| e^{a\epsilon y} \\ &\leq Ct^{-1/2} e^{a\epsilon x}.
	\end{align*}
	Squaring both sides, we get $$\Bbb E[(\nabla^{\pm} L_t(x))^2] \leq Ct^{-1} e^{2a\epsilon x}.$$
	Using the assumption that $t \geq \epsilon^{-3/2}$, we have that $t^{-1/2} \leq \epsilon^{3/4}$, so that $t^{-1} \leq \epsilon^{3/4}t^{-1/2} \leq \epsilon^{3/4} (t-s)^{-1/2} $. Noting that $\epsilon^{3/4} \leq \epsilon^{1/8}$ when $\epsilon \leq 1$, the desired bound for $\Bbb E[(\nabla^{\pm} L_t(x))^2]$ follows.
	\\
	\\
	Next, we are going to bound $\Bbb E[(\nabla^{\pm} D^t_s(x))^2 ]$. First note that $\| \frac{d}{du} \langle M(y) \rangle_u \|_p \leq C\epsilon e^{2a\epsilon y}$, where $C$ does not depend on $y$. This follows by \eqref{75}, the ``brute-force" bound $|\nabla^+Z_u(x) \nabla^- Z_u(x)| \leq C\epsilon Z_u(x)^2$, and \eqref{77}. Note that \begin{align*}\Bbb E[(\nabla^{\pm}D_s^t(x))^2] &= \Bbb E[ \langle \nabla^{\pm} D^t(x) \rangle_s ] \\&= \Bbb E \bigg[ \int_0^s \sum_y (\nabla^{\pm} \mathbf p_{t-u}^R(x,y))^2 d\langle M(y) \rangle_u \bigg] \\ &\leq \int_0^s \sum_y (\nabla^{\pm} \mathbf p_{t-u}^R(x,y))^2 \cdot \Bbb E \big| \frac{d}{du} \langle M(y) \rangle _u \big| du\\&\leq C \epsilon \int_0^s (t-u)^{-1} \sum_y \nabla^{\pm} \mathbf p_{t-u}^R(x,y) e^{2a\epsilon y} du \\ &\leq C\epsilon e^{2a\epsilon x} \int_0^s (t-u)^{-3/2}du \\ &\leq C \epsilon e^{2a\epsilon x} (t-s)^{-1/2}.  \end{align*}We used Proposition \ref{3} or \ref{15} in the fourth line, and Corollary \ref{4} or \ref{16} in the fifth line. In the last line we just performed the integral and omitted the subtracted term. The desired claim now follows by noting that $\epsilon \leq \epsilon^{1/8}$ for $\epsilon \leq 1$. This completes the proof of desired bounds for $I_1,I_2,I_4,$ and $I_5$.
	\\
	\\
	Now we just need to bound $I_3$, which was defined as $\Bbb E \big[ \int_s^t \sum_y K_{t-u}(x,y;A) d\langle M(y) \rangle_u \big| \mathcal F_s \big]$. Using Equation \eqref{75}, we can expand $d\langle M(y) \rangle_u$ into $(\epsilon +o(\epsilon)) Z_u(y)^2du + \nabla^+ Z_u(y) \nabla^-Z_u(y)du$. Consequently, we find that \begin{equation}\label{122}|I_3| \leq  C\epsilon \bigg| \int_s^t \sum_y K_{t-u}(x,y;A) \Bbb E[ Z_u(y)^2 |\mathcal F_s] du \bigg|+ \int_s^t \sum_y |K_{t-u}(x,y;A)| |U^{\epsilon} (y,s,u)| du . \end{equation}
	Now we make a claim which we hold off until later: \begin{equation}\label{claim2}
	    \Bbb E \bigg[\epsilon \bigg| \int_s^t \sum_y K_{t-u}(x,y;A) \Bbb E[ Z_u(y)^2 |\mathcal F_s] du \bigg| \bigg] \leq C \epsilon^{1/8} e^{2a\epsilon x} (t-s)^{-1/2}. \tag{\textbf{Claim 2}}
	\end{equation}
	Before proving this, let us first see why it implies \eqref{claim1}. Using the bounds for $I_1,I_2,I_4$, and $I_5$ together with \eqref{122} and \eqref{claim2} shows that \begin{equation}\label{123}\Bbb E|U^{\epsilon}(x,s,t)| \leq C\epsilon^{1/8}e^{2a\epsilon x} (t-s)^{-1/2} +\int_s^t \sum_y |K_{t-u}(x,y;A)| \Bbb E|U^{\epsilon} (y,s,u)| du .\end{equation}
	Now we can iterate this, and we get that $$\Bbb E|U^{\epsilon}(x,s,t)| \leq C\epsilon^{1/8} e^{2a\epsilon x} (t-s)^{-1/2} $$$$+ C\epsilon^{1/8}\sum_{n=1}^{\infty} \int_{\Delta_n(s,t)} (u_n-s)^{-1/2} \sum_{y_1,...,y_n} |K_{t-u_1}(x,y_1;A)| \prod_{i=1}^{n-1} |K_{u_{i}-u_{i+1}}(y_i,y_{i+1};A)| e^{2a\epsilon y_n} du_n \cdots du_1  $$
	where $\Delta_n(s,t) := \{(u_1,...,u_n) : s \leq u_n \leq \cdots \leq u_1 \leq t\}$. Now note by the triangle inequality that $e^{2a\epsilon y_n} \leq e^{2a\epsilon x} e^{2a \epsilon |y_1-x|}\prod_{i=1}^n e^{2a \epsilon |y_{i+1}-y_i|}$. Now using Propositions \ref{10} and \ref{11} (or \ref{25}), one may check that the integral over $\Delta_n(s,t)$ is at most $C \epsilon \cdot c_*^n\cdot e^{2a\epsilon x}$, where $c_*<1$ is a constant independent of $x$ and $\epsilon$. Using the fact that $\epsilon \leq C(t-s)^{-1/2}$, the desired claim now follows easily. $\Box$
	\\
	\\
	\textit{Proof of \eqref{claim2}:} This will be the final part of the proof. We will prove that $$\Bbb E \bigg[ \epsilon \bigg| \sum_{y \in \Lambda} \int_s^t K_{t-u} (x,y;A) \Bbb E[Z_{u}(y)^2 | \mathcal F_s] du \bigg| \bigg] \leq C \epsilon^{1/8} e^{2a \epsilon x} (t-s)^{-1/2} $$ uniformly in $x\in \Bbb Z_{\geq 0}$ and $s,t\in [0,\epsilon^{-2}\tau]$, where $$K_t(x,y;A) = \nabla^+\mathbf p_t^R(x,y;A) \nabla^- \mathbf p_t^R(x,y;A)$$ and $\mathbf p_t^R(\cdot, \cdot;A)$ is the Robin heat kernel with parameter $\mu_A$, and $\mathcal F_t= \sigma(\eta_s(x):s \leq t, x \in \Lambda)$.
	\\
	\\
	In the special case of ASEP-H, the same result holds with $\epsilon^{1/8}$ improved to $\epsilon^{1/2-\delta}$. This can be seen by dissecting the proof below, and it is mainly because of the difference between the methods used to prove Lemmas \ref{9} and \ref{24}.
	\\
	\\
	In order to prove the above inequality, we write 
	\begin{align*}\epsilon \sum_{y \in \Lambda}\int_s^t |K_{t-u} (x,y;A)| \Bbb E[Z_{u}(y)^2 | \mathcal F_s] du &= \epsilon \sum_{y \in \Lambda}\int_s^t |K_{t-u} (x,y;A)| \Bbb E[Z_{u}(y)^2 - Z_t(x)^2| \mathcal F_s] du \\ &\;\;\; + \epsilon \Bbb E[Z_{t}(x)^2 | \mathcal F_s] \sum_{y \in \Lambda}\int_s^t |K_{t-u}(x,y;A) -K_{t-u}(x,y;0)| du \\ &\;\;\;+ \epsilon \Bbb E[Z_{t}(x)^2 | \mathcal F_s] \sum_{y \in \Lambda}\int_s^t |K_{t-u}(x,y;0)| du.
	\end{align*}
	Let us name the three terms on the RHS as $J_1, J_2, J_3$, in that specific order. To bound $J_1$, we write $Z_u(y)^2-Z_t(x)^2=(Z_u(y)-Z_t(x)) (Z_u(y)+Z_t(x))$, so by Cauchy-Schwarz we have that $\|Z_u(y)^2-Z_t(x)^2\|_1 \leq \|Z_u(y)+Z_t(x)\|_2 \|Z_u(y)-Z_t(x)\|_2$. Using \eqref{77}, we have that $\|Z_u(y)+Z_t(x)\|_2 \leq Ce^{a\epsilon x}$. Using \eqref{78} and \eqref{79}, we have that for $\alpha<1/2$ \begin{align*}\|Z_u(y)-Z_t(x)\|_2 &\leq \|Z_u(y)-Z_u(x)\|_2 + \|Z_u(x)-Z_t(x)\|_2 \\ &\leq C(\epsilon |x-y|)^{\alpha}e^{a\epsilon (x+y)} + C\epsilon^{\alpha} (1 \vee |t-u|^{\alpha/2} ) e^{2a\epsilon x} .\end{align*} Let us take $\alpha = 1/8$ for this proof. Note from Proposition \ref{3} (or \ref{15}) that $ |K_t(x,y;A)| \leq (1 \wedge t^{-1}) |\nabla^+ \mathbf p_t^R(x,y)|$. Now we make a few observations: firstly, note that $(1 \wedge t^{-1})(1 \vee t^{\alpha/2}) = (1\wedge t^{-(2-\alpha)/2})$. Secondly, by the identity $r^{\alpha} \leq e^r$, we find that $(\epsilon |x-y|)^{\alpha} \leq \epsilon^{\alpha} (1 \vee t^{\alpha/2}) e^{-(1 \wedge t^{-1/2}) |x-y|}$. Combining all of these observations, one finds that \begin{align*}
	    \Bbb E[|J_1|] &\leq \epsilon \cdot C \epsilon^{\alpha} \int_s^t (1 \wedge (t-u)^{-(2-\alpha)/2}) \sum_{y \in \Lambda} |\nabla^+ \mathbf p_{t-u}^R(x,y)| \big[ e^{a\epsilon (x+y)} e^{-(1 \wedge (t-u)^{-1/2}) |x-y|} +e^{2a\epsilon x} \big] \\ &\leq C \epsilon \cdot \epsilon^{1/8} \int_0^t (1 \wedge (t-u)^{-3/2}) e^{2a\epsilon x} du \\ &\leq C (t-s)^{-1/2} \cdot \epsilon^{1/8} e^{2a\epsilon x} \int_0^{\infty} (1 \wedge u^{-3/2}) du.
	\end{align*}
	In the second line we substituted $\alpha=1/8$, and we used Corollary \ref{4} (or \ref{16}) to bound the sum over $y$. In the next line we noted that $\epsilon \leq C(t-s)^{-1/2}$ since $t -s\leq \epsilon^{-2}\tau$. In the final line we made the substitution $u \to t-u$. This proves that $J_1$ satisfies the desired bound.
	\\
	\\
	To show that $J_2$ satisfies a bound of the desired type, we apply Equation \eqref{77} together with Lemma \ref{9} or \ref{24} to say that
	\begin{align*}
	\Bbb E[|J_2|] &\leq \epsilon \|Z_t(x)\|_2^2 \sum_{y \geq 0} \int_0^{\epsilon^{-2}\tau} |K_t(x,y;A)-K_t(x,y;0)| dt \\ & \leq \epsilon \cdot C e^{2a \epsilon x} \cdot \epsilon^{1/8} \\ & \leq C(t-s)^{-1/2} e^{2a \epsilon x} \epsilon^{1/8}
	\end{align*}
	where we used the fact that $t-s \leq \epsilon^{-2}\tau$.
	\\
	\\
	As for $J_3$, we can split it into two further terms $$J_3 = \epsilon \Bbb E[Z_t(x)^2|\mathcal F_s]\bigg(\sum_{y \geq 0} \int_0^{\infty} K_{u}(x,y;0)du - \sum_{y \geq 0} \int_{t-s}^{\infty} K_{u}(x,y;0)du \bigg).$$ As a consequence of Proposition \ref{108} or \ref{107}, the first term satisfies $$\epsilon \Bbb E[Z_t(x)^2| \mathcal F_s] \sum_{y \geq 0} \int_0^{\infty} K_{\tau}(x,y;0)d\tau =O(\epsilon^2) \Bbb E[Z_t(x)^2| \mathcal F_s].$$ After taking expectations, this will be bounded by $C\epsilon^2 \|Z_t(x)\|_2^2$, which by \eqref{77} (and the fact that $t-s \leq \epsilon^{-2}\tau$) is bounded by $C\epsilon (t-s)^{-1/2} e^{2a\epsilon x}$. To bound the second term on the RHS, note by Proposition \ref{3} (or \ref{15}) and Corollary \ref{4} (or \ref{16}) that $\sum_y |K_u(x,y)| \leq 1 \wedge u^{-3/2}$. After integrating from $t-s$ to $\infty$, this will be bounded by $C(t-s)^{-1/2}$, which completes the proof.
    \end{proof}

	\section{Extension of Results to Narrow-Wedge Initial Data}
	
	In this section we will consider the weakly asymmetric limit of the height functions for Open ASEP started from an initial configuration which has zero particles. There are several reasons why such initial data poses a problem. The first problem is as follows:
	
	\begin{itemize}
		\item The associated sequence of initial data for the rescaled Gartner-transformed height functions is $\mathcal Z_0^{\epsilon}(X) = e^{-\epsilon^{-1/2}X} $. Note that this converges weakly (in the PDE sense) to $0$ as $\epsilon \to 0$, therefore one may expect that $\mathcal Z^{\epsilon}(T,X)$ will just converge almost surely to zero as $\epsilon \to 0$, which is indeed the case. So the limit is trivial.
	\end{itemize}
	
	Hence it is clear that some sort of normalization is necessary in order to obtain a nontrivial limiting object. The solution to this problem is to introduce a logarithmically diverging correction to the height function, as in [ACQ11] or [Cor12]. More specifically, we want our Gartner-transformed initial data $\mathcal Z_0^{\epsilon}(X) = e^{-\epsilon^{-1/2}X} $ to converge to something nontrivial as $\epsilon \to 0$, and the only sensible way to do this is to multiply by a factor of $\epsilon^{-1/2}$ since that will make $\mathcal Z_0^{\epsilon}$ converge weakly (in the PDE sense) to $\delta_0$ as $\epsilon \to 0$. Hence we will \textit{redefine} $Z_t(x)$ and $\mathcal Z^{\epsilon}(T,X)$ as the following quantities: $$Z_t(x) = \frac{\varrho}{\epsilon^{1/2}} e^{\lambda h_t(x) + \nu t}$$ $$\mathcal Z^{\epsilon}(T,X) = Z_{\epsilon^{-2}T}(\epsilon^{-1}X)$$ where $\varrho:= \big(\int_I e^{-Z}dZ\big)^{-1}$. So $\varrho=1$ if $I=[0,\infty)$, and $\varrho = 1/(1-e^{-1})$ if $I=[0,1]$. The choice for this constant will be made clear later (see Lemma \ref{56} below).
	\\
	\\
	We hope that it is clear that this is \textbf{not} the same $Z_t(x)$ and $\mathcal Z^{\epsilon}(T,X)$ appearing in Sections 2 and 4, due to the normalizing factor which is $\epsilon^{-1/2}\varrho$.
	\\
	\\
	Although this normalization scheme should presumably give us a nontrivial object in the limit, the proof involves some subtleties. The main issue is that
	\begin{itemize}
		\item The associated sequence of initial data for the (redefined) Gartner-transformed height functions $\mathcal Z_0^{\epsilon}(T,X)=\frac{\varrho}{\epsilon^{1/2}} e^{-\epsilon^{-1/2}X}$ is no longer ``near-equlilbrium," as defined in Assumption \ref{51}. Therefore, the results of Section 4 no longer apply in our case.
	\end{itemize}
	The solution to this issue involves quite a few subtleties, and therefore we will devote this section to proving convergence of $\mathcal Z^{\epsilon}$ to the SHE in this case.
	\\
	\\
	Recall that $\|X\|_p = \Bbb E[|X|^p]^{1/p}$.
	
	\begin{lem}\label{34}
		For all $T \geq 0$, there exists $C=C(T)$ such that for every $\epsilon>0$ sufficiently small $$\sup_{t\leq T} \Bbb \|Z_t(x)\|_p \leq C Z_0(x).$$
	\end{lem}
	
	\begin{proof} Since we are starting from an empty configuration of particles and since the exponential jump rate satisfies $p \sim \frac{1}{2}+O(\sqrt \epsilon) \leq 1$, we know that the position of the largest occupied site at time $t$ is stochastically dominated by a Poisson random variable $N(t)$, with mean $t$. Therefore $h_t(x)$ is stochastically dominated by $N(t)+h_0(x)$. Therefore for all $x \in \Lambda$ and $t \leq T$, we have that $$Z_t(x) = \epsilon^{-1/2}e^{\epsilon^{1/2}h_t(x)+\nu t} \leq \epsilon^{-1/2}e^{\epsilon^{1/2} N(t)}e^{\epsilon^{1/2}h_0(x)+\nu t} = e^{\epsilon^{1/2} N(t)}Z_0(x)$$ and consequently $$\|Z_t(x)\|_p \leq \|e^{\epsilon^{1/2}N(t)}\|_p Z_0(x) = e^{\epsilon^{1/2}(e^{pt}-1)} Z_0(x)$$ from which one may deduce the claim. \end{proof}
	
	\begin{prop}\label{36} Fix a terminal time $\tau \geq 0$. Let $\alpha \in [0,1/2)$. We have the following bounds, uniformly over all (small enough) $\epsilon>0$, $x,y \in \Lambda$, and $s,t \in [1, \epsilon^{-2}\tau]$ with $s<t$:
		\begin{equation}\label{46}\|Z_t(x)\|_p \leq C (\epsilon^2t)^{-1/2},\;\;\;\;\;\;\;\;\;\;\;\;\;\;\;\;\;\;\;\;\end{equation}
		\begin{equation}\label{47}\|Z_t(x)-Z_t(y)\|_p \leq C(\epsilon |x-y|)^{\alpha} (\epsilon^2 t)^{-(1+\alpha)/2},\;\;\;\;\;\;\;\;\;\;\end{equation}
		\begin{equation}\label{48}\|Z_t(x)-Z_s(x)\|_p \leq C \epsilon^{\alpha} (1 \vee |t-s|) ^{\alpha/2}(\epsilon^2s)^{-(1+\alpha)/2}.\end{equation}
		
	\end{prop}
	
	\begin{proof} The proof here will be loosely based on the one given in [CST16, Proposition 1.8], however we found a small mistake there (the same one mentioned in the proof of Proposition \ref{85}) so several new ideas will also be used. These will involve the long-time estimates, Propositions \ref{lt1} and \ref{lt2}.
		\\
		\\
		We have from Theorem 2.9 that 
		\begin{align*}
		Z_t(x) &= \mathbf p_t^R * Z_0 (x) + \int_0^t \sum_{y \geq 0} \mathbf p_{t-s}^{R} (x,y) dM_s(y)
		\end{align*}
		
		Using Lemma \ref{73} and (\ref{76}), and we obtain for $t \geq 1$ that $$\bigg\| \int_0^t \sum_{y \geq 0} \mathbf p_s^R(x,y)dM_s(y) \bigg\|_p^2 \leq C \epsilon \int_0^t (t-s)^{-1/2}\sum_{y \geq 0} \mathbf p_{t-s+1}^R(x,y) \|Z_s(y)\|_p^2 ds.$$ 
		
		Since $(x+y)^2 \leq 2x^2+2y^2$, we have \begin{align*}
		\|Z_t(x)\|_p^2 \leq 2 \big|\mathbf p_t * Z_0(x)\big|^2 + C \epsilon \int_0^t (t-s)^{-1/2} \sum_{y \geq 0}  \mathbf p_{t-s+1}^R(x,y) \|Z_s(y)\|_p^2 ds.
		\end{align*}
		
		Using $\mathbf p_t^R(x,y) \leq Ct^{-1/2}$ together with $\sum_{y \geq 0} e^{-\epsilon^{1/2}y} =(1-e^{-\epsilon^{1/2}})^{-1} \sim \frac{1}{2}\epsilon^{-1/2}$ implies that \begin{align}\label{38}\mathbf p_t^R * Z_0(x) = \varrho \epsilon^{-1/2} \sum_y \mathbf p_t^R(x,y) e^{-\epsilon^{1/2}y} \leq C \epsilon^{-1} t^{-1/2} =C(\epsilon^2t)^{-1/2}\end{align} which in turn implies that $|\mathbf p_t^R * Z_0(x)|^2 \leq C (\epsilon^2 t)^{-1/2} (\mathbf p_t*Z_0)(x)$.
		\\
		\\
		Summarizing the above computations, we have proved that for all $t \in [1, \epsilon^{-2}T]$, we have \begin{equation}\label{41}\|Z_t(x)\|_p^2 \leq C(\epsilon^2 t)^{-1/2} (\mathbf p_t^R *Z_0)(x) + C \epsilon \int_0^t (t-s)^{-1/2} \sum_{y \geq 0}  \mathbf p_{t-s+1}^R(x,y) \|Z_s(y)\|_p^2 ds.\end{equation}
		Now we would like to iterate this inequality, but the problem is that we have only proved (\ref{41}) for $t \geq 1$, however we need to prove it for all $t \geq 0$ in order to apply the iteration. By Lemma \ref{34} and Proposition \ref{33} (or \ref{26}), for $t \leq 1$ we have \begin{equation}\label{37}\|Z_t(x)\|_p \leq C Z_0(x) \leq Ce^{t}\mathbf p_t(x,x) Z_0(x) \leq Ce^{1} \sum_{y} \mathbf p_t(x,y) Z_0(y) =C (\mathbf p_t *Z_0)(x) .\end{equation} Now squaring the LHS and RHS of (\ref{37}), and then applying (\ref{38}), we obtain that $\|Z_t(x)\|_p^2 \leq C(\epsilon^2 t)^{-1/2} (\mathbf p_t^R * Z_0)(x) $ for $t \leq 1$. This proves that (\ref{41}) still holds when $t\leq 1$, thus justifying our ability to iterate it.
		\\
		\\
		Iterating (\ref{41}) and applying the semigroup property of $\mathbf p_t^R$ proves that \begin{align}\label{97}\|Z_t(x)\|_p^2 &\leq C(\epsilon^2 t)^{-1/2}(\mathbf p_t^R*Z_0)(x)\\& \;\;\;\;+ \sum_{k=0}^{\infty} C^k\epsilon^k \bigg(\int_{\Delta_k(t)} t_0^{-1/2} \prod_{j=1}^k (t_j-t_{j-1})^{-1/2}(t-t_k)^{-1/2} dt_0\cdots dt_k \bigg) (\mathbf p_{t+k}^R*Z_0)(x)\nonumber\end{align} where $\Delta_k(t) = \{(t_0,...,t_k) \in \Bbb R^{k+1}\; |\; t_0 < \cdots <t_k<t\}$. A recursion will reveal that the integral within the parentheses is bounded $Ct^{k/2}/(k/2)!$. Recall the long-time estimates (Proposition \ref{lt1} or \ref{lt2}) which show that $$\mathbf p_t^R(x,y) \leq C(t^{-1/2}+\epsilon)e^{K\epsilon^2t}$$where $C,K$ are constants \textbf{not} depending on the terminal time $\tau$. This in turn implies that \begin{equation}\label{lt3}\mathbf p_t^R * Z_0(x) \leq C(\epsilon^{-1}t^{-1/2}+1)e^{K\epsilon^2t}\end{equation}Using (\ref{97}) and the remark underneath, we then see that \begin{align}\label{98}
		\|Z_t(x)\|_p^2 &\leq C(\epsilon^2t)^{-1/2}(\mathbf p_t^R * Z_0)(x)+C\sum_{k=0}^{\infty} \frac{C^k \epsilon^k t^{k/2}}{(k/2)!} (\epsilon^{-1}t^{-1/2}+1)e^{K\epsilon^2(t+k)}\nonumber \\ &\leq C(\epsilon^2t)^{-1/2} (\mathbf p_t^R * Z_0)(x) + C\big[(\epsilon^2t)^{-1/2}+1\big]e^{K'\epsilon^2t}\nonumber \\ &\leq C(\epsilon^2t)^{-1/2} \big[(\mathbf p_t^R * Z_0)(x)+1 \big].\end{align} In the final inequality we used the fact that $\epsilon^2t \leq \tau$, and $1 \leq \tau^{1/2}(\epsilon^2t)^{-1/2}$. The proof of the first bound (\ref{46}) is completed by noting from (\ref{38}) that $(\mathbf p_t^R * Z_0)(x) \leq C(\epsilon^2t)^{-1/2}$, so that (\ref{98}) is bounded by $C(\epsilon^2t)^{-1}$.
		\\
		\\
		To prove the second bound (\ref{47}), note that $$\|Z_t(x)-Z_t(y)\|_p \leq |\mathbf p_t^R * Z_0(x) - \mathbf p_t^R * Z_0(y) | + \bigg\| \int_0^t \sum_z \big(\mathbf p_t^R(x,z)-\mathbf p_t^R(y,z)\big) dM_s(z) \bigg\|_p .$$ Let us call the terms on the RHS as $J_1,J_2$, respectively. We will show that each $J_i$ satisfies a bound of the desired type. For $J_1$, note by Proposition \ref{3} (or \ref{15}) that $$J_1 \leq \epsilon^{-1/2} \sum_z \big|\mathbf p_t^R(x,z)-\mathbf p_t^R(y,z)\big|e^{-\epsilon^{1/2}z} \leq C \epsilon^{-1/2} t^{-(1+\alpha)/2} |x-y|^{\alpha} \sum_z e^{-\epsilon^{1/2}z} $$$$ \leq C \epsilon^{-1} t^{-(1+\alpha)/2} |x-y|^{\alpha} = C (\epsilon^2 t)^{-(1+\alpha)/2} (\epsilon |x-y|)^{\alpha}.$$ 
		For $J_2$, we are going to apply Lemma \ref{73} with $F(s,z) = \mathbf p_{t-s}^R(x,z)-\mathbf p_{t-s}^R(y,z)$. Note that if $|t-s-u| \leq 1$, then by the triangle inequality, Proposition \ref{3} (or \ref{15}), and Proposition \ref{33} (or \ref{26}), we have \begin{align*}(\mathbf p_u^R(x,z)-\mathbf p_u^R(y,z))^2 &\leq |\mathbf p_u^R(x,z)-\mathbf p_u^R(y,z)| (\mathbf p_u^R(x,z) + \mathbf p_u^R(y,z)) \\ & \leq C (1\wedge u^{-(1+\alpha)/2}) |x-y|^{\alpha} \cdot e^{1} (\mathbf p_{t-s+1}^R(x,z) + \mathbf p_{t-s+1}^R(y,z)) \\ & \leq C(t-s)^{-(1+\alpha)/2} |x-y|^{\alpha} \cdot (\mathbf p_{t-s+1}^R(x,z)+\mathbf p_{t-s+1}^R(y,z)).
		\end{align*}
		Hence $\sup_{|s-u| \leq 1} F_u(s,z)^2$ is bounded by the last expression. Thus by Lemma \ref{73} and Equation (\ref{98}) above, we find that for $t \geq 1$ and $\alpha \in [0,1)$ we have:
		\begin{align*}
		J_2^2 & \leq C \epsilon |x-y|^{\alpha} \int_0^t (t-s)^{-(1+\alpha)/2} \sum_z \big(\mathbf p_{t-s+1}^R(x,z)+\mathbf p_{t-s+1}^R(y,z)\big) \|Z_s(z)\|_p^2 ds \\ & \leq C|x-y|^{\alpha} \int_0^t (t-s)^{-(1+\alpha)/2} s^{-1/2} \sum_z \big(\mathbf p_{t-s+1}^R(x,z)+\mathbf p_{t-s+1}^R(y,z)\big) \big[ (\mathbf p_{s}^R * Z_0)(z)+1 \big] ds \\ &= C|x-y|^{\alpha} t^{-\alpha/2} \big( (\mathbf p_{t+1}^R * Z_0)(x) + (\mathbf p_{t+1}^R * Z_0)(y) +1 +1\big) \\ & \leq C (\epsilon |x-y|)^{\alpha} (\epsilon^2t)^{-(1+\alpha)/2} .
		\end{align*}
		Let us justify each of the inequalities above. In the second line we used (\ref{98}) to bound $\|Z_s(z)\|_p^2$ by $C(\epsilon^2s)^{-1/2}\big[ (\mathbf p_{s}^R * Z_0)(z)+1 \big]$. In the third line, we used the semigroup property to rewrite the sum over $z$, moreover we applied Corollaries \ref{4} and \ref{16} with $a_i=0$ in order to bound the $+1$ term next to $\mathbf p_s * Z_0$, and then we used the fact that $\int_0^t(t-s)^{-(1+\alpha)/2} s^{-1/2}ds = Ct^{-\alpha/2}$ which can be proved by making the substitution $s=tu$. In the final line, we used (\ref{38}) to bound $\mathbf p_{t+1}^R * Z_0$ by $C(\epsilon^2t)^{-1/2}$. Taking square roots in the above expression, we find that for $\alpha<1$, $$J_2 \leq C (\epsilon |x-y|)^{\alpha/2} (\epsilon^2t)^{-(1+\alpha)/4}$$ and now we note that $(\epsilon^2t)^{-(1+\alpha)/4} \leq T^{(1+\alpha)/4} (\epsilon^2t)^{-(1+\alpha)/2}$ which completes the proof of the second bound (\ref{47}).
		\\
		\\
		Now for the final bound (\ref{48}). We again use the semigroup property and Proposition \ref{74} to write $$Z_t(x) = (\mathbf p_{t-s}^R * Z_s)(x)+ \int_s^t \sum_{y} \mathbf p_{t-u}^R (x,y)dM_s(y).$$ Therefore $$\|Z_t(x)-Z_s(x)\|_p \leq \big\|(\mathbf p_{t-s}^R * Z_s)(x)-Z_s(x)\big\|_p + \bigg\| \int_s^t \sum_y \mathbf p_{t-u}^R(x,y) dM_s(y) \bigg\|_p.$$
		Let us call the terms on the RHS as $I_1,I_2$ respectively. Note that 
		\begin{align}\label{45}
		I_1 &= \bigg\| \sum_y \mathbf p_{t-s}^R (x,y) Z_s(y) -Z_s(x)\bigg\|_p \nonumber \\ & \leq \bigg\| \sum_y \mathbf p_{t-s}^R(x,y) \big( Z_s(y)-Z_s(x) \big) \bigg\|_p + \bigg| \sum_y \mathbf p_{t-s}^R(x,y)-1 \bigg| \cdot \|Z_s(x)\|_p \nonumber \\ & \leq \sum_y \mathbf p_{t-s}^R(x,y) (\epsilon^2s)^{-(1+\alpha)/2} (\epsilon|y-x|)^{\alpha} + C \epsilon (t-s)^{1/2} \cdot (\epsilon^2s)^{-1/2}.
		\end{align}
		In the final inequality we applied Proposition \ref{7} (or \ref{21}) together with (\ref{46}) and (\ref{47}). Just as in (\ref{44}) with $a=0$, we have that $$\sum_y \mathbf p_{t-s}^R (x,y)(\epsilon |x-y|)^{\alpha} \leq C \epsilon^{\alpha} ( 1 \vee |t-s|^{\alpha/2}).  $$ Also, since $s,t \leq \epsilon^{-2}T$, it follows that $\epsilon (t-s)^{1/2}\leq T^{(1-\alpha)/2} \epsilon^{\alpha} (t-s)^{\alpha/2}$ and similarly, $(\epsilon^2 s)^{-1/2} \leq T^{\alpha/2} (\epsilon^2s)^{-(1+\alpha)/2}$ for $\alpha <1$, hence by absorbing those powers of $T$ into the constant $C$, we get \begin{equation*} \epsilon (t-s)^{1/2} (\epsilon^2 s)^{-1/2} \leq C \epsilon^{\alpha}  (t-s)^{\alpha/2} (\epsilon^2s)^{-(1+\alpha)/2}. \end{equation*}
		Together with (\ref{45}), the preceding two expressions give the desired bound on $I_1$. In fact the bound holds for all $\alpha<1$, not just $\alpha<1/2$.
		\\
		\\
		Now we consider $I_2$. By the first bound (\ref{46}), note that $\|Z_u(y)\|_p \leq C(\epsilon^2u)^{-1/2} \leq C (\epsilon^2 s)^{-1/2}$ whenever $u \in [s,t]$. Furthermore, if $r \leq \epsilon^{-2}T$, then we know by Corollary \ref{4} (with $a_1=a_2=0$) that $\sum_y \mathbf p_r^R(x,y)$ is bounded by a constant independent of $x$ and $r$. Using these two facts, \begin{align*}
		I_2^2 &\leq C \epsilon \int_s^t (t-u)^{-1/2} \sum_y \mathbf p_{t-u+1}^R (x,y) \|Z_u(y)\|_p^2 du \\ & \leq C \epsilon (\epsilon^2 s)^{-1} \int_s^t (t-u)^{-1/2} \sum_y \mathbf p_{t-u+1}^R(x,y) du \\ & = C \epsilon (\epsilon^2s)^{-1} \int_s^t (t-u)^{-1/2}du \\ &= C(\epsilon^2s)^{-1} \cdot \epsilon (t-s)^{1/2}. 
		\end{align*}
		So taking square roots, we see that for $\alpha<1$ we have $$I_2 \leq C (\epsilon^2s)^{-1/2} \epsilon^{1/2} (t-s)^{1/4} \leq C(\epsilon^2 s)^{-\alpha/2} \epsilon^{\alpha/2} (t-s)^{\alpha/4}$$ where we again used the fact that $s, t \leq \epsilon^{-2}T$ in the final line. This is equivalent to (\ref{48}), thus completing the proof of the estimates.
	\end{proof}
	
	\begin{cor}\label{50}
		For any $0<\delta \leq \tau$, the laws of the rescaled processes $\{\mathcal Z^{\epsilon}\}_{\epsilon>0}$ are tight on the Skorokhod space $D([\delta, \tau],C(I))$. Moreover, any limit point lies in $C([\delta,\tau],C(I))$.
		\\
		\\
		For $T \in [\delta,\tau]$, let $\mathcal L(T): C([\delta,\tau], C(I)) \to C(I)$ denote the evaluation map at $T$. Let $\Bbb P$ be a limit point of the $\{\mathcal Z^{\epsilon}\}$. Then the process $(\mathcal L(T+\delta))_{T \in [0,\tau-\delta]}$ has the same distribution under $\Bbb P$ as the solution of the (multiplicative) SHE started from initial data whose distribution is that of $\mathcal L(\delta)$ under $\Bbb P$. 
	\end{cor}
	
	\begin{proof}
		To show tightness on $D([\delta,\tau], C(I))$ for $0<\delta \leq \tau$, we rewrite the estimates of Proposition \ref{36} in terms of the rescaled macroscopic processes $\mathcal Z^{\epsilon}$: for $\alpha<1/2$, $S,T\in [\delta,\tau]$, and $X,Y \in I$ we have $$\|\mathcal Z^{\epsilon}(T,X)\|_p \leq CT^{-1/2}\;\;\;\;\;\;\;\;\;\;\;\;\;\;$$ $$\|\mathcal Z^{\epsilon}(T,X)-\mathcal Z^{\epsilon}(T,Y)\|_p \leq C T^{-(1+\alpha)/2} |X-Y|^{\alpha}\;\;\;\;\;\;\;\;\;\;\;\;$$ $$\|\mathcal Z^{\epsilon}(T,X)-\mathcal Z^{\epsilon}(S,X) \|_p \leq CS^{-(1+\alpha)/2} (|T-S|^{\alpha/2} \vee \epsilon^{\alpha})$$ With the assumption that $S,T \geq \delta$, it follows that $T^{-1/2} \leq \delta^{-1/2}$, $T^{-(1+\alpha)/2}, S^{-(1+\alpha)/2} \leq \delta^{-(1+\alpha)/2}$. Therefore. The $\{\mathcal Z^{\epsilon}\}_{\epsilon\in (0,1]}$ are uniformly bounded, uniformly spatially Hölder, and uniformly temporally Hölder (except for jumps of order $\epsilon^{\alpha}$) with large probability. Now we apply the version of Arzela-Ascoli for Skorokhod spaces (together with Prohorov's theorem) to obtain tightness, see [Bil97, Chapter 3].
		\\
		\\
		The fact that any limit point lies in $C([\delta,\tau],C(I))$ is a straightforward consequence of Kolmogorov's continuity criterion.
		\\
		\\
		To prove the final statement, note that the above bounds imply that the sequence of initial data $\{\mathcal Z^{\epsilon}( \delta, \cdot)\}$ is ``near-equilibrium'' as defined in Assumption \ref{51}. Thus the results of Section 4 apply, and the proof is finished.
	\end{proof}
	
	\begin{defn}\label{127} Let $D((0,\infty),C(I))$ denote the set of all functions from $ (0,\infty) \to C(I)$ which are continuous on the right and have left limits. For $0< \delta \leq \tau$, let $\mathcal L_{\delta, \tau}: D((0,\infty) , C(I)) \to D([\delta,\tau],C(I))$ be given by $\varphi \mapsto \varphi|_{[\delta, \tau]}$. Henceforth, we will equip $D((0,\infty),C(I))$ with the smallest topology for which each of the maps $\mathcal L_{\delta,\tau}$ is continuous. We will also equip $C((0,\infty),C(I))$ with the corresponding topology.
		
	\end{defn}
	
	As a notational convention, $f_*\mu$ will denote the pushforward of a measure $\mu$ under a map $f$.
	\\
	\\
	A few remarks about the topology on $D((0,\infty),C(I))$: 
	
	\begin{itemize}
		\item Note that this topology is metrizable via $\sum_n 2^{-n} (1 \wedge \rho_n)$, where $\rho_n$ is a metric inducing the topology of $D([n^{-1},n], C(I))$.
		
		\item Note $\phi_n \stackrel{n \to \infty}{\longrightarrow} \phi$ in $D((0,\infty), C(I))$ if and only if $\phi_n|_{[\delta,\tau]} \stackrel{n \to \infty}{\longrightarrow} \phi|_{[\delta,\tau]}$ in $D([\delta,\tau],C(I))$, for every $0< \delta \leq \tau$. Similarly, for a sequence of probability measures $\Bbb P_n$ on $D((0,\infty), C(I))$, we have that $\Bbb P_n \to \Bbb P$ weakly iff $(\mathcal L_{\delta,\tau})_*\Bbb P_n \to (\mathcal L_{\delta,\tau})_*\Bbb P$ weakly for all $0<\delta \leq \tau$.
		
		\item This topology is an analogue of the topology of uniform convergence on compact sets, but for the Skorokhod Space.
	\end{itemize}
	
	\begin{lem}\label{62} Let $\Bbb Q^{\epsilon}$ denote the law of $\mathcal Z^{\epsilon}$ on $C((0,\infty),\Bbb R)$. Then there exists a measure $\Bbb Q$ on $C((0,\infty), C(I))$ which is a limit point of the $\Bbb Q^{\epsilon}$ on  $D((0,\infty),C(I))$.
		
	\end{lem}
	
	\begin{proof} The basic idea is to use the Kolmogorov Extension Theorem in conjunction with the previous results. For $0<\delta\leq \tau$, we will let $\Bbb P^{\epsilon}_{\delta,\tau}$ denote the law of $\mathcal Z^{\epsilon}$ on $D([\delta, \tau],C(I))$. For $0<\delta'\leq \delta \leq \tau \leq \tau'$, we define $\mathcal L^{\delta',\tau'}_{\delta,\tau}:D([\delta',\tau'],C(I)) \to D([\delta,\tau],C(I))$ be the map $\varphi \mapsto \varphi|_{[\delta,\tau]}$.
		\\
		\\
		We will use an inductive construction. Let $\Bbb P_1$ be a limit point of the $\{\Bbb P_{1,1}^{\epsilon}\}_{\epsilon\in(0,1]}$ on $D(\{1\},C(I))$. Then by Corollary \ref{50} we can find a sequence $\epsilon_j \downarrow 0$ such that $\Bbb P_{1,1}^{\epsilon_j} \to \Bbb P_1$ weakly. Suppose (for the inductive hypothesis) that for each $k \leq n-1$ we have constructed a probability measure $\Bbb P_k$ on $C([k^{-1},k], C(I))$ and a sequence $(\epsilon^k_j)_{j=1}^{\infty}$ with the following two properties:
		
		\begin{enumerate}
			\item For each $1 \leq k \leq n-1$, the sequence $\Bbb P_{k^{-1},k}^{\epsilon^k_j}$ converges weakly to $\Bbb P_k$ as $j \to \infty$.
			
			\item For each $2 \leq k \leq n-1$, $(\epsilon^k_j)_{j=1}^{\infty}$ is a subsequence of $(\epsilon^{k-1}_j)_{j=1}^{\infty}$
		\end{enumerate}
		
		By Corollary \ref{50}, the sequence $\{ \Bbb P^{\epsilon^{n-1}_j}_{n^{-1},n} \}_{j \geq 1}$ is tight, therefore we can find a probability measure $\Bbb P_{n}$ on $C([n^{-1},n],C(I))$ and subsequence $(\epsilon^{n}_j)_{j=1}^{\infty}$ of $(\epsilon_j^{n-1})_j$ such that $\Bbb P^{\epsilon^{n}_j}_{n^{-1},n} \to \Bbb P_{n}$ weakly as $j \to \infty$. Therefore the inductive hypothesis holds true for $k=n$.
		\\
		\\
		In the end, we obtain a sequence $\Bbb P_k$ of probability  measures on $C([k^{-1},k], C(I))$ which is consistent in the sense that $(\mathcal L^{(k+1)^{-1},k+1}_{k^{-1},k})_* \Bbb P_{k+1} = \Bbb P_k$, for every $k$. By the Kolmogorov Extension Theorem, there exists a unique probability measure $\Bbb Q$ on $C((0,\infty),C(I))$ such that $(\mathcal L_{k^{-1},k})_*\Bbb Q = \Bbb P_k $ for all $k$.
		\\
		\\
		To show that the measure $\Bbb Q$ is actually a limit point of the $\Bbb Q^{\epsilon}$, consider the sequences $(\epsilon^n_j)$ from before. Then the ``diagonal" sequence $\Bbb Q^{\epsilon^n_n}$ converges weakly to $\Bbb Q$ as $n \to \infty$. Indeed $\epsilon^n_n$ is an eventual subsequence of $\epsilon^k_n$ for every $k$, and therefore $(\mathcal L_{k^{-1},k})_* \Bbb Q^{\epsilon^n_n}=\Bbb P^{\epsilon_n^n}_{k^{-1},k}$ converges weakly (as $n \to \infty$) to $ (\mathcal L_{k^{-1},k})_* \Bbb Q=\Bbb P_k$ for every $k$.
	\end{proof}
	
	\begin{lem}\label{56} Let $\Bbb Q$ be a limit point of the $\{\mathcal Z^{\epsilon}\}$ on $C((0,\infty),C(I))$, as constructed in the previous lemma, and let $\mathcal L_T : C((0,\infty),C(I)) \to C(I)$ denote the evaluation map at $T$. For any $p \geq 1$, there exists a constant $C=C(p)$ such that for $X \in I$ and $T \leq 1$, $$\|\mathcal L_T(X)\|_p^2 \leq CT^{-1/2} P_T(X,0)$$$$\|\mathcal L_T(X) - P_T(X,0)\|_p^2 \leq C P_T(X,0)\;\;\;\;\;\;\;\;\;\;\;\;\;\;\;\;\;\;\;\;\;\;\;\;\;\;\;$$ where as usual, $\|F\|_p = \big(\int |F|^p d\Bbb Q\big)^{1/p} $.
		
	\end{lem}
	
	\begin{proof} Recall Equation (\ref{97}) and the remark underneath, which say that $$\|Z_t(x)\|_p^2 \leq C(\epsilon^2t)^{-1/2}(\mathbf p_t^R * Z_0)(x)+C\sum_{k=0}^{\infty} \frac{C^k \epsilon^k t^{k/2}}{(k/2)!}(\mathbf p_{t+k}^R * Z_0)(x).$$
		Rewriting the preceding expression in terms of the macroscopic variables, we have 
		\begin{equation}\label{89}
		\|\mathcal Z_T^{\epsilon}(X) \|_p^2 \leq CT^{-1/2}(P_T^{\epsilon} *_{\epsilon}\mathcal Z_0^{\epsilon})(X)+\sum_{k=0}^{\infty}\frac{C^kT^{k/2}}{(k/2)!} (P_{T+\epsilon^2k}^{\epsilon} *_{\epsilon}\mathcal Z_0^{\epsilon})(X) 
		\end{equation} where $(f *_{\epsilon}g) (X):=\epsilon \sum_{y \in \epsilon \Lambda} f(X,Y)g(Y)$, for any functions $f:(\epsilon \Lambda)^2 \to \Bbb R$ and $g: \epsilon \Lambda \to \Bbb R $.
		\\
		\\
		We define the quantity $F(T,X, \epsilon)$ to be the RHS of (\ref{89}). We are going to prove that $$\limsup_{\epsilon \to 0} F(T,X,\epsilon) \leq CT^{-1/2}P_T(X,0)$$ where the $P_T$ on the RHS is the continuum heat kernel. As a first step, we claim that for any $k>0$, \begin{equation}\label{90}
		\lim_{\epsilon \to 0} (P_{T+\epsilon^2k}^{\epsilon} *_{\epsilon} \mathcal Z_0^{\epsilon})(X) = P_T(X,0)
		\end{equation} Indeed, we have that
		\begin{align}\label{54} \varrho
		\epsilon^{1/2} \sum_{Y \in \epsilon \Lambda} P_{T+\epsilon^2k}^{\epsilon}(X,Y)e^{-\epsilon^{-1/2}Y} &= \varrho\epsilon^{-1/2} \int_I P_{T+\epsilon^2k}^{\epsilon}(X, \epsilon \lfloor \epsilon^{-1}Y \rfloor) e^{-\epsilon^{1/2}\lfloor \epsilon^{-1}Y \rfloor} dY \nonumber\\ &=\varrho\int_I P_{T+\epsilon^2k}^{\epsilon}(X, \epsilon \lfloor \epsilon^{-1/2}Z \rfloor)e^{-\epsilon^{1/2} \lfloor \epsilon^{-1/2}Z \rfloor}dZ \nonumber\\& \stackrel{\epsilon \to 0}{\longrightarrow} \varrho\int_I P_T (X, 0) e^{-Z}dZ \nonumber\\ &= P_T(X,0).
		\end{align}
		In the first equality we used that $\epsilon \sum_{Y \in \epsilon \Lambda}f(Y) = \int_I f(\epsilon \lfloor \epsilon^{-1}Y \rfloor)dY$, for any function $f$. In the second equality we made the substitution $Z=\epsilon^{-1/2}Y$. In the third line we used uniform convergence (see Theorem \ref{controb}) of $P_{T+\epsilon^2k}^{\epsilon}(X,\cdot)$ to $P_{T}(X, \cdot)$ together with the fact that $\epsilon \lfloor \epsilon^{-1/2}Z \rfloor \to 0$ and $\epsilon^{1/2} \lfloor \epsilon^{-1/2}Z \rfloor \to Z$. In the final line we used the fact that $\varrho = \big(\int_I e^{-Z}dZ \big)^{-1}$, as defined earlier in this section. This proves (\ref{90}).
		\\
		\\
		In order to finish showing that $\lim_{\epsilon \to 0} F(T,X,\epsilon) \leq CT^{-1/2}P_T(0,X)$, we will take the limit as $\epsilon \to 0$ on the RHS of Equation (\ref{89}), and then pass the limit through the infinite sum and apply (\ref{90}). However, we need to justify interchanging the $\lim_{\epsilon \to 0}$ with the infinite sum $\sum_{k=0}^{\infty}$. To justify this interchange, we recall Equation (\ref{lt3}) which says (after passing to macroscopic variables) that \begin{align*}\sup_{\epsilon \in (0,1]} (P_{T}^{\epsilon} *_{\epsilon} \mathcal Z_0^{\epsilon})(X) \leq C(T^{-1/2}+1)e^{KT} \end{align*} where $C,K$ do \textbf{not} depend on the terminal time $\tau$. Note that $T^{-1/2}+1 \leq CT^{-1/2}e^{KT}$, hence the RHS of the last expression may be further bounded by $C'T^{-1/2}e^{2KT}$. Using this bound shows that \begin{align}\label{95}\sum_{k=0}^{\infty}\frac{C^kT^{k/2}}{(k/2)!} \sup_{\epsilon \in (0,1]} (P_{T+\epsilon^2k}^{\epsilon} *_{\epsilon}\mathcal Z_0^{\epsilon})(X) &\leq CT^{-1/2}e^{2KT}\sum_{k=0}^{\infty}\frac{C^kT^{k/2}}{(k/2)!} e^{2Kk} \nonumber \\ &\leq CT^{-1/2}e^{K'T}.\end{align} By (\ref{95}) and the dominated convergence theorem, we can interchange the limit as $\epsilon \to 0$ with the infinite sum in (\ref{89}), as discussed before. This completes the proof that \begin{equation}\label{100}\limsup_{\epsilon \to 0} F(T,X, \epsilon)\leq CT^{-1/2}P_T(X,0).\end{equation}
		In order to prove the second bound, we note from Lemma \ref{73} and (\ref{76}) that for $t \geq 1$, \begin{align*}\big\|Z_t(x)-\mathbf p_t^R * Z_0(x) \big\|_p^2 &\leq C \epsilon\int_0^t (t-s)^{-1/2} \sum_y \mathbf p_{t-s+1}^R(x,y)\|Z_s(y)\|_p^2ds .\end{align*}
		In terms of macroscopic variables, this says that for $T\geq \epsilon^2$ \begin{align}\label{101}\big\| \mathcal Z_T^{\epsilon}(X) - (P_T^{\epsilon} *_{\epsilon} \mathcal Z_0^{\epsilon})(X) \big\|_p^2 &\leq C\int_0^T (T-S)^{-1/2} \big(P_{T-S+\epsilon^2}^{\epsilon} *_{\epsilon} \|\mathcal Z_S\|_p^2 \big)dS \nonumber\\ &\leq C \int_0^T (T-S)^{-1/2} \big(P_{T-S + \epsilon^2}^{\epsilon} *_{\epsilon} F(S, \cdot, \epsilon) \big)(X)dS\end{align} where we used (\ref{89}) in the last line. Now using (\ref{95}), it is easily shown that $$S\;\mapsto\; (T-S)^{-1/2} \cdot \sup_{\epsilon \in (0,1]}\big(P_{T-S + \epsilon^2}^{\epsilon} *_{\epsilon} F(S, \cdot, \epsilon) \big)(X)$$ is a function which is integrable over $[0,T]$. Moreover, by (\ref{100}) we have that $$\limsup_{\epsilon \to 0} \big(P_{T-S+\epsilon^2}^{\epsilon} *_{\epsilon} F(S,\cdot, \epsilon)\big)(X) \leq S^{-1/2} C\big(P_{T-S} * P_S(\cdot,0)\big)(X) = CS^{-1/2}P_T(X,0).$$ To finish the proof, let $\epsilon \to 0$ in (\ref{101}) and apply the dominated convergence theorem to interchange the $\lim$ and the $\int_0^T$, and finally note that $\int_0^T (T-S)^{-1/2}S^{-1/2}dS$ is a constant not depending on $T$.
	\end{proof}
	
	\begin{lem}\label{63} Let $\Bbb Q$ and $\mathcal L_T$ be as in the previous lemma. Then we can enlarge the probability space $\big(C((0,\infty),C(I)), \;\mathcal{B} orel, \;\Bbb Q)$ so as to admit a space-time white noise $W$ such that $\Bbb Q$-almost surely, for every $0<S<T$: $$\mathcal L_T = P_{T-S}*\mathcal L_S + \int_S^T P_{T-U}*(\mathcal L_U dW_U) .$$
		
	\end{lem} 
	
	\begin{proof}Let us fix $S \geq 0$ for now. Define $$Y^S_T(\varphi) = (\mathcal L_{T+S}, \varphi) - (\mathcal L_S, \varphi) - \frac{1}{2} \int_S^{S+T} (\mathcal L_U, \varphi'')dU.$$ By Corollary \ref{50}, we know that $(\mathcal L_{S+T})_{T \geq 0}$ has the same distribution (under $\Bbb Q)$ as the solution of the SHE started from $\mathcal L(S)_*\Bbb Q$.
		Therefore (see (\ref{49})), the law of $\mathcal (\mathcal L_{S+T})_{T \geq 0}$ solves the martingale problem for the SHE, as defined in \ref{57}. Just as in the proof of Proposition \ref{58}, this shows that $Y^S$ is an orthogonal martingale measure. So we let $\overline{W}$ be an independent white noise on $[0,\infty) \times I$ (\textit{not} depending on $S$), then define the following space-time process for $T \geq 0$: $$W^S_T(\varphi):= \int_0^T \int_I \mathcal L_{S+U}(X)^{-1} \varphi(X) 1_{\{L_{S+U}(X) \neq 0\}} Y^S(dX\; dU) + \int_S^{S+T} \int_I 1_{\{\mathcal L_U(X)=0\}}\varphi(X) \overline W(dX\;dU).$$Just as in the proof of proposition \ref{58}, $W^S$ is a cylindrial Wiener process such that $\Bbb Q$-almost surely, for any $T \geq 0$ we have that \begin{equation}\label{60}\mathcal L_{S+T} = P_T * \mathcal L_S + \int_0^{T} P_{T-U} *( \mathcal L_U dW^S_U).\end{equation} This construction actually shows that the white noises $W^S$ are consistent, in the sense that for any $S_1<S_2$ we have that $W^{S_2}_T = W^{S_1}_{S_2-S_1+T}-W^{S_1}_{S_2-S_1}$ (because the $Y^S$ satisfy the same relation).
		\\
		\\
		Finally, if $T >0$, then we define $$W_T := W_{(T-1)\vee 0}^1 + \sum_{k=1}^{\infty} W^{2^{-k}}_{2^{-k} \wedge (T-2^{-k}) \vee 0}. $$ Using the property that that $W^{S_2}_T = W^{S_1}_{S_2-S_1+T}-W^{S_1}_{S_2-S_1}$ for $S_1<S_2$, it follows that the terms appearing in the infinite sum are independent. Hence the convergence of the infinite series may be checked by the martingale convergence theorem. The noise $W$ may be equivalently described as the a.s. and $L^2$ limit of the $W^S$ as $S \to 0$. One may then directly check that $W$ is a cylindrical Wiener process with the property that $W^S_T = W_{S+T}-W_S$. This, together with (\ref{60}), shows that the desired relation holds.
	\end{proof}
	
	\begin{thm}[Main Result of this Section]\label{61} The rescaled processes $\mathcal Z^{\epsilon}$ converge weakly in $D((0,\infty), C(I))$ to the solution of the SHE with $\delta_0$ initial data (as constructed in Proposition \ref{71}).
		
	\end{thm}
	
	\begin{proof} Let $\Bbb Q$ be \textit{any} limit point of the laws of the $\{\mathcal Z^{\epsilon}\}$ on $C((0,\infty),C(I))$. By Lemma \ref{62}, we know that at least one such $\Bbb Q$ exists.
		\\
		\\
		As usual, we let $\mathcal L_T: C((0,\infty),C(I)) \to C(I)$ denote the canonical $T$-coordinate, and we will let $\|F\|_p:=\big(\int |F|^pd\Bbb Q\big)^{1/p}$. By Lemma \ref{63}, after possibly extending the probability space, there exists a white noise $W$ such that $\Bbb Q$-a.s., for any $0<S<T$ we have that \begin{equation*}\mathcal L_T = P_{T-S}*\mathcal L_S + \int_S^T P_{T-U}*(\mathcal L_U dW_U) .\end{equation*}
		Therefore, if we can prove that $\int_0^T P_{T-U} * (\mathcal L_U dW_U)$ exists and is in $L^2(\Bbb Q)$ for any $T>0$, then we would have that \begin{align}\label{39}
		&\;\;\;\;\;\bigg\| \mathcal L_T(X) - P_T(X,0)- \int_0^T P_{T-U} * (\mathcal L_U dW_U) \big|_X \bigg\|_2 \nonumber \\ &\leq \big\| P_{T-S}* \mathcal L_S(X) - P_T(X,0) \big\|_2 +\bigg\| \int_0^S P_{T-U} * (\mathcal L_U dW_U) \big|_X \bigg\|_2 .
		\end{align} Notice that the top expression does not depend on $S$, whereas the bottom one does. So if we can prove that (for any fixed $T$) both of the terms in (\ref{39}) tend to $0$ as $S \to 0$, then it will immediately follow that $\Bbb Q$ is the distribution of the solution of the SHE started from $\delta_0$. Thus the remainder of the proof will be dedicated to this purpose.
		\\
		\\
		For the remainder of this proof, we will fix some time $T >0$, and we will always consider $S<1 \wedge T$.
		\\
		\\
		Let us start with the first term in (\ref{39}). By the semigroup property and Minkowski's inequality, we see that \begin{align*}\big\| P_{T-S}* \mathcal L_S(X) - P_T(X,0) \big\|_p &= \bigg\| \int_I P_{T-S}(X,Y) \big[ \mathcal L_S(Y) - P_S(Y,0) \big] dY \bigg\|_p \\ &\leq \int_I P_{T-S}(X,Y) \big\| \mathcal L_S(Y) - P_S(Y,0) \big\|_p dY \\ & \leq C\int_I P_{T-S}(X,Y) \cdot P_S(Y,0)^{1/2} dY\end{align*}
		where we used Lemma \ref{56} in the final line. By applying Proposition \ref{27} with $b=0$, we see that $P_{T-S}(X,Y) \leq C(T-S)^{-1/2}$. Similarly, applying Proposition \ref{27} with $b=2$, we see that $P_S(Y,0) \leq CS^{-1/2} e^{-2Y/\sqrt{S}}$. Therefore, continuing the above chain of inequalities, we find that \begin{align*}
		\int_I P_{T-S}(X,Y) \cdot P_S(Y,0)^{1/2} dY & \leq C(T-S)^{-1/2}S^{-1/4} \int_I e^{-Y/\sqrt{S}}dY \\ & \leq C(T-S)^{-1/2} S^{-1/4} \int_0^{\infty} e^{-Z}(S^{1/2}dZ) \\ &= C (T-S)^{-1/2} S^{1/4}
		\end{align*}
		where we made the substitution $Z=S^{-1/2}Y$ in the second line. Since $T$ was assumed to be fixed, and since $(T-S)^{-1/2} \leq \sqrt{2}T^{-1/2}$ whenever $S<T/2$, this computation shows that \begin{equation*}\limsup_{S\to 0} \big\| P_{T-S}* \mathcal L_S(X) - P_T(X,0) \big\|_p\; \lesssim \;\lim_{S \to 0} S^{1/4} = 0.\end{equation*} This shows that the first term appearing on the RHS of (\ref{39}) approaches $0$ as $S \to 0$.
		\\
		\\
		Next, we are going to consider the second term appearing on the RHS of (\ref{39}). As stated above, we first need to prove existence of the stochastic integral $\int_0^T P_{T-U} * (\mathcal L_U dW_U)$. This amounts to showing that $$\int_0^T \int_I P_{T-U}(X,Y)^2 \Bbb E[\mathcal L_U(Y)^2] dYdU<\infty $$ where the expectation is with respect to $\Bbb Q$. Using Lemma (\ref{56}), it holds that $\Bbb E[ \mathcal L_U(Y)^2 ] \leq C U^{-1/2} P_U(Y,0) $. Using the first bound in Proposition \ref{27} with $b=0$, we also have that $P_{T-U}(X,Y)^2 \leq C(T-U)^{-1/2} P_{T-U}(X,Y)$. Therefore, \begin{align*}
		\int_0^T \int_I P_{T-U}(X,Y)^2 \Bbb E[\mathcal L_U(Y)^2] dYdU & \leq C \int_0^T (T-U)^{-1/2} U^{-1/2} \bigg[\int_I P_{T-U}(X,Y) P_U(Y,0)dY\bigg]dU \\ &\leq C \int_0^T (T-U)^{-1/2}U^{-1/2} dU \cdot P_T(X,0) \\ &= CP_T(X,0)<\infty.
		\end{align*}
		In the second line, we used the semigroup property of the $P_T$, and in the third line we used the fact that $\int_0^T (T-U)^{-1/2}U^{-1/2}dU$ is a constant not depending on $T$, which can be proved by making the substitution $V=U/T$. This proves that the stochastic integral $\int_0^T P_{T-U} * (\mathcal L_U dW_U)$ exists. Now we just need to show that the second term in (\ref{39}) tends to $0$ as $S \to 0$. Applying the Itô isometry and then using the same exact arguments as above, we see that \begin{align*}
		\bigg\| \int_0^S P_{T-U} * (\mathcal L_U dW_U) \big|_X \bigg\|_2^2 &= \int_0^S \int_I P_{T-U}(X,Y)^2 \Bbb E[\mathcal L_U(Y)^2] dYdU \\ &\leq C \int_0^S (T-U)^{-1/2}U^{-1/2} dU \cdot P_T(X,0).
		\end{align*}
		Clearly $P_T(X,0)$ does not depend on $S$, and as long as $U<S<T/2$ we have $(T-U)^{-1/2} \leq (T-S)^{-1/2} \leq (T/2)^{-1/2}$, which no longer depends on $U$ or $S$. Therefore the last integral is $O(S^{1/2})$ as $S \to 0$, which completes the proof.
	\end{proof}


\begin{thebibliography}{alpha}
		
		\bibitem[ACQ11]{ACQ11}G. Amir, I. Corwin, J. Quastel. \textit{Probability distribution of the free energy of the continuum directed random polymer in 1+1 dimensions} Comm. Pure and Applied Math. 64, 466. 2011.
		
		\bibitem[BBCW17]{BBCW17}G. Barraquand, A. Borodin, I. Corwin, M. Wheeler. \textit{Stochastic six-vertex model in a half-quadrant and half-line open ASEP}. arXiv preprint arXiv:1704.04309.
		
		\bibitem[BG97]{BG97}L. Bertini, G. Giacomin. \textit{Stochastic Burgers and KPZ equations from particle systems.} Comm. Math. Phys., 183 (3), 571-607.
		
		\bibitem[Bil97]{Bil97}P. Billingsley. \textit{Convergence of Probability Measures.} Wiley. 1997.
		
		\bibitem[Cor12]{Cor12}I. Corwin. \textit{The KPZ Equation and Universality Class.} Random matrices: Theory and Appl. 1 (01), 1130001. 2012.
		
		\bibitem[CS16]{CS16}I. Corwin, H. Shen. \textit{Open ASEP in the weakly asymmetric regime}. arXiv preprint arXiv:1610.04931. 2016.
		
		\bibitem[CST16]{CST16}I. Corwin, H. Shen, and L.C. Tsai. \textit{ASEP (q, j) converges to the KPZ equation} arXiv preprint arXiv:1602.01908. 2016.
		
		\bibitem[CT15]{CT15}I. Corwin, L.C. Tsai. \textit{KPZ equation limit of higher-spin exclusion processes}. Ann. Prob., 45 (3), 1771-1798.
		
		\bibitem[DPZ92]{DPZ92}G. Da Prato, J. Zabczyk. \textit{Stochastic equations in infinite dimensions.} Journals Amer. Math. Soc. To Appear.
		
		\bibitem[GJ14]{GJ14}P. Goncalves, M. Jara. \textit{Nonlinear fluctuations of weakly asymmetric interacting particle systems}. Arch. Ration. Mech. Anal. 212, no. 2, 597-644. 2014.
		
		\bibitem[GP16]{GP16}M. Gubinelli, N. Perkowski. \textit{Energy solutions of KPZ are unique.} arXiv preprint arXiv:1508.07764. 2015.
		
		\bibitem[GPS17]{GPS17}P. Goncalves, N. Perkowski, M. Simon. \textit{Derivation of the Stochastic Burgers Equation with Dirichlet Boundary Conditions from WASEP.} arXiv preprint arXiv:1710.11011. 2017.
		
		\bibitem[DT16]{DT16}A. Dembo. L.C. Tsai. \textit{Weakly asymmetric non-simple exclusion process and the Kardar–Parisi–Zhang equation}. Comm. Math. Phys. 341 (1), 219-261.
		
		\bibitem[Gar88]{Gar86}J. Gartner.
		\textit{Convergence towards Burgers' equation and propagation of chaos for weakly asymmetric exclusion processes.}  Stochastic Process. Appl. 27 (1988), no. 2, 233–260.
		
		\bibitem[Hai09]{Hai09}
		M. Hairer,
		\textit{An Introduction to SPDEs},
		arXiv eprint arXiv:0907.4178.
		July 2009.
		
		\bibitem[IM63]{IM63}K. Itô, H. McKean. \textit{Brownian motions on a half-line.} Illinois Journal of Math. 7, no. 2, 181-231. 1963. 
		
		\bibitem[Lig75]{Lig75}T. Liggett. \textit{Ergodic theorems for the asymmetric simple exclusion process.} Trans. Amer. Math. Soc. 213, (1975), 237-261.
		
		\bibitem[MQR17]{MQR17}K. Matetski, J. Quastel, D. Remenik.\textit{ The KPZ fixed point}. arXiv preprint arXiv:1701.00018.
		
		\bibitem[Mue91]{Mue91}C. Mueller. \textit{On the support of solutions to the heat equation with noise}. Stochastics, 37, 4, 225-246, 1991.
		
		\bibitem[NMW82]{NMW82}K. Naqvi, K. Mork, S. Waldenstrom. \textit{Symmetric random walk on a regular lattice with an elastic barrier: diffusion equation and boundary condition}. Chem. Phys. Letters 92, no. 2, 160-164. 1982.
		
		\bibitem[Wal86]{Wal86}J. Walsh. \textit{An Introduction to Stochastic Partial Differential Equations}.
		Lecture Notes in Math., Vol. 1180, pp. 265-439. 2986.
		
	\end{thebibliography}
\end{document}